\documentclass{article}
\usepackage{amssymb,amsfonts,amsmath,amsthm,amsopn,amstext,amscd,latexsym,xy,stmaryrd}
\usepackage[hidelinks]{hyperref}
\theoremstyle{plain}
\usepackage{verbatim}
\usepackage{mathrsfs}

\input xy

\usepackage{tikz}
\usetikzlibrary{decorations.markings}
\usetikzlibrary{arrows}

\xyoption{all}
\newtheorem{theorem}{Theorem}[section]
\newtheorem{lemma}[theorem]{Lemma}
\newtheorem{proposition}[theorem]{Proposition}
\newtheorem{corollary}[theorem]{Corollary}
\newtheorem{definition}[theorem]{Definition}

\theoremstyle{remark}
\newtheorem{remark}[theorem]{Remark}
\numberwithin{equation}{section}
\numberwithin{paragraph}{section}

\DeclareMathOperator{\Hom}{Hom}

\DeclareMathOperator{\Ad}{Ad}

\DeclareMathOperator{\rank}{rank}

\DeclareMathOperator{\ad}{ad}
\DeclareMathOperator{\ord}{ord}

\DeclareMathOperator{\ind}{ind}
\DeclareMathOperator{\Fitt}{Fitt}

\DeclareMathOperator{\val}{val}

\DeclareMathOperator{\tr}{tr}
\newcommand{\St}{\mathrm{St}}
\newcommand{\uni}{\mathrm{uni}}

\DeclareMathOperator{\End}{End}

\DeclareMathOperator{\Frob}{Frob}
\DeclareMathOperator{\Ann}{Ann}

\DeclareMathOperator{\Spec}{Spec}

\DeclareMathOperator{\coker}{coker}
\newcommand{\st}{\St}
\DeclareMathOperator{\fl}{fl}
\DeclareMathOperator{\chibar}{\overline \chi}
\newcommand{\loc}{{\rm loc}}

\renewcommand{\min}{\operatorname{min}}
\renewcommand{\mod}{\operatorname{mod}}

\newcommand{\card}[1]{|#1|}
\newcommand{\cA}{{\mathcal A}}
\newcommand{\cB}{{\mathcal B}}

\newcommand{\cD}{{\mathcal D}}
\newcommand{\cE}{{\mathcal E}}

\newcommand{\cI}{{\mathcal I}}

\newcommand{\cK}{{\mathcal K}}

\newcommand{\cN}{{\mathcal N}}
\newcommand{\cO}{{\mathcal O}}

\newcommand{\cR}{{\mathcal R}}
\newcommand{\cS}{{\mathcal S}}
\newcommand{\cT}{{\mathcal T}}

\newcommand{\ffrm}{{\mathfrak m}}

\newcommand{\frp}{{\mathfrak p}}
\newcommand{\frq}{{\mathfrak q}}

\newcommand{\GL}{\mathrm{GL}}

\newcommand{\CNLO}{{\mathrm{CNL}_\cO}}
\newcommand{\CNL}{{\mathrm{CNL}}}

\newcommand{\es}{\varnothing}

\newcommand{\cOmega}{{\widehat{\Omega}}}
\newcommand{\cotimes}{\widehat{\otimes}}
\newcommand{\invlim}{\varprojlim}
\newcommand{\ds}{\displaystyle}
\newcommand{\half}{\mathcal{H}}
\newcommand{\Rt}{\widetilde{R}}

\newcommand{\rhobar}{\overline{\rho}}

\newcommand{\A}{\mathbf A}
\newcommand{\Q}{\mathbf Q}
\newcommand{\Z}{\mathbf Z}
\newcommand{\T}{\mathbf T}
\newcommand{\F}{\mathbf F}
\newcommand{\ra}{\rightarrow}
\newcommand{\into}{\hookrightarrow}
\newcommand{\onto}{\twoheadrightarrow}

\newcommand{\cng}[1]{\Psi_{#1}}
\newcommand{\cngid}[1]{\eta_{#1}}
\newcommand{\coh}{{\rm coh}}

\newcommand{\df}{{\mathrm d}}
\def\Mat#1#2#3#4{{{ 
\left( \begin{array}{cc}  #1 & #2 \\ #3 &  #4   \end{array}  \right)}}}

\newcommand{\unr}{{\mathrm{unr}}}

\newcommand\blfootnote[1]{%
  \begingroup
  \renewcommand\thefootnote{}\footnote{#1}%
  \addtocounter{footnote}{-1}%
  \endgroup
}

\title{Wiles defect for Hecke algebras  that are not complete intersections}

\author{Gebhard B\"ockle, Chandrashekhar B. Khare, Jeffrey Manning}

\begin{document}
\maketitle

\begin{abstract}
In his work on modularity theorems, Wiles proved a numerical criterion for a map of rings $R\to T$ to be an isomorphism of complete intersections. He used this to show that certain deformation rings and Hecke algebras associated to a mod $p$ Galois representation at non-minimal level are isomorphic and complete intersections, provided the same is true at minimal level.

In this paper we study Hecke algebras acting on cohomology of Shimura curves arising from maximal orders in indefinite quaternion algebras over the rationals localized at a semistable irreducible mod $p$ Galois representation $\rhobar$.  If $\rhobar$ is scalar at some primes dividing the discriminant of the quaternion algebra, then the Hecke algebra is still isomorphic to the deformation ring, but  is not a complete intersection, or even Gorenstein, so the Wiles numerical criterion cannot apply.

We consider a  weight 2 newform  $f$ which contributes to the cohomology of the Shimura curve and gives rise to an augmentation  $\lambda_f$ of the Hecke algebra. We quantify the failure of the Wiles numerical criterion at $\lambda_f$ by computing the associated  {\it Wiles defect} purely in terms of the local behavior at primes dividing the discriminant of the global Galois representation $\rho_f$ which $f$ gives rise to by the Eichler--Shimura construction.  One of the main tools used in the proof is Taylor--Wiles--Kisin patching.
\end{abstract}



\section{Introduction}

Let $p>2$ be prime. Let $E$ be a finite extension of $\Q_p$ with valuation ring $\cO$, residue field $k$ and a uniformizer $\varpi$.  All our rings will be $\cO$-algebras.   We say that an $\cO$-algebra $R$
that is a complete Noetherian local ring,
and is finite flat as an $\cO$-module, with residue field $k$,
is a complete intersection
if $R \simeq \cO[[T_1,\cdots,T_n]]/(f_1,\cdots,f_n)$ for some~$n$.%
\blfootnote{2020 Mathematics Subject Classification:~11F80}
\blfootnote{Keywords: Wiles defect, modularity theorem, Cohen-Macaulay}

 Wiles' numerical isomorphism criterion played an important role in his work \cite{Wiles} on modularity of elliptic curves.  
 Wiles  characterised rings that were finite flat over $\mathcal O$ which were complete intersections, and used the criterion to deduce modularity lifting theorems in non-minimal level  from those in minimal level.

  Wiles and Lenstra proved the following result (\cite{Wiles}, \cite{Lenstra}).
  
\begin{theorem}\label{numerical criterion}
  Let $R$ and $T$ be complete Noetherian local $\cO$-algebras with residue
  field $k$, $T$ finite flat as an $\cO$-module, 
  and $\phi:R \rightarrow T$ a surjective map of local $\cO$-algebras. Let
  $\lambda:T \rightarrow \cO$ be a homomorphism of local $\cO$-algebras, 
  and set $\Phi_R={\rm ker}(\lambda\phi)/{\rm ker}(\lambda\phi)^2$ and 
  $\eta_T=\lambda(\Ann_T({\rm ker}(\lambda)))$.
  Then $\card{\cO/\eta_T} \leq \card{\Phi_R}$, where in the case
  $\card{\cO/\eta_T}$ is infinite this is interpreted to mean
  that $\Phi_R$ is infinite too. Assume that $\eta_T$ is not
  zero. Then the following are equivalent:

\begin{itemize}

\item The equality $\card{\Phi_R}=\card{\cO/\eta_T}$ is satisfied.

\item The rings $R$ and $T$ are complete intersections, and $\phi$ is
      an isomorphism.

\end{itemize}

\end{theorem}

This paper arose out of trying to answer the following question which arises in a situation studied in \cite{RibetMult2}. Consider the elliptic curve $X_1(11)$ defined by $y^2z+yz^2=x^3-x^2z$, and the mod $p=3$ representation $\rhobar$ arising from it. Its Serre invariants are $k(\rhobar)=2, N(\rhobar)=11, \epsilon(\rhobar)=1$ and its image is $\GL_2(\F_3)$.

Let $q\ne 3,11$ be a prime at which $\tr(\rhobar({\rm Frob}_{q}))\equiv \pm(q+1)\pmod p$. By  Ribet's level-raising results \cite{RibCong} we know that there is a newform $f \in S_2(\Gamma_0(11q))$ which gives rise to $\rhobar$. More precisely we fix an embedding $\iota: K_f \hookrightarrow \overline \Q_p$, with $K_f=\Q(a_n(f))$ the number field generated by the Fourier coefficients $f=\Sigma_na_n(f) q^n$ of $f$, and assume  $\cO$ is the ring of integers  of the completion of  $\iota(K_f)$.  Then by $f$ gives rise to $\rhobar$, we mean the Galois representation $\rho_{f,\iota}:G_\Q \ra \GL_2(\cO)$ arising from $f,\iota$ by the Eichler-Shimura  construction is residually isomorphic to $\rhobar$.

Consider the Shimura curve $X$ over $\Q$  arising from  the maximal order of the indefinite quaternion algebra that is ramified at $11$ and $q$. Let $\T$ be the $\cO$-algebra
generated  by Hecke operators $T_r$ for $r$ prime, acting on $H^1(X,\cO)$. Then $f$ gives rise (via Jacquet-Langlands)  to a homomorphism $\lambda_f:  \T \ra \cO$, with kernel $\frp_f$ contained in a maximal ideal $\ffrm$. Write $Q = \{11,q\}$, let $\T^Q:=\T_{\ffrm}$ and let $R^Q$ be the appropriate Galois deformation ring of $\rhobar$ parameterizing lifts of $\rhobar$ of fixed determinant which are Steinberg at $11$ and $q$ (which we define explicitly in section \ref{ssec:global def} as $R^{\st}$). Then there is a  Galois representation $\rho_{\ffrm}:G_\Q \ra \GL_2(\T^Q)$ with residual representation $\rhobar$ inducing a surjective morphism $\phi^Q:R^Q\onto \T^Q$ of $\cO$-algebras.

This is now the setting of Theorem \ref{numerical criterion}, applied to these rings and the morphism $\lambda_f:\T^Q\to\cO$. Thus one can consider $\Phi_f^Q:=\Phi_{R^Q}$ and $\eta_f^Q := \eta_{\T^Q}$, referred to as the \emph{cotangent space} and \emph{congruence ideal} of $f$, and Theorem \ref{numerical criterion} implies $|\Phi_f^Q|\ge |\cO/\eta_f^Q|$. We also refer to $\cO/\eta_f^Q$ as the \emph{congruence module} of $f$.  

In the case when $\rhobar({\rm Frob}_{q})\ne\pm {\rm Id}$, automorphy lifting techniques, in the form of the Taylor--Wiles--Kisin patching method, imply that $R^Q$ and $\T^Q$ are complete intersections and $\phi^Q$ is an isomorphism. In this case, Theorem \ref{numerical criterion} gives the equality $|\Phi_f^Q| = |\cO/\eta_f^Q|$. 

On the other hand, we say that $q$ is a \emph{trivial prime} for $\rhobar$ if $\rhobar({\rm Frob}_{q})=\pm {\rm Id}$, which implies that $q\equiv 1\pmod{p}$, as $\det\rhobar$ is the cyclotomic character (as shown in \cite{RibetMult2}, the first such prime for this specific $\rhobar$ is $q=193$, at which $\rhobar({\rm Frob}_{193})=- {\rm Id}$). For such a $q$, patching techniques no longer imply that $R^Q$ and $\T^Q$ are complete intersections. In fact, computations of Shotton in \cite{Shotton}, combined with patching, still imply that $\phi^Q$ is an isomorphism, but also show that $\T^Q$ is \emph{not} a complete intersection (nor is it Gorenstein). Theorem \ref{numerical criterion} then implies that $|\Phi_f^Q| > |\cO/\eta_f^Q|$.

Trivial primes have played an important role   in automorphy lifting (for instance, Taylor's method of Ihara avoidance \cite{TaylorIharaAvoidance}), and  lifting Galois representations (\cite{HR}, \cite{FKP}). Trivial primes are very useful, because of their versatility in  allowing lifts of different types or ``killing dual  Selmer classes'' as in 
 loc. cit. Trivial primes  behave rather differently from  the other primes, and our paper explores this difference in the context of  the Wiles numerical criterion. 

Thus our primary goal in this paper is to determine the ratio $|\Phi^Q_f|/|\cO/\eta^Q_f|$, or equivalently the $\cO$-module $ {\eta^Q_f}/\Fitt_\cO(\Phi^Q_f)$, in this situation. This is called the {\em Wiles defect} of $\T^Q$  relative to $\lambda_f$ as defined in \cite{TiUr}.  

Specializing our main theorem Theorem \ref{mc}  to this particular case gives:

\begin{theorem}\label{thm:example}
 Let $n_q$ be the maximal integer such that $\rho_{f,\iota}$ locally at $q$ is $\pm{\rm Id}$ mod $\varpi^{n_q}$ (i.e. the maximal $n_q$ such that $\rho_{f,\iota}|_{G_{\Q_q}}$ is congruent to $\pm {\rm Id}$  mod $\varpi^{n_q}$).  Then ${\eta^Q_{f}}/\Fitt_\cO(\Phi^Q_{f})$ is $\cO/\varpi^{2n_q}$.
\end{theorem}

We use a different (logarithmic) normalization of the Wiles defect (see Definition \ref{WilesDefect} below) to make
it behave well with change of coefficient ring $\cO$.
With our  definition the Wiles defect $\delta_{\lambda_f,\T^Q}(\T^Q)$  (of $\T^Q$  relative to $\lambda_f$) in the theorem above becomes $2n_q/e$ where $e$ is the ramification index of $\cO$.  

In the full version of our results, see Theorem \ref{mc}, we allow $f$ to be an arbitrary newform of squarefree level and trivial nebentypus, for which the residual representation $\rhobar$ satisfies the Taylor--Wiles conditions, as well as a condition on the ramification, and we allow the quaternion algebra $D$ to be ramified at multiple trivial primes for $\rhobar$. For ease of exposition, we will continue to focus on the case outlined above for the remainder of the introduction.

\begin{remark}\label{Remark-OnComputations}
As a small point, we remark that while Theorem \ref{thm:example} is technically a complete answer to the question of determining the Wiles defect, the numbers $n_q$ may be difficult to compute in the specific case considered here. The reason for this is that while the original representation $\rhobar:G_{\Q}\to \GL_2(\F_3)$ at level $11$ is defined by a modular form with rational coefficients, the coefficient field $K_f$ of the newform $f\in S_2(\Gamma_0(11q))$ may be arbitrarily large.

For example, in the case of the first trivial trivial prime $q = 193$, one can show by process of elimination that the newform $f\in S_2(\Gamma_0(11\cdot 193))$ giving rise to $\rhobar$ is the form computed in \cite[\href{http://www.lmfdb.org/ModularForm/GL2/Q/holomorphic/2123/2/a/h/}{Newform 2123.2.a.h}]{lmfdb}, for which we have $[K_f:\Q] = 46$. This may make a precise determination of $n_{193}$ computationally infeasible.

In section \ref{semistable} we consider a different specialization of Theorem \ref{mc} to the case when the newform $f$ has rational coefficients (but we do not assume that the residual representation $\rhobar$ arises from a modular form with rational coefficients at minimal level). In this case, $f$ corresponds to a semistable elliptic curve $E_f$, and we show that $n_q$ may be computed in terms of the Tate uniformization of $E$ at $q$. This allows us to give several numerical computations of Wiles defects via our main theorem.
\end{remark}

In the study of congruences between modular forms (cf. \cite{Hida}), and computing congruence ideals attached to a newform $f$, it is found that \emph{cohomological} congruence ideals $\eta_f^{Q,\coh}$, defined similarly to $\eta_f^Q$ but in terms of the $\T^Q$ module $H^Q:=H^1(X,\cO)_{\ffrm}$ rather than $\T^Q$,\footnote{We are suppressing some subtleties here to do with the fact that $H^1(X,\cO)_{\ffrm}$ has generic rank $2$ as a $\T^Q$-module, rather than $1$.} are easier to study (and are related directly to algebraic parts of $L$-values) than the congruence ideal $\eta_f^Q$ of the relevant Hecke algebra $\T^Q$.

In \cite{DiamondMult1}, Diamond gives a generalization of Theorem \ref{numerical criterion} which shows that $|\Phi_f^Q| = |\cO/\eta_f^{Q,\coh}|$ if and only if $\phi^Q:R^Q\to \T^Q$ is an isomorphism of complete intersections and $H^Q$ is free of rank $2$ over $\T^Q$.

To relate the congruence ideals $\eta^Q_f$ and $\eta^{Q,\coh}_f$, the method in most references (\cite{Hida1}) is to prove by delicate arithmetic-geometric means  that $H^Q$ is actually free of rank $2$ over $\T^Q$, which trivially implies that $\eta^Q_f = \eta^{Q,\coh}_f$.

In the case when $q$ is not a trivial prime, the same patching arguments that imply $R^Q$ and $\T^Q$ are complete intersections were shown in \cite{DiamondMult1} and \cite{Fujiwara} to imply that $H^Q$ is free over $\T^Q$ (a fact which was also shown more directly in \cite{RibetMult2}). On the other hand, when $q$ is a trivial prime, the main result of \cite{RibetMult2} shows that $H^Q$ is not free over $\T^Q$. Instead it is proved that $\dim_k H^Q/\ffrm=4$, whereas $H^Q[1/p]$ is free of rank $2$ over $\T^Q\otimes\Q_p$, which is referred to as a multiplicity two phenomenon.

The work of the third author in \cite{Manning} uses patching techniques to get another proof of Ribet's result in a more general situation. In addition to re-proving Ribet's multiplicity two result, this approach also yields an additional statement about the endomorphism ring of $H^Q$ as a $\T^Q$-module (see Theorem 1.2 of loc. cit.). In Theorem \ref{thm:surjective} we show that this endomorphism statement implies that $\eta_f^Q = \eta_f^{Q,\coh}$, even in the case when $H^Q$ is not free over $\T^Q$.\footnote{Note that this does not contradict Diamond's numerical criterion from \cite{DiamondMult1}, since in the case when $H^Q$ is not free $\T^Q$ will not be a complete intersection, and so $|\Phi_f^Q| > |\cO/\eta_f^Q| = |\cO/\eta_f^{Q,\coh}|$.}

This shows that a weaker condition than freeness can suffice to show that the two congruence modules are the same, which holds in certain cases when freeness is actually false. Moreover, the techniques of \cite{Manning} give a possible approach for proving this condition in more generality, by carefully analyzing the geometry of certain local Galois deformation rings. We hope that these ideas will be useful more generally to relate congruence ideals of Hecke algebras and congruence ideals of cohomology groups arising from arithmetic manifolds on which they naturally act. 

\subsection{Strategy of proof of Theorem \ref{mc}}

Here is a sketch  of the strategy of the proof of our main theorem,  Theorem \ref{mc}. The notation we use here is not the same as in the main text,  and we simplify the arguments, skipping  over some nuances, to indicate the broad strategy. We consider the full deformation ring $R(Q)$ which arises from considering semistable deformations of determinant the $p$-adic cyclotomic character $\epsilon$,  that are unramified outside $Q=\{3,11,q\}$, and finite flat at 3. We consider  the  anemic  Hecke algebra $\T(Q)$ (without operators $U_q$ for $q \in Q$)  that  acts faithfully on the cohomology $H^1(X_0(11q^2),\cO)_{\ffrm_q}$ of the modular curve.  (Here $\ffrm_Q$ is a maximal ideal of the full Hecke algebra which is a pull back of the maximal ideal $\ffrm$ of $\T^Q$.) There is a surjective map of $\cO$-algebras $R(Q) \onto \T(Q)$ 
that is an    isomorphism of complete intersections $R(Q)\xrightarrow{\sim} \T(Q)$ by \cite{Wiles} and \cite{DiamondMult1}.

The rings  $R^Q$ and $\T^Q$ are quotients of $R(Q)$ and $\T(Q)$ respectively. We can therefore pull our augmentation $\lambda_f:\T^Q\to\cO$ back to $\lambda_f:\T(Q)\to\cO$, and use that for this pull back, the Wiles defect is trivial by   Theorem \ref{numerical criterion} and  Theorem \ref{fred}. For this one uses another augmentation $\lambda_g: \T(Q) \to \cO$ arising from a newform $g \in S_2(\Gamma_0(N(\rhobar))$ which exists because of the level lowering results of \cite{RibetInv100}. Thus to compute the Wiles defect of $\T^Q$, it suffices to compute the change in cotangent spaces (at $\lambda_f$)  between $R(Q)$ and $R^Q$ and the change in congruence modules between $\T(Q)$ and $\T^Q$.

The change in cotangent spaces is considered in section \ref{sec:patchandgrow}. The main idea is that the Taylor--Wiles--Kisin patching method (see Theorem \ref{thm:patching} for details) implies that $R(Q)$ and $R^Q$ are quotients of certain explicit rings $R_\infty$ and $R_\infty^{\st}$ by a regular sequence. Here $R_\infty$ and $R_\infty^{\st}$ are both power series rings over completed tensor products of various local Galois deformation rings. In all cases relevant to us, the local deformation rings were computed precisely in \cite{Shotton}, and thus we can determine $R_\infty$ and $R_\infty^{\st}$ completely explicitly.

Pulling back $\lambda_f$ to augmentations on $R_\infty^{\st}$ and $R_\infty$, one can then consider the cotangent spaces $\Phi_{R_\infty^{\st}}$ and $\Phi_{R_\infty}$, which depend only on the rings $R_\infty^{\st}$ and $R_\infty$ and the map $\lambda:R_\infty^{\st}\to \cO$, and not on the structures of $R(Q)$ and $R^Q$. While it is not possible to directly determine $\Phi_{R^Q}$ and $\Phi_{R(Q)}$ from $\Phi_{R_\infty^{\st}}$ and $\Phi_{R_\infty}$ (at least without knowing the regular sequence generating $\ker(R_\infty\onto R(Q))$ and $\ker(R_\infty^{\st}\onto R^Q)$) we can show via a commutative algebra argument (Theorem \ref{mc1}) that $\ker(\Phi_{R(Q)}\onto \Phi_{R^Q})$ is isomorphic to $\ker(\Phi_{R_\infty}\onto \Phi_{R_\infty^{\st}})$. The cardinality of the later kernel can then be computed precisely by a computation with the local deformation rings (Proposition \ref{key-local-comp}) which gives the desired comparison between $|\Phi_{R(Q)}|$ and $|\Phi_{R^Q}|$.

The change in congruence modules between $\cO/\eta_{R(Q)}$ and $\cO/\eta_{R^Q}$ is considered in sections \ref{sec:congruences} and \ref{sec:ribet}. We first equate these with the corresponding cohomological congruence modules, via Theoprem \ref{fred} in the modular curve case, and the work of \cite{Manning} combined with Theorem \ref{thm:surjective} (as described above) in the Shimura curve case (see Theorem \ref{thm:sat} and Corollary \ref{cor:sat}). The change in cohomological congruence modules is then computed by an application of  the work of \cite{RibetInv100} and  \cite{RiTa} (see Proposition \ref{change-of-eta}).

Relying on the results of \cite{RiTa} in this way requires us to assume that $\rhobar$ is ramified at least at one prime in the discriminant of $D$. By using the Jacquet--Langlands isomorphism to add or remove certain primes from the discriminant of $D$, we may reduce this condition to simply requiring that $\rhobar$ is ramified and Steinberg at enough primes, regardless of whether they are in the discriminant of $D$ (see Corollary \ref{eta:cor2}).

For this paper we have, for simplicity, restricted to the case when $\rhobar$ is a semistable representation of $G_{\Q}$, so this is only a fairly minor restriction. However this would be a much more restrictive condition if we were to give the natural generalizations of our results to possibly non-semistable representations of $\rhobar:G_F\to\GL_2(k)$, for $F$ a totally real number field. It would thus be worthwhile to find an alternative approach to computing the change in congruence modules, or to directly computing the Wiles defect, which does not depend on the results of \cite{RiTa}.

\subsection{Future work}

Theorem \ref{mc} shows that the Wiles defect depends solely on the local properties of the representation $\rho_{f,\iota}$ at the primes in the discriminant of $D$, and thus only on the ring $R_\infty^{\st}$ and the induced augmentation $\lambda:R_\infty^{\st}\to \cO$ from section \ref{sec:patchandgrow}. Based on this, one might hope that the Wiles defect could be computed directly from the local deformation rings, similarly to how the change in cotangent spaces was computed in section \ref{sec:patchandgrow}. Such a computation would remove the reliance on the results \cite{RiTa} from our approach, and would thus allow us to significantly weaken the hypotheses of our main theorem. It would also potentially allow us to compute the Wiles defect in contexts where there is no natural generalization of \cite{RiTa}. We intend to pursue these ideas in a follow-up paper.

The work of Hida \cite{Hida}  proves a formula of the type  $${L(1,\ad(f)) \over { i\pi\Omega}}=\card{ \cO/\eta_f(\T(Q))},$$ (up to a unit in $\cO$), relating $L$-values and congruence modules (we blur here the distinction between congruences module for the Hecke algebra $\T(Q)$ and the  congruence module for its module  $H^1(X_0(NQ),\cO)_\ffrm$). Here  $L(s,\ad(f))$ is  the $L$-function and $\Omega$ a suitable period  of   the adjoint motive associated to $f$. The work of Wiles \cite{Wiles}, when combined with \cite{Hida},  gives a proof of  the Bloch-Kato conjecture  (in many cases) for the corresponding  $L$-value by showing the equality of the orders of the congruence module $\cO/\eta_f({\T(Q)})$  and the cotangent space $\Phi_{R(Q)}$  associated  to $f$ (the latter is the  Bloch-Kato Selmer group associated to the adjoint motive of $f$).  Our work does not directly involve $L$-functions, but one could ask  if the Wiles defect we study has a relation to $L$-values (perhaps related to the quotient of ``cohomological'' and ``motivic'' periods associated to the form $f$ or its Jacquet-Langlands correspondent on the quaternion algebra $D_Q$). The observations in \S \ref{semistable} that relate  the Wiles defect,  in the case of augmentations of Hecke algebras arising from semistable elliptic curves,  via our main theorem,  to tame regulators of local  periods of the Tate  elliptic curves at the primes in $Q$ (cf. \cite{MT}), give  a geometric decription of the Wiles defect   in the case we study  and  are suggestive of such a relationship.

 We end the introduction by making a few remarks contrasting the contexts of \cite{TiUr} and our paper, which both  study the Wiles defect. The paper \cite{TiUr} considers Hecke algebras which act on the Betti cohomology of hyperbolic 3-manifolds. Thus  the reason for the failure of complete intersection of the Hecke algebras studied in loc. cit. is global, namely that the relevant Betti cohomology (the relevant degree being 1 and 2)  is not concentrated in a single degree (and thus the defect $\ell_0=1$). In \S 5 of \cite{TiUr}, the Wiles defect of the Hecke algebras studied there is  conjecturally related  to (integral) regulators of motivic elements, and to the ratio of cohomological and motivic periods. 
 
 In the classical situation studied in this paper, of Hecke algebras acting on the cohomology $H^1(X_Q,\cO)$ of Shimura curves ${X_Q}_{/\Q}$
 arising from an indefinite quaternion algebra $D_Q$ over $\Q$, the reason for the failure of the Hecke algebra to be a complete intersection is local, namely that the {\em Steinberg}  deformation rings at {\em trivial} primes dividing the discriminant of $D_Q$ are not complete intersections. We compute in many cases the Wiles defect (cf. Theorem \ref{mc}), and give an explicit formula for the Wiles defect (which is a global quantity) as a sum of local terms at such primes. The local terms  at a prime $q$  measure the depth of the congruence of $\rho_f|_{G_q}$ to the trivial representation up to sign,  where $G_q$ is a decomposition group at $q$.

\subsection{Structure of the paper}
We describe the contents of the paper.  The key sections of the paper are as follows. In section \ref{sec:congruence} we develop the commutative algebra we need in this paper related to congruence modules (cf. Theorem \ref{thm:surjective})  and define the Wiles defect (cf. Definition \ref{WilesDefect}).  In sections \ref{sec_deformation_theory} and \ref{hecke} we introduce the deformation rings and Hecke algebras which are the main objects studied in the paper.
In sections \ref{sec:patch} and \ref{sec:patchandgrow} we use patching to show that certain relative local cotangent spaces are isomorphic to related relative global cotangent spaces (cf. Theorem \ref{thm:patching}, Theorem \ref{mc1} and Corollary \ref{cor:main}). In section \ref{sec:local-calc} we give the computation of the cardinality of relative local cotangent spaces at a trivial prime $q$ (cf. Proposition \ref{key-local-comp}).  In section \ref{sec:congruences}  we show that to prove our main theorem about the Wiles defect $\delta_\lambda(\T^Q)$ we study, Theorem \ref{mc}, it is enough to compute a related  cohomological Wiles defect  $\delta_{\lambda,T^Q} (H^1(X_Q,\cO)_{\ffrm}  )$ related to the $\T^Q$ module $H^1(X_Q,\cO)_{\ffrm}$ (cf. Corollary \ref{cor:sat}). In section \ref{sec:ribet} we use \cite{RiTa}
to compute  the change of the cohomological $\eta$-invariant when we pull back our augmentation $\T^Q$ to $\T(Q)$ (cf. Proposition \ref{change-of-eta}). In section \ref{sec:defect} we pull all our earlier work together to prove Theorem \ref{mc}. We end in section \ref{semistable} by computing the Wiles defect for semistable elliptic curves in terms of tame regulators of  local Tate periods (cf. Proposition \ref{prop:MT} and Corollary \ref{cor:MT}) and giving some numerical examples.

\subsection*{Acknowledgements}

G.B. received support from the DFG grant FG 1920.  We would like to thank the referees for a critical reading that led to  an improvement of the exposition and  the correction of  a few inaccuracies of an earlier version.

\subsection*{Notation}

We will fix a prime $p>2$ throughout the paper and fix an embedding $\iota: \overline \Q \hookrightarrow  \overline \Q_p$.

For any field $F$, we will let $G_F$ denote its absolute Galois group. For any prime $q$ we will write $G_q:=G_{\Q_q}$, and let $I_q\le G_q$ be the inertia subgroup. For the entirety of this paper, we will fix a choice of embedding $G_q\into G_{\Q}$ for each prime $q$. All representations of $G_F$ will be assumed to be continuous.

We denote by $\epsilon$ the $p$-adic cyclotomic character $\epsilon:G_\Q \ra \Z_p^\times$ and  also  any  of the characters  to $R^\times$ we get  for any $\Z_p$-algebra $R$, by composing it with the map $\Z_p \ra R$.

For a newform $f$, we denote by $\rho_{f, \iota}$ (or simply $\rho_f$ as we have fixed the embedding $\iota$) the Galois representation associated  by Eichler--Shimura to $f,\iota$.

We will denote by $\cO$ the valuation ring of a finite extension $E$ of $\Q_p$ and by $k$ the residue field of $\cO$, and we let $\CNLO$ be the category of complete, Noetherian local $\cO$-algebras with residue field $k$.

For an $\cO$-algebra $R$, an augmentation $\lambda$ of $R$ will always mean a surjective $\cO$-algebra homomorphism $\lambda:R\to \cO'$, where $\cO'$ is the ring of integers in a finite extension of $E$ (we will almost always take $\cO=\cO'$). For an $\cO$-module $M$ that is a finite abelian group, we denote by $\ell_\cO(M)$ the length of $M$ as an $\cO$-module.
 For $\alpha \in \cO$, we denote by $\ord_\cO(\alpha)=\ell_\cO(\cO/(\alpha))$.

For a Galois representation $\rhobar:G_{\Q}\to \GL_2(\overline{\F}_p)$ which is finite flat at $p$, we will let $N(\rhobar)$ represent its Artin conductor.

\section{Trivial primes for $\rhobar$}\label{sec:trivial primes}

Let $\rhobar:G_{\bf Q} \rightarrow \GL_2(k)$ be a continuous and irreducible mod $p$ representation.

\begin{definition}
  We say that a prime  $q$ is \emph{trivial} for $\rhobar$ if $q $ is 1 mod $p$, $\rhobar$  is unramified at $q$ and $\rhobar({\rm Frob}_q)= \pm {\rm Id}$. We say a prime  $q$ is  \emph{Steinberg} for $\rhobar$ if  $\rhobar|_{G_{q}}$ up to an unramified twist is of the form  \[  \left( \begin{array}{cc} \epsilon& \ast \\ 0 & 1 \end{array} \right),\]  with $ \epsilon$ the mod $p$ cyclotomic     
         character and $G_q$ a decomposition group at $q$.
 \end{definition}

We impose that $\rhobar$  satisfies the following conditions:

\begin{itemize}

\item ${\rm det}(\rhobar)=\epsilon$.

\item Each prime $\ell (\ne p)$ at which $\rhobar$ is ramified is Steinberg for $\rhobar$. 

\item $\rhobar|_{G_p}$ is finite flat.

\item    $\rhobar$ is modular, i.e., arises from a newform in $S_2(\Gamma_0(N(\rhobar)))$  with respect to an
         embedding $\iota_p:\overline{{\bf Q}} \rightarrow
         \overline{\bf Q}_p$ that we fix.   This is implied by the previous hypotheses by Serre's modularity conjecture which is now a theorem \cite{KW}.

\end{itemize}

Note that  $N(\rhobar)>1$ and square-free, and our conditions imply that $\Ad^0(\rhobar)$ is absolutely irreducible
and the order of ${\rm im}(\rhobar)$ is divisible by $p$ (otherwise $\rhobar$ has Serre weight 2
and Artin conductor 1 by assumptions on determinant and ramification above). This  implies that $\rhobar|_{\Q(\zeta_p)}$ is irreducible and thus satisfies the ``Taylor-Wiles hypothesis'' needed for the arguments below. 

In our work below we will fix positive square-free and coprime integers $N$ and $Q$, with the property that:

\begin{itemize}  
\item  $N$ divides $N(\rhobar)$; 
\item  $N(\rhobar)$ divides $NQ$;
\item all primes dividing $Q$ are Steinberg for $\rhobar$.
\end{itemize}  

We also fix a prime $r>>0$ such that no lift of $\rhobar|_{G_r}$ is ramified at $r$. This exists by our hypothesis on $\rhobar$ and \cite[Lemma 3]{DT2}.  This allows us to impose $\Gamma_1(r^2)$ level structure whenever needed, and so we may ignore below  any issues arising from isotropy.
 
\section{Cotangent spaces, congruence modules and the Wiles defect}\label{sec:congruence}

\begin{definition} Given any $R \in \CNLO$ and any continuous homomorphism $\lambda: R \ra \cO$  we define the \emph{cotangent space of $R$ with respect to $\lambda$} to be $\Phi_{\lambda,R}:= (\ker \lambda)/(\ker \lambda)^2$.
\end{definition}

Note that 
\cite[\href{https://stacks.math.columbia.edu/tag/00RP}{Lemma 00RP}, \href{https://stacks.math.columbia.edu/tag/02HP}{Lemma 02HP}]{stacks-project} implies that $\Phi_{\lambda,R} = \Omega_{R/\cO} \otimes_\lambda \cO$, where $\Omega_{R/\cO}$ is the $R$-module of K\"ahler differentials.

First note the following:

\begin{lemma}\label{lem:cot fin}
	If $R$ is a finite, local, reduced $\cO$-algebra, then $\Phi_{\lambda,R}$ is a finite group.
\end{lemma}
\begin{proof}
	As $R$ is reduced, $R[1/\varpi] = R\otimes_\cO E$ is as well. It follows that $R[1/\varpi]$ is a product of finite extensions of $E$, and so $\Omega_{R/\cO}\otimes_\cO E = \Omega_{R[1/\varpi]/E} = 0$. Thus
	\[\Phi_{\lambda,R}\otimes_\cO E = (\Omega_{R/\cO} \otimes_\lambda \cO)\otimes_\cO E = (\Omega_{R/\cO}\otimes_\cO E)\otimes_\lambda E  = 0\]
	But now as $R$ is a finite $\cO$-algebra, $\Omega_{R/\cO}$ is a finitely generated $R$-module, and hence $\Phi_{\lambda,R}$ is a finitely generated $\cO$-module. It follows that $\Phi_{\lambda,R}$ is indeed finite.
\end{proof}

For the remainder of this section, let $R$ denote a finite, local $\cO$-algebra. Assume that $R$ is $\varpi$-torsion free and reduced. This implies that the map $R \to R[1/\varpi] = R\otimes_\cO E$ is injective, and $R[1/\varpi]$ is a product of finite extensions of $E$. Also fix an augmentation $\lambda:R\to \cO$. 

For any finitely generated $R$-module $M$, let $M^* = \Hom_{\cO}(M,\cO)$, with its natural $R$-module structure. We say that $M$ is \emph{reflexive} if the natural map $M\to M^{**}$ is an isomorphism (note that this is equivalent to simply saying that $M\cong M^{**}$, since $M^*$, and hence $M^{**}$, is reflexive for any $M$). Note that if $M$ is finitely generated over $R$, then it is also finitely generated over $\cO$, so it follows that $M$ is a reflexive $R$-module if and only if it is finitely generated over $\cO$ and $\varpi$-torsion free.

For any reflexive $R$-module $M$, note that we have natural isomorphisms
\[(M[\ker \lambda])\otimes_\cO E \cong (M\otimes_\cO E)/\ker \lambda \cong (M/\ker\lambda)\otimes_{\cO}E.\]
We will define the $\lambda$-rank of $M$ to be $\rank_\lambda M = \dim_E (M[\ker \lambda])\otimes_\cO E$, so that if $\rank_\lambda M = d$ then  $M[\ker \lambda] \cong \cO^d$ as an $\cO$-module. It is not hard to see that $\rank_\lambda M = \rank_\lambda M^*$.

Now let $\omega = R^* = \Hom_{\cO}(R,\cO)$ be the dualizing module of $R$. Then $\omega$ is a reflexive $R$-module and we have $\omega^* = R$.

For any reflexive $R$-module $M$ we have a natural, $R$-equivariant perfect pairing $\langle\ ,\ \rangle_M:M^*\times M\to \cO$ (given by $\langle f,x\rangle_M = f(x)$). Note that this induces an isomorphism $M[\ker \lambda] \cong \Hom_{\cO}(M^*/\ker \lambda,\cO)$. Also, identifying $M$ with $M^{**}$, we have $\langle x,y\rangle_{M^*} = \langle y,x\rangle_M$. 

We now make the following definition:

\begin{definition}\label{def:cng}
	For any reflexive $R$-module $M$, the \emph{congruence module of $M$ with respect to $\lambda$}, denoted $\cng{\lambda}(M)$, is the cokernel of the composition
	\[\iota_{M,\lambda}:M[\ker \lambda]\to M\xrightarrow{\sim} \Hom_\cO(M^*,\cO) \to \Hom_{\cO}(M^*[\ker \lambda],\cO).\]
\end{definition}

The map $\iota_{M,\lambda}$ is an isomorphism after inverting $\varpi$, so it is always injective, and $\cng{\lambda}(M)$ is always finite.

\begin{lemma}\label{lem:cng det}
	If $M$ is reflexive and $x_1,\ldots,x_d$ (resp. $f_1,\ldots,f_d$) is an $\cO$-basis for $M[\ker \lambda]$ (resp. $M^*[\ker \lambda]$) then $\cng{\lambda}(M)$ is isomorphic to the cokernel of the map $\cO^d\to \cO^d$ given by the matrix $(\langle f_i,x_j\rangle_M)_{ij}$. In particular, $\card{\cng{\lambda}(M)} = \card{\cO/\left(\det(\langle f_i,x_j\rangle_M)_{ij}\right)}$.
\end{lemma}
\begin{proof}
	If we identify $M[\ker \lambda]$ and $M^*[\ker \lambda]$ with $\cO^d$ using the bases, $\iota_{M,\lambda}$ is identified with the map $\cO^d\to \cO^d$ given by the matrix $(\langle f_i,x_j\rangle_M)_{ij}$. The claim follows.
\end{proof}

In light of this lemma, in the case where $\rank_\cO M[\ker \lambda] = \rank_\cO M^*[\ker \lambda] = 1$, we define the ideal $\cngid{\lambda}(M) = (\langle f_1,x_1\rangle_M)\subseteq \cO$, where $M[\ker \lambda] = \cO x_1$ and $M^*[\ker \lambda] = \cO f_1$. Then we have $\cng{\lambda}(M)\cong \cO/\cngid{\lambda}(M)$, and $\cngid{\lambda}(M)$ is independent of the choice of $x_1$ and $f_1$.

\begin{lemma}\label{lem:cng det self-dual}
	If $M$ is a reflexive $R$-module with an $R$-equivariant perfect pairing $\langle\ ,\ \rangle:M\times M\to \cO$ (i.e. $M$ is \emph{self-dual}) and $x_1,\ldots,x_d$ is a basis for $M[\ker \lambda]$ then $\cng{\lambda}(M)$ is isomorphic to the cokernel of the map $\cO^d\to \cO^d$ given by the matrix $(\det(\langle x_i,x_j\rangle)_{ij}$. In particular, $\card{ \cng{\lambda}(M) }= \card{\cO/\left(\det(\langle x_i,x_j\rangle)_{ij}\right)}$. (Note that $\cng{\lambda}(M)$ is independent of this choice of pairing $\langle\ ,\ \rangle$, by definition.)
\end{lemma}
\begin{proof}
	The pairing $\langle\ ,\ \rangle$ induces an isomorphism $\varphi:M\xrightarrow{\sim}M^*$ of $R$-modules for which $\langle x,y\rangle = \langle \varphi(x),y\rangle_M$. Letting $f_i = \varphi(x_i)$, the claim now follows from the previous lemma.
\end{proof}

\begin{lemma}\label{lem:cng direct sum}
	If $M_0$ and $M_1$ are reflexive $R$-modules, then $\cng{\lambda}(M_0\oplus M_1) \cong \cng{\lambda}(M_0)\oplus \cng{\lambda}(M_1)$.
\end{lemma}
\begin{proof}
	Let $M = M_0\oplus M_1$ and note that we have natural isomorphisms $M[\ker \lambda] = (M_0\oplus M_1)[\ker \lambda] = M_0[\ker \lambda]\oplus M_1[\ker \lambda]$, $M^* = \Hom_{\cO}(M_0\oplus M_1,\cO) \cong M_0^*\oplus M_1^*$, $M^*[\ker \lambda] =  M_0^*[\ker \lambda]\oplus M_1^*[\ker \lambda]$ and $\Hom_\cO(M^*[\ker \lambda],\cO)\cong \Hom_\cO(M_0^*[\ker \lambda],\cO)\oplus \Hom_\cO(M_1^*[\ker \lambda],\cO)$.
	
	Under these identifications we have 
	\[\iota_{M,\lambda} = \iota_{M_0,\lambda}\oplus \iota_{M_1,\lambda}:M_0[\ker \lambda]\oplus M_1[\ker \lambda]\to \Hom_\cO(M_0^*[\ker \lambda],\cO)\oplus \Hom_\cO(M_1^*[\ker \lambda],\cO),\]
	from which the claim follows.
\end{proof}

\begin{lemma}\label{lem:cng dual}
	For any reflexive module $M$, $\cng{\lambda}(M)\cong \cng{\lambda}(M^*)$.
\end{lemma}
\begin{proof}
	Follows from Lemma \ref{lem:cng det}, and the fact that $\langle f_i,x_j\rangle_M = \langle x_j,f_i\rangle_{M^*}$.
\end{proof}

\begin{lemma}\label{lem:cng(R)}
	$\cngid{\lambda}(R)=\lambda(R[\ker \lambda])$.
\end{lemma}
\begin{proof}
	Let $R[\ker \lambda] = \cO\alpha$, and note that $R^*[\ker \lambda] = \omega[\ker \lambda] = \Hom_{\cO}(R/\ker \lambda,\cO) = \cO\lambda\subseteq R^*$. Then by Lemma \ref{lem:cng det}, $\cngid{\lambda}(R) = (\langle \lambda,\alpha\rangle_R) =(\lambda(\alpha)) = \lambda(R[\ker\lambda])$.
\end{proof}

Lemma \ref{lem:cng(R)} shows that our definition agrees with the definition used in \cite{Wiles}. We also note the following, which implies that our definition of $\Psi_\lambda(M)$ agrees with the module $\Omega$ considered in \cite[Theorem 2.4]{DiamondMult1}:

\begin{lemma}
Let $I_\lambda = R[\ker \lambda]\subseteq R$. For any reflexive $R$-module $M$ we have $\ds\cng{\lambda}(M)\cong \frac{M}{M[\ker\lambda]+M[I_\lambda]}$.
\end{lemma}
\begin{proof}
Write $\frp_\lambda = \ker\lambda$. We first claim that the natural map $M\xrightarrow{\sim} \Hom_\cO(M^*,\cO)\to \Hom_{\cO}(M^*[\frp_\lambda],\cO)$ induces an isomorphism $M/M[I_\lambda]\xrightarrow{\sim}\Hom_{\cO}(M^*[\frp_\lambda],\cO)$.

The map $M\to \Hom_{\cO}(M^*[\frp_\lambda],\cO)$ clearly factors through $M/\frp_\lambda M$. Also as $M^*[\frp_\lambda]$ is saturated in $M^*$, this map is surjective, so we have a surjection $M/\frp_\lambda M\onto \Hom_{\cO}(M^*[\frp_\lambda],\cO)$.

It is easy to see that this map is an isomorphism after inverting $\varpi$ (this follows from the fact that $M[1/\varpi] = \frp_\lambda M[1/\varpi]\oplus I_\lambda M[1/\varpi]$ and $M^*[1/\varpi] = \frp_\lambda M^*[1/\varpi]\oplus I_\lambda M^*[1/\varpi]$), and so the map identifies $\Hom_{\cO}(M^*[\frp_\lambda],\cO)$ with the maximal $\varpi$-torsion free quotient of $M/\frp_\lambda M$. But now as $M[1/\varpi]/\frp_\lambda M[1/\varpi] = (M[1/\varpi])[I_\lambda]$, $M[I_\lambda]$ is the saturation of $\frp_\lambda M$ in $M$ as an $\cO$-module, and so $M/M[I_\lambda]$ is the maximal $\varpi$-torsion free quotient of $M/\frp_\lambda M$. Thus we indeed have the desired natural isomorphism $M/M[I_\lambda]\xrightarrow{\sim} \Hom_{\cO}(M^*[\frp_\lambda],\cO)$.

But now identifying $M/M[I_\lambda] = \Hom_{\cO}(M^*[\frp_\lambda],\cO)$ under this isomorphism, the map $\iota_{M,\lambda}:M[\frp_\lambda]\to  \Hom_{\cO}(M^*[\frp_\lambda],\cO)$ from Definition \ref{def:cng} is identified with the map $M[\frp_\lambda]\into M\onto M/M[I_\lambda]$ and so 
\[\cng{\lambda}(M) = \coker(\iota_{M,\lambda}) = \frac{M/M[I_\lambda]}{(M[\frp_\lambda]+M[I_\lambda])/M[I_\lambda]}\cong \frac{M}{M[\frp_\lambda]+M[I_\lambda]}.\]
\end{proof}

Our main goal for this paper is to compare the orders of $\Phi_{\lambda,R}$, $\cng{\lambda}(R)$ and $\cng{\lambda}(M)$ (for a specific choice of $M$). We thus make the following definition:

\begin{definition}\label{WilesDefect}
	Let $R$ denote a finite, local $\cO$-algebra, which is $\varpi$-torsion free and reduced. Let $M$ be a reflexive $R$-module with $\rank_\lambda M = d>0$.

	 The \emph{Wiles defect} of $M$ is the quantity
	\[\delta_{\lambda,R}(M) = \frac{d\log\card{\Phi_{\lambda,R}}-\log\card{\cng{\lambda}(M)}}{d \log\card{\cO/p}},\] 
	which we will denote by $\delta_\lambda(M)$ when $R$ is clear from context.

\end{definition}

Note that $\delta_{\lambda,R}(M)$ is a rational number. Also note that by this definition, if $E'/E$ is a finite extension with ring of integers $\cO'$, then $\delta_{\lambda\otimes \cO',R\otimes \cO'}(M \otimes \cO') = \delta_{\lambda,R}(M)$. Essentially, $\delta_{\lambda,R}(M)$ is independent of our choice of coefficient ring. This  allows us to assume that $\cO$ is as large as needed in the arguments below.

Also note that for any $k\ge 1$, $\delta_{\lambda,R}(M^{\oplus k}) = \delta_{\lambda,R}(M)$ by Lemma \ref{lem:cng direct sum}.

As shown in \cite{DiamondMult1}, we have:

\begin{theorem}
	Let $R$ be as above, and let $M$ be a reflexive $R$-module. Let $d = \rank_\lambda M$, and assume that $M$ has rank at most $d$ at each generic point of $R$. Then $\delta_{\lambda,R}(M)\ge 0$, and we have $\delta_{\lambda,R}(M) = 0$ if and only if $R$ is a complete intersection and $M$ is free of rank $d$ over $R$.
\end{theorem}

In the situation we will consider in this paper, $R$ will not be a complete intersection (or even Gorenstein) and $M$ will not be free over $R$, and so both of the Wiles defects $\delta_{\lambda,R}(R)$ and $\delta_{\lambda,R}(M)$ will be positive. Our goal will be to determine these defects.

We will use the following theorem below to compare $\delta_{\lambda,R}(R)$ and $\delta_{\lambda,R}(M)$:

\begin{theorem}\label{thm:surjective}
	Let $M$ be a reflexive $R$-module \emph{of $\lambda$-rank $1$}. Then $\card{\cng{\lambda}(R)}\ge \card{\cng{\lambda}(M)}$, and so $\delta_{\lambda,R}(R)\le \delta_{\lambda,R}(M)$. Equality holds if the natural map $R\to \End_{R}(M)$ is an isomorphism.
\end{theorem}
\begin{proof}
	First note that $M^*[\ker \lambda]\subseteq M^* = \Hom_\cO(M,\cO)$ consists precisely of the $\cO$-module homomorphisms $f:M\to\cO$ satisfying $f(rx) = \lambda(r)f(x)$ for all $r\in R$, $x\in M$.
	
	Let $R[\ker \lambda] = \cO\alpha$ and $M^*[\ker \lambda] = \cO\beta$ for some $\alpha\in R$, $\beta\in M^*$. Regarding $\beta$ as a map $\beta:M\to \cO$, note that $\beta$ must be surjective (as otherwise we would have $\varpi^{-1}\beta\in M^*[\ker \lambda]$, contradicting the choice of $\beta$).
	
	By Lemmas \ref{lem:cng det} and \ref{lem:cng(R)}, $\cngid{\lambda}(R) = (\lambda(\alpha))$ and $\cngid{\lambda}(M) = \beta(M[\ker \lambda])$. Now as $\alpha\in R[\ker \lambda]$, we clearly have $\alpha M\subseteq M[\ker \lambda]$ and thus
	\[\cngid{\lambda}(M) = \beta(M[\ker \lambda]) \supseteq \beta(\alpha M) = \lambda(\alpha)\beta(M) = \cO\lambda(\alpha) = \cngid{\lambda}(R),\]
	which automatically implies $\card{\cng{\lambda}(R)}\ge \card{\cng{\lambda}(M)}$.
	
	Moreover, if we have $\card{\cng{\lambda}(R)}>\card{\cng{\lambda}(M)}$, then we must have $\beta(\alpha M)\subsetneq\beta(M[\ker \lambda])$, and so $\beta(\alpha M) \subseteq \varpi\beta(M[\ker \lambda]) = \beta(\varpi M[\ker \lambda])$. As $\alpha M\subseteq M[\ker\lambda]$ and $\beta|_{M[\ker\lambda]}:M[\ker\lambda]\to\cO$ is injective (as $M[\ker\lambda]$ is a free rank $1$ $\cO$-module) this implies that $\alpha M\subseteq \varpi M[\ker \lambda] \subseteq \varpi M$.
	
	As $M$ is $\varpi$-torsion free, it follows that we may define an endomorphism $f:M\to M$ by $f(x) = \varpi^{-1}\alpha x$. Clearly $f\in \End_R(M)$ and if $i:R\to \End_R(M)$ is the natural map we have $i(\alpha) = \varpi f$. In particular $i(\alpha)\in \varpi\End_R(M)$. But now as $R[\ker \lambda]$ is saturated in $R$, $\alpha\not\in \varpi R$, and so $i:R\to \End_R(M)$ cannot be an isomorphism.
	
	Thus if we assume that $i$ is an isomorphism, it follows that $\card{\cng{\lambda}(R) }= \card{\cng{\lambda}(M)}$.
\end{proof}

\begin{remark}
It is worth pointing out that the converse to the second claim in Theorem \ref{thm:surjective} is not true. That is, it is possible that $\rank_\lambda M = 1$ and $\card{\cng{\lambda}(R)} = \card{\cng{\lambda}(M)}$ for some $\lambda$ (and in fact, for \emph{every} $\lambda$) but $R\to \End_R(M)$ fails to be an isomorphism.

To see this, let $\widetilde{R}$ be any reduced, local $\varpi$-torsion free $\cO$-algebra, and let $\widetilde{\lambda}_1:\widetilde{R}\to\cO^{(1)},\ldots, \widetilde{\lambda}_r:\widetilde{R}\to \cO^{(r)}$ be all of the augmentations of $\widetilde{R}$ (where each $\cO^{(i)}$ is a finite extension of $\cO$). For the sake of simplicity assume that $\cO^{(1)} = \cdots = \cO^{(r)}  = \cO$ (and so $\widetilde{R}[1/\varpi] \cong E^r$). For each $i$, write $\widetilde{R}[\ker \widetilde{\lambda}_i] = \cO\alpha_i$.

Now let $R\subseteq \widetilde{R}$ be the $\cO$-subalgebra of $\widetilde{R}$ generated by $\alpha_1,\ldots,\alpha_r$. It is not hard to see that $R$ is also a reduced, local $\varpi$-torsion free $\cO$-algebra, that $R[1/\varpi]= \widetilde{R}[1/\varpi]\cong E^r$ and that the augmentations of $R$ are all of the form $\lambda_i:=\widetilde{\lambda}_i|_{R}:R\to\cO$. But one can now easily check that $R[\ker \lambda_i] = \cO\alpha_i = \widetilde{R}[\ker \widetilde{\lambda}_i] = \widetilde{R}[\ker \lambda_i]$ for all $i$. So treating $\widetilde{R}$ as an $R$-module we now see that $\rank_{\lambda_i}\widetilde{R} = 1$ and $\cng{\lambda_i}(R) = \cO/(\lambda_i(\alpha_i)) = \cng{\widetilde{\lambda}_i}(\widetilde{R})$ for all $i$. On the other hand we clearly have $\widetilde{R}\subseteq \End_R(\widetilde{R})$ so provided that $R\ne \widetilde{R}$ (that is, that $\widetilde{R}$ is not generated as an $\cO$-algebra by $\alpha_1,\ldots,\alpha_r$) the map $R\to \End_R(\widetilde{R})$ cannot be an isomorphism. 

It is easy to find examples of rings $\widetilde{R}$ for which $R\ne \widetilde{R}$, so in general the hypothesis that $R\to \End_R(M)$ is an isomorphism is strictly stronger than the hypothesis that $\card{\cng{\lambda}(R)} = \card{\cng{\lambda}(M)}$ for some, or even all, $\lambda$.
\end{remark}

\section{Galois deformation theory}\label{sec_deformation_theory}

\subsection{Generalities}

We  use section 5 of  \cite{Thorne}  as reference for the  material recalled in this subsection.

Let $p$ be an odd prime, and let $E$ be a coefficient field with valuation ring $\cO$ and residue field $k$. We fix a continuous, absolutely irreducible representation $\overline{\rho} : G_{\Q} \to \GL_2(k)$ and assume that $\det \overline{\rho}=\epsilon$, the cyclotomic character. We will assume that $k$ contains the eigenvalues of all elements in the image of $\overline{\rho}$. We also fix a finite set $S$ of primes containing $p$, and all primes at which $\overline{\rho}$ is ramified (we also allow $S$ to contain finitely many other primes in addition to these).

Let $q \in S$. We write $\cD_q^\square : \CNL_{\cO} \to \mathrm{Sets}$ for the functor that associates to $R \in \mathrm{CNL}_{\cO}$ the set of all continuous homomorphisms $r : G_{\Q_q} \to \GL_2(R)$ such that $r \mod \ffrm_R = \overline{\rho}|_{G_{\Q_q}}$ and $\det r$ is the cyclotomic character $\epsilon$. It is easy to see that $\cD_q^\square$ is represented by an object $R_q^{\square} \in \CNL_{\cO}$.
\begin{definition}\label{def_gen}
Let $q \in S$. A local deformation problem for $\overline{\rho}|_{G_{\Q_q}}$ is a subfunctor $\cD_q \subset \cD_q^\square$ satisfying the following conditions:
\begin{itemize}
\item $\cD_q$ is represented by a quotient $R_q$ of $R_q^\square$.
\item For all $R \in \CNL_{\cO}$, $a \in \ker(\GL_2(R) \to \GL_2(k))$ and $r \in \cD_q(R)$, we have $a r a^{-1} \in \cD_q(R)$.
\end{itemize}
\end{definition}
We will write $\rho_q^\square : G_{\Q_q} \to \GL_2(R_q^\square)$ for the universal lifting. If a quotient $R_q$ of $R_q^\square$ corresponding to a local deformation problem $\cD_q$ has been fixed, we will write $\rho_q : G_{\Q_q} \to \GL_2(R_q)$ for the universal lifting of type $\cD_q$.

We recall \cite[Lemma~5.12]{Thorne}.

\begin{lemma}\label{lem_local}
Let $R_q \in \mathrm{CNL}_{\cO}$ be a quotient of $R_q^\square$ satisfying the following conditions:
\begin{enumerate}
\item The ring $R_q$ is reduced, and not isomorphic to $k$.
\item Let $r : G_{\Q_q} \to \GL_2(R_q)$ denote the specialization of the universal lifting, and let $a \in \ker(\GL_2(R_q) \to \GL_2(k))$. Then the homomorphism $R_q^\square \to R_q$ associated to the representation $a r a^{-1}$ by universality factors through the canonical projection $R_q^\square \to R_q$.
\end{enumerate}
Then the subfunctor of $\cD_q^\square$ defined by $R_q$ is a local deformation problem.
\end{lemma}

We now specialize to the situation of this paper.
Consider our fixed $\rhobar$, and the integers $N$ and $Q$, with $N|N(\rhobar)|NQ$, that we have fixed. (Recall that all the primes in $Q$ are Steinberg for $\rhobar$.)  Let $S$ be the set of primes which divide $N(\rhobar)pQ$. We briefly recall the local and global deformation rings that we consider in the paper, and some of their known properties. For convenience, we will also assume that the coefficient ring $E$ contains all $(q^2-1)$th roots of unity, for each $q|Q$ (recall that as shown in the previous section, the Wiles defect is unaffected by expanding the coefficient ring, so this assumption is harmless).

\subsection{Local deformation rings}\label{Subsec:LocDefRings}

For any prime $q$, let $R_q^\square\in\CNLO$ be the full framed local deformation ring of $\rhobar|_{G_q}$  for lifts of fixed determinant $\epsilon$.  We consider  below quotients of $R_q^\square$
which are defined as in \cite{Kisin} as reduced, flat over $\cO$ quotients of $R_q^\square$, and are  characterised by the $\overline \Q_p$-valued points of their generic fiber; \cite{Kisin} computes  the  dimension of generic fibers of the quotients we consider, and  proves that they are regular.  The quotients we consider satisfy Lemma \ref{lem_local} and thus  give rise to deformation conditions in the sense of Definition \ref{def_gen}. We need finer integral properties of the quotients we consider that are stated in Proposition \ref{prop:R_v-Atq}.

\medskip

\noindent For $\ell|N$,
\begin{itemize}

\item  let $R_\ell^{\min}$ be the minimally ramified  quotient of $R_\ell$ (which parametrises minimally ramified deformations)  and let $R_p^{\fl}$ be the flat quotient of $R_p$ (which parametrises flat deformations). 
\end{itemize}

\noindent For $q|Q$, the Steinberg quotient  $R_q^\St$ of $R_q^\square$  is defined as follows:

\begin{itemize}

\item If $q$ is ramified in $\rhobar$, then the Steinberg quotient $R_q^\St$ is defined as the minimally ramified  quotient~of~$R_q^\square$.

\item If $q$ is unramified in $\rhobar$ and $q\not\equiv -1\pmod p$, then  the Steinberg quotient $R_q^\St$  is the  unique reduced quotient of $R_q^\square$ characterised by the fact that the $L$-valued points of its generic fiber, for any finite extension $L/E$, correspond to representations which up to an unramified twist are  of the form $\begin{pmatrix} \epsilon &*\\0& 1\end{pmatrix}$.

\item If $q$ is unramified in $\rhobar$ and $q\equiv -1\pmod p$, then $\overline{\rho}(\Frob_q)$ takes two distinct eigenvalues $\alpha_q, \beta_q \in k$ such that $\alpha_q / \beta_q = q$ (note that $\alpha_q,\beta_q=\pm 1$).  
 Thus $\rhobar|_{G_q}$ in some basis $e_1,e_2$ is of the form 
  \[  \left( \begin{array}{cc} \varepsilon \chibar&  0\\ 0 & \chibar \end{array} \right)\]  where  we assume that $\chibar(\Frob_q)=\beta_q$. We define $R_q^{\St(\beta_q)}$ as the unique reduced quotient of $R_q^\square$ characterised by the fact that the $L$-valued points of its generic fiber, for any finite extension $L/E$, correspond to representations of the form 
 \[ \left( \begin{array}{cc} \varepsilon \chi& \ast \\ 0 & \chi\end{array} \right)\]
such that $\chi$ reduces to $\chibar$.

Below we will fix a $\beta_q$, and, for uniformity of notation, generally we write $R_q^\St$ for $R_q^{\St(\beta_q)}$.

\end{itemize}

\noindent For $q|Q$ unramified in $\rhobar$, the unipotent quotient  $R_q^\uni$ of $R_q^\square$ is defined as
\begin{itemize}
\item 
the unique reduced quotient of $R_q^\square$ such that $\Spec R_q^\uni=\Spec R_q^\St\cup\Spec R_q^{\min}$ inside $\Spec R_q^\square$, where here $R_q^{\min}$ is the quotient of $R_q^\square$ that parametrized unramified deformations.

If $q\equiv -1\pmod p$, we write $R_q^{\uni(\beta_q)}$ for $R_q^\uni$ if we wish to indicate the $\beta_q$ used to define $R_q^\st$.
\end{itemize}

\medskip

Recall from \cite[\S~3]{BLGHT2011} that $R\in\CNL_\cO$ is called {\em geometrically integral}, if for all finite field extensions $E'\supset E$ with integral closure $\cO'$ of $\cO$, the ring $R\otimes_\cO\cO'$ is an integral domain. 

The following result summarizes basic ring theoretic properties of the $R_q^?$.
\begin{proposition}\label{prop:R_v-Atq}
The following hold:
\begin{enumerate}
\item We have $R^{\fl}_q \cong \cO[[X_1,X_2,X_3,X_4]]$ for $q=p$ and $R^{\min}_q \cong \cO[[X_1,X_2,X_3]]$ for all $q\neq p$.
\item For $q\ne p$, the ring $R_q^\square$ is a complete intersection, reduced, and flat over $\cO$ of relative dimension $3$. 
\item For $q|Q$, the ring $R_q^{\St}$ is Cohen--Macaulay, flat of relative dimension $3$ over $\cO$ and geometrically integral, and if $q$ is not a trivial prime for $\rhobar$, we in fact have $R^{\St}_q\cong \cO[[X_1,X_2,X_3]]$.
\item For each $q|Q$ and each minimal prime $\frp$ of $R_q^\square$, $R_q^\square/\frp$ is flat over $\cO$ and geometrically integral.
\item For $q|Q$ such that $\rhobar$ is unramified,  the reduced ring $R_q^\uni$ is Cohen-Macaulay and flat over $\cO$ of relative dimension $3$.
\end{enumerate}
Moreover the rings $R_q^?$ in 1.--5. are the completion of a finite type $\cO$-algebra at a maximal ideal. 
\end{proposition}
\begin{proof}
Assertion 1 is from \cite[Sections~2.4.1 and~2.4.4]{CHT} and Assertion 2 from \cite[Theorem~2.5]{Shotton2016}, except that one has to modify the statements and proofs there by including the further condition that the determinant be $\epsilon$, which in turn results in a drop of dimensions by 1. Assertion 3 is from \cite[Section 5]{Shotton}, where here one has to impose the additional condition that the image of Frobenius has determinant $1$ and one has to modify the proofs accordingly. That $R_q^\St$ is geometrically integral is not stated explicitly in \cite{Shotton}. However the base changes $R_q^\St\otimes_\cO\cO'$ have again an interpretation as a universal Steinberg ring and are thus all integral domains. Assertion 4 follows similarly from the results in \cite[Section 5]{Shotton}.

Assertion 5 will be addressed in the proof of Proposition~\ref{key-local-comp}, where we give a description of $R_q^\uni$ as a quotient of a power series ring over $\cO$ by some explicit relations. In that proof we also recall more detailed information about $R_q^\St$ and $R_q^\square$; this would easily provide a direct proof of 2.~except for the reducedness.

In all cases it is obvious from the explicit nature of the rings $R_q^?$ given in the references that the rings are the completion of a ring of finite type over $\cO$ at a maximal ideal. 
\end{proof}

Define $R_q = R_q^\square$ for $q\in Q$, and further:
\begin{align*}
R_{\loc} &= \left[\widehat{\bigotimes_{q|Q}}R_q\right]\widehat{\otimes}\left[\widehat{\bigotimes_{\ell|N}}R^{\min}_\ell\right]\widehat{\otimes}R^{\fl}_p&
&\text{and,}&
R_{\loc}^{\St} &= \left[\widehat{\bigotimes_{q|Q}}R^{\St}_q\right]\widehat{\otimes}\left[\widehat{\bigotimes_{\ell|N}}R^{\min}_\ell\right]\widehat{\otimes}R^{\fl}_p.
\end{align*}

\begin{lemma}\label{Lemma-CompletedTensor}
Suppose $\cR_i$, $i=1,2$, lie in $\CNL_\cO$ and are flat over $\cO$. Let $\cR=\cR_1\widehat\otimes_\cO \cR_2$ be their completed tensor product. Then the following hold:
\begin{enumerate}
\item The ring $\cR$ is flat over $\cO$.
\item If both $\cR_i$ are Cohen-Macaulay, then so is $\cR$.
\item If both $\cR_i$ are Gorenstein, then so is $\cR$.
\item If both $\cR_i$ are local complete intersection rings, then so is $\cR$.
\item If both $\cR_i$ are reduced and if $\cR_2$ is the completion of a finite type $\cO$-algebra $R_2$ at a maximal ideal $\ffrm_2$, then $\cR$ is reduced. 
\item If both $\cR_i$ are geometrically integral, then so is $\cR$.
\item If the minimal primes of $\cR_1$ are $\frp_1,\ldots,\frp_r$ and the minimal primes of $\cR_2$ are $\frq_1,\ldots,\frq_s$ and each $\cR_1/\frp_i$ and $\cR_2/\frq_j$ is flat over $\cO$ and geometrically integral, then the distinct minimal primes of $\cR$ are $\frp_i\cR+\frq_j\cR$ for $i=1,\ldots,r$ and $j=1,\ldots,s$.
\item If both $\cR_i$ are the completion of a finite type $\cO$-algebra $R_i$ at a maximal ideal $\ffrm_i$, then so is $\cR$.
\end{enumerate}
\end{lemma}
\begin{proof}
Part 1 is proved in the first paragraph of \cite[Proof of 3.4.12]{Kisin}. 

For part~2 recall from \cite[\href{https://stacks.math.columbia.edu/tag/045J}{Lemma 045J}]{stacks-project}, that for a flat local homomorphism $R\to S$ of local Noetherian rings, the ring $S$ is Cohen-Macaulay if and only if $R$ and $S/\ffrm_R S$ are Cohen-Macaulay. We deduce on the one hand that the rings $\cR_i/\pi\cR_i$ are Cohen-Macaulay for $i=1,2$ and on the other that to show that $\cR$ is Cohen-Macaulay, we need to establish that $\cR/\pi\cR=(\cR_1/\pi\cR_1)\widehat\otimes_k (\cR_2/\pi\cR_2)$ is Cohen-Macaulay. 

Now by the Cohen structure theorem in the form given in \cite[\href{https://stacks.math.columbia.edu/tag/032D}{Lemma 032D}]{stacks-project}, there exist power series rings $\cS_i\subset\cR_i/\pi\cR_i$ over $k$ such that $\cR_i/\pi\cR_i$ is finite over $\cS_i$, $i=1,2$.\footnote{That statement of \cite[\href{https://stacks.math.columbia.edu/tag/032D}{Lemma 032D}]{stacks-project} assumes the ring to be a domain. However its proof (Case I, with $\Lambda=k$) does not make use of this hypothesis.} Since $\cR_i/\pi\cR_i$ is Cohen-Macaulay, it follows from \cite[\href{https://stacks.math.columbia.edu/tag/00R4}{Lemma 00R4}]{stacks-project}, that $\cR_i/\pi\cR_i$ is free of finite rank over $\cS_i$. Hence $\cR/\pi\cR$ is a ring extension of the power series ring $\cS=\cS_1\widehat\otimes_k \cS_2$ that is free of finite rank over $\cS$. The ring $\cS$ is regular and in particular Cohen-Macaulay, and all fibers of $\cS\to\cR/\pi\cR$ are Artin rings and thus also Cohen-Macaulay. Invoking again \cite[\href{https://stacks.math.columbia.edu/tag/045J}{Lemma 045J}]{stacks-project}, we see that $\cR/\pi\cR$ is Cohen-Macaulay. Let us mention that our proof of part 2 is inspired by \cite[Appendix~A]{Caraiani+5}.

The first half of the proof of part 2 goes through with Gorenstein or local complete intersection instead of Cohen-Macaulay, as well, by appealing to \cite[\href{https://stacks.math.columbia.edu/tag/0BJL}{Proposition 0BJL}]{stacks-project} and \cite[\href{https://stacks.math.columbia.edu/tag/09Q2}{Proposition 09Q2}]{stacks-project}. The remaining argument for part~3 is \cite[Theorem~6]{WITO}. There is also a simple direct proof. One can combine systems of parameters for $\cR_i/\pi\cR_i$, $i=1,2$, to obtain a system of parameters for $(\cR_1/\pi\cR_1)\widehat\otimes_k (\cR_2/\pi\cR_2)$, and then deduce that the socle of the resulting quotient has $k$-dimension $1$ using that this holds for the quotients of the individual factors. Even more elementary is the completion of the argument for part~4. We write $\cR_i/\pi\cR_i$ as a quotient of a power series ring $\cS_i$ over $k$ by an ideal $\cI_i=(f_{i,1},\ldots, f_{i,h_i})$ such that $h_i=\dim \cS_i-\dim \cR_i/\pi\cR_i$. Then $\cR/\pi\cR=\cS/\cI$ for $\cS$ the power series ring $\cS_1\widehat\otimes_k\cS_2$ over $k$ and $\cI=(\{f_{1,j}\otimes 1:j\in\{1,\ldots,h_1\}\}\cup\{ 1\otimes f_{2,j}:j\in\{1,\ldots,h_2\})$, so that $\dim \cS-\dim \cR/\pi\cR=h_1+h_2$.

To establish part 5, observe first that the $\cR_i$ as well as $R_2$ and $\cR_1\otimes R_2$ are Nagata rings by \cite[\href{https://stacks.math.columbia.edu/tag/0335}{Proposition 0335}]{stacks-project}. By \cite[\href{https://stacks.math.columbia.edu/tag/07NZ}{Lemma 07NZ}]{stacks-project}, the completion of a Nagata ring is reduced if and only if the ring itself is reduced. From this and our hypothesis, we see that $\cR_1$ and $R_2$ are reduced and it suffices to show that $\cR_1\otimes_\cO R_2$ is reduced. Now $\cR_1$ and $R_2$ each embed into a finite product of fields $\cK_j$ and $K_{j'}$. These fields must contain $E$, and it follows that $\cR_1\otimes_\cO R_2$ embeds into the finite product of the rings $\cK_j\otimes_E K_{j'}$. The latter are reduced because $E$ is of characteristic $0$ and thus $K_{j'}$ is separable over $E$; see \cite[\href{https://stacks.math.columbia.edu/tag/030U}{Lemma 030U}]{stacks-project}. Therefore $\cR_1\otimes_\cO R_2$ has no non-zero nilpotent elements and is hence reduced.

Parts~6 and 7 are from \cite[Lemma 3.3]{BLGHT2011}, and part 8 is obvious.
\end{proof}

Proposition~\ref{prop:R_v-Atq} and  with Lemma~\ref{Lemma-CompletedTensor} yield:
\begin{proposition}\label{prop:R_v-Loc}
The ring $R_{\loc}$ is a reduced complete intersection, the ring $R^{\St}_{\loc}$ is a Cohen--Macaulay domain, and both are flat over $\cO$.

Moreover, each irreducible component of $\Spec R_{\loc}$ is of  the form 
\[\Spec \left[\widehat{\bigotimes_{q|Q}}R_q^\square/\frp^{(q)}\right]\widehat{\otimes}\left[\widehat{\bigotimes_{\ell|N}}R^{\min}_\ell\right]\widehat{\otimes}R^{\fl}_p\]
where each $\Spec R_q/\frp^{(q)}$ is an irreducible component of $\Spec R_q^\square$ (that is each $\frp^{(q)}$ is a minimal prime of~$R_q$).
\end{proposition}

\subsection{Global deformation rings}\label{ssec:global def}

Now we set up the notation for the corresponding global deformation rings that we use later.

Let $R^{\square}$  (resp. $R^{\square, \St}$) be  the  framed global deformation rings that parametrise strict equivalence classes of  tuples  $(\rho, \{ t_q \}_{q \in S})$, where:

\begin{itemize}

\item  $\rho : G_\Q \to \GL_2(R)$ is a lifting  of $\rhobar$, with $R\in\CNL_\cO$, of determinant $\epsilon$

 \item $\rho$ is  unramified outside $S$, of determinant $\epsilon$,   and for each $\ell \in S$, $\rho|_{G_\ell}$ satisfies the local deformation condition arising from:
\begin{itemize}
	\item $R_\ell^{\rm min}$ (for $\ell | N$), $R_p^{\fl}$ (for $\ell=p$),  $R_\ell^\square$  for $\ell \in Q$
	\item respectively, $R_\ell^{\rm min}$ (for $\ell | N$), $R_p^{\fl}$ (for $\ell=p$),  $R_\ell^\St$ (for $\ell \in Q$)
\end{itemize}
 
 \item  $t_q$ is an element of $\ker(\GL_2(R) \to \GL_2(k))$.

  \end{itemize}
  
  Here two framed liftings $(\rho_1, \{ t_q \}_{q \in S})$ and $(\rho_2, \{ u_q \}_{q \in S})$ are said to be strictly equivalent if there is an element $a \in \ker(\GL_2(R) \to \GL_2(k))$ such that $\rho_2 = a \rho_1 a^{-1}$ and $u_q = a t_q$ for each $q \in S$. 
  
  As $\rhobar$ is irreducible, we also have the corresponding unframed  rings $R$ and $R^\St$,  which are canonically $\cO$-subalgebras of   $R^{\square}$ and $R^{\square, \St}$, and  furthermore (non-canonically) 
$R^{\square}=R[[X_1,\cdots X_{4j-1}]]$ and $R^{\square, \St}=R^\St[[X_1,\cdots X_{4j-1}]]$ with $j$ the cardinality of $S$. (We will   also by harmless abuse of notation sometimes  view  $R$ and $R^\St$ as (non-canonical)  quotients of  the  corresponding framed  deformation rings by the framing variables $X_i$.)

The natural transformation $(\rho, \{ t_q \}_{q \in S}) \mapsto ( t_q^{-1} \rho|_{G_{\Q_q}} t_q)_{q \in S}$ induces canonical maps  $R_{\loc} \to R^{\square}$  and  $R^{\St}_{\loc} \to R^{\square, \St} $  which are homomorphisms in $\CNL_\cO$, and we have 
$R^{\square, \St}=  R^{\square} \otimes_{R_{\loc}} R^{\St}_{\loc}$.

\section{Hecke algebras and newforms}\label{hecke}

  \subsection{Modular curves and Shimura curves}\label{ssec:modular curve}
  
  From now on, let $D$ be a quaternion algebra over $\Q$ (either definite or indefinite) and let $Q$ be the set of finite primes at which $D$ is ramified.  (By abuse of notation, we will also frequently use $Q$ to denote the product of all the primes in the set $Q$. The context will make clear which meaning is intended.) As all of our results are trivial in the case when $Q = \es$, we will assume that $Q\ne \es$, for convenience. Fix a square-free  integer $N$ not divisible by any prime in $Q$. Let $\Gamma_0(N\Pi_{q \in Q} q^2) $ and $\Gamma_0^Q(N) $ be congruence subgroups for $\GL_2$ and $D^\times$, respectively, where   $\Gamma_0^Q(N)$ is maximal compact at primes in $Q$, and upper triangular mod $\ell$ for all $\ell|N$. We will  denote  $\Gamma_0(N\Pi_{q \in Q} q^2)$ by $\Gamma_0(NQ^2)$ using our abuse of notation.
  
  We consider  a prime $p$  below not dividing $2NQ$ and a finite extension $E/\Q_p$, and let $\cO$ be the ring of integers in $E$, $\varpi$ a uniformizer and $k=\cO/\varpi$ the residue field. We will assume below that $E$ is sufficiently large  so that $\cO$ contains the Fourier coefficients of all newforms in $S_2(\Gamma_0(NQ))$. 

  Let $K_0(NQ^2)\subseteq \GL_2(\A_{\Q,f})$ and $K_0^Q(N)\subseteq D^\times(\A_{\Q,f})$ be the corresponding compact open subgroups.
  If $D$ is definite, let
\[S^Q(\Gamma_0^Q(N)) = \left\{f: D^\times(\Q)\backslash D^\times(\A_{\Q,f})/K_0^Q(N) \to \cO\right\}.\]
 If $D$ is indefinite, let $X^Q_0(N)$ be the (compact) Riemann surface $D^\times(\Q)\backslash \left(D^\times(\A_{\Q,f})\times \half\right)/K_0^Q(N)$ (where $\half$ is the complex upper half plane). Give $X^Q_0(N)$ its canonical structure as an algebraic curve over $\Q$, 
and let    $S^Q(\Gamma_0^Q(N))=H^1(X^Q_0(N),\cO)$.
We also consider $S(NQ^2)=H^1(X_0(NQ^2),\cO)$.  Let $\T(NQ^2)$ and $\T^Q(N)$ be the  $\cO$-algebras at level $\Gamma_0(NQ^2)$ and $\Gamma_0^Q(N)$, respectively, generated by  the Hecke operators $T_r$ for  primes $r$ coprime to $pNQ$ acting on  $H^1(X_0(NQ^2),\cO)$ and $S^Q(\Gamma_0^Q(N))$.
Note that by the Jacquet-Langlands correspondence, $\T^Q(N)$ is a quotient of $\T(NQ^2)$.

Let $f \in S_2(\Gamma_0(NQ))$  be a newform of level $NQ$ such that all its Fourier coefficients lie in $E$, and consider the corresponding $\cO$-algebra homomorphism $\lambda_f:\T(NQ^2) \to \cO$.   We will fix this newform and our main results will be in relation to $f$. By the Jacquet-Langlands correspondence, this also gives a related homomorphism $\T^Q(N) \to \cO$ that we denote by the same symbol $\lambda_f$.    We denote the corresponding maximal ideals which contain the prime ideal $\ker(\lambda_f)$ by   $\ffrm$. Let $\rho_f : G_\Q \to \GL_2(\cO)$ be the Galois representation associated by Eichler and Shimura to $f$ and assume   that the corresponding residual  Galois representation  $\rhobar_f=\rhobar :G_\Q\to \GL_2(k)$ is  absolutely irreducible. By enlarging $\cO$ if necessary,  we may assume that $k$ contains all eigenvalues of $\rhobar(\sigma)$ for all $\sigma\in G_\Q$. We assume that $N|N(\rhobar)|NQ$. The Galois representation $\rho_f:G_\Q \to \GL_2(E)$, with irreducible residual representation $\rhobar$,  is locally at $q$ of the form \[   \left( \begin{array}{cc} \epsilon& \ast \\ 0 & 1 \end{array} \right),\]   up to twist by an unramified character $\chi$ of order dividing 2.  The $\beta_q \in \{\pm 1\}$  of the previous section  will be  chosen so that  $\rho_f|_{G_q}$ gives rise to a point of  $\Spec   R_q^{\st(\beta_q)}$ in what follows (and thus depends on whether $\chi$ is trivial  at $q$ or not).

There is an oldform $f^Q$ in $S_2(\Gamma_0(NQ^2))$  with corresponding newform $f$ which is characterised by the property that it is an eigenform for the  Hecke operators  $T_\ell$ for $\ell$ prime with $(\ell,NQ^2)=1$ and $U_\ell $ for $\ell|NQ^2$ and  such that $a_q(f^Q)=0$, i.e., $f^Q|U_q=0$  for $q \in Q$. Let  $\lambda_{f^Q}:  \T^{\rm full}(NQ^2) \to \cO$   be the induced homomorphism of the full Hecke algebra $\T^{\rm full}(NQ^2)$   acting on $H^1(X_0(NQ^2),\cO)$ which is generated as an  $\cO$-algebra by the  action of the Hecke operators $T_\ell$ for $(\ell,NQ^2)=1$ and $U_\ell $ for $\ell|NQ^2$ on $S(NQ^2)=H^1(X_0(NQ^2),\cO)$. Let $\ffrm_Q$ be the  maximal ideal of $\T^{\rm full}(NQ^2)$ that contains $\ker(\lambda_{f^Q})$.  

By Ribet's level lowering result \cite{RibetInv100} there is a newform $g \in S_2(\Gamma_0(N(\rhobar)))$ which gives rise to $\rhobar$  and an augmentation $\lambda_g:\T(NQ^2) \to \cO$ whose kernel is  contained in the maximal ideal $\ffrm$. There is an oldform $g^Q$ in $S_2(\Gamma_0(NQ^2))$  with corresponding newform $g$ which is characterised by the property that it is an eigenform for the  Hecke operators  $T_\ell$ for $(\ell,NQ^2)=1$ and $U_\ell $ for $\ell|NQ^2$ and  such that $a_q(g^Q)=0$, i.e., $g^Q|U_q=0$  for $q \in Q$. Let  $\lambda_{g^Q}:  \T^{\rm full}(NQ^2) \to \cO$   be the induced homomorphism of the full Hecke algebra $\T^{\rm full}(NQ^2)$   acting on $H^1(X_0(NQ^2),\cO)$ which is generated as an  $\cO$-algebra by the  action of the Hecke operators $T_\ell$ for $(\ell,NQ^2)=1$ and $U_\ell $ for $\ell|NQ^2$ on $S(NQ^2)=H^1(X_0(NQ^2),\cO)$. The kernel of $\lambda_{g^Q}$ is contained in the maximal ideal $\ffrm_Q$ that $f^Q$ gives rise to.  

The homomorphism $\lambda_f: \T^Q(N) \to \cO$ 
 extends to the full Hecke algebra  acting on  $S^Q(\Gamma_0^Q(N))$, and we denote  by $\ffrm_Q'$ the maximal ideal
 of $\T^Q(N)$ which contains its kernel.

We define $\T$ and $\T^{\st}$  to be the image of  $\T(NQ^2)$ and $\T^Q(N)$  in the endomorphisms of the finitely generated $\cO$-modules  $H^1(X_0(NQ^2),\cO)_{\ffrm_Q}$ and $S^Q(\Gamma_0^Q(N))_{\ffrm_Q'}$ respectively. We are primarily interested in $\T^{\st}$ in this paper, but we study it via looking at the surjective map of $\cO$-algebras $\T \onto \T^\St$ produced by the Jacquet-Langlands correspondence that yields a Hecke equivariant inclusion $S_2(\Gamma_0^Q(N)) \into S_2(\Gamma_0(NQ^2))$.

We have the corresponding universal modular deformation $\rho^{\mod}:G_\Q \to \GL_2(\T)$ by results of Carayol which is a specialiastion of the universal representation $G_\Q \to \GL_2(R)$.

Define
\[M(NQ^2)= \Hom_{\T[G_\Q]}(\rho^{\mod},S(NQ^2)_{\ffrm_Q}^*).\]
If $D$ is definite, define $M^{\st}(N) = S^Q(\Gamma_0^Q(N))_{\ffrm_Q'}^*$ and if $D$ is indefinite, define
\[M^{\st}(N)= \Hom_{\T[G_\Q]}(\rho^{\mod},S^Q(\Gamma_0^Q(N))_{\ffrm_Q'}^*).\]

As in \cite{Carayol2} we have   the following:

\begin{lemma}
  The evaluation map $M(NQ^2)\otimes_\T \rho^{\mod}\to S(NQ^2)_{\ffrm_Q}^*$ is an isomorphism, as is $M^{\st}(N)\otimes_\T \rho^{\mod}\to S^Q(\Gamma_0^Q(N))_{\ffrm_Q'}^*$ when $D$ is indefinite. In particular, as $\T$-modules we have $S(NQ^2)_{\ffrm_Q}^* = M(NQ^2)^{\oplus 2}$ and, when $D$ is indefinite,  $S^Q(\Gamma_0^Q(N))_{\ffrm_Q'}^* = M^{\st}(N)^{\oplus 2}$.
\end{lemma}

 We remark that, using  the notation of \S  \ref{sec:congruences}, below this implies that   for any $\lambda:\T \to \cO$, $\delta_{\lambda,\T}(S(NQ^2)_{\ffrm_Q}) = \delta_{\lambda,\T}(M(NQ^2))$ and  for any $\lambda:\T^\st \to \cO$, $\delta_{\lambda,\T^{\st}}(S^Q(\Gamma_0^Q(N))_{\ffrm_Q'}) = \delta_{\lambda,\T^{\st}}(M^{\st}(N))$.

 We have the following key theorem due to Wiles \cite{Wiles}  and   Diamond  \cite{DiamondMult1}.
 
 \begin{theorem}\label{fred}
 The surjective map $ R \onto \T$ of complete Noetherian local $\cO$-algebras  is an isomorphism of complete intersections. Furthermore  
 $S(NQ^2)_{\ffrm_Q}=H^1(X_0(NQ^2),\cO)_{\ffrm_Q}$ is free of rank 2 over $\T$, and $M(NQ^2)$ is free of rank one over $\T$.
 \end{theorem}

 \begin{proof}
 This follows from the arguments of  \cite[Theorem 3.4]{DiamondMult1}.  We sketch Diamond's argument. One first  proves an $R_\es=\T_\es$ theorem as in  \cite[Theorem 3.1]{DiamondMult1} for minimal deformation and Hecke rings for $\rhobar$ that come with an augmentation induced by a newform $g \in S_2(\Gamma_0(N(\rhobar))$ as above. One also  proves that  these rings are complete isomorphisms and the  appropriate localization of the cohomology of  a  modular curve $X_\es$, $H^1(X_\es,\cO)_\ffrm$, is free as a $\T_\es$ module.   Then one  observes that  $H^1(X_0(NQ^2),\cO)_{\ffrm_Q} \otimes E $ is free of rank 2 over $\T \otimes E$, which follows from the choice of $\ffrm_Q$ above  so as to contain the kernel of the augmentation  of the full Hecke algebra $\T^{\rm full}(NQ^2)$ induced by an oldform  $g^Q$  with  associated newform $g$  as above (see also  \cite[Proposition 4.7]{DDT}) .  Using level raising arguments  due to \cite{RibCong} , one then  deduces  that if $x,y$ is a basis of   the free $\cO$ module $H^1(X_0(NQ^2),\cO)_{\ffrm_Q}[\lambda_{g^Q}]=H^1(X_0(NQ^2),\cO)_{\ffrm_Q}[\lambda_{g}]$
 of rank 2, then $\# \ker(\lambda_R)/\ker(\lambda_R)^2\leq\#\cO/\langle x, y \rangle_Q$ where the pairing $\langle \  , \ \rangle_Q$ is a certain alternating $\T$-bilinear, $\cO$-perfect pairing on  $H^1(X_0(NQ^2),\cO)_{\ffrm_Q} $. By appealing to \cite[Theorem 2.4]{DiamondMult1} one deduces that  $ R \onto \T$   is an isomorphism of complete intersections and 
 $S(NQ^2)_{\ffrm_Q}=H^1(X_0(NQ^2),\cO)_{\ffrm_Q}$ is free of rank 2 over $\T$. In particular $\T$ is Gorenstein.  As $S(NQ^2)_{\ffrm_Q}^*=M(NQ^2)^{\oplus 2}$ we deduce from this that $M(NQ^2)$ is projective, and as $\T$ is a local ring, that $M(NQ^2)$ is a free $\T$ module of rank one.
 \end{proof}

Note that classical generic multiplicity one results for modular forms imply that $M(NQ^2)$ and $M^{\st}(N)$ have rank $1$ at each generic point of   $\T^{\st}$. In contrast to what we have proved about $M(NQ^2)$,  when $Q$ contains a prime $q$ that is trivial for $\rhobar$, it is proved in \cite{Manning}  that $M^\st(N)$ is not  free of rank one over $\T^\st$, nor is $\T^\st$ a complete intersection. 

 \begin{remark}
  In \cite{DiamondMult1}, a slightly less restricted deformation functor  is considered than the one represented by $R$ above (for instance at primes dividing $N(\rhobar)$ he  puts no restriction on ramification while we have imposed minimality at primes dividing $N$). Whilst we could also work
  with his conditions, we have chosen to work with our slightly more restrictive deformation conditions at primes dividing $N$, and the arguments
  of loc. cit. we have quoted apply  {\it mutatis mutandis}. 
  \end{remark}
  
  \begin{remark}
  In \cite{DiamondMult1} the analog of the rings $R$ and $\T$ above are denoted $R_\Sigma$ and $\T_\Sigma$ (where $\Sigma$ is the set $Q$ here), the maximal ideal $\ffrm_Q$ is denoted by  $\ffrm_\Sigma$, and  the analog of  the modular curve $X_0(NQ^2)$ we consider here is denoted  by $X_\Sigma$. We remark that  as in loc. cit. although $\T$  is an ``anemic Hecke algebra'',  defined without including  the operators $U_q$ for $q \in Q$, the maximal ideal $\ffrm_Q$ is a  suitable maximal ideal of the full Hecke algebra with $U_q \in \ffrm_Q$.  This  ensures that  $H^1(X_0(NQ^2),\cO)_{\ffrm_Q} \otimes E $ is free of rank 2 over $\T \otimes E$ which is a pre-requisite  to apply   \cite[Theorem 2.4]{DiamondMult1} and   get the refined integral statement that  $H^1(X_0(NQ^2),\cO)_{\ffrm_Q}$ is free of rank 2 over $\T$. This  way of avoiding use of Hecke operators $U_q$ for $q|Q$, when proving $R \onto \T$ is an isomorphism  with non-minimal conditions at $Q$,  goes back to the original arguments of Wiles in  \cite[Chapter 3]{Wiles} that prove modularity results in the non-minimal case. Wiles uses the numerical criterion   Theorem \ref{numerical criterion},
  modularity results in the minimal case and level raising.  Diamond further enhances the argument in \cite{DiamondMult1} to get that $H^1(X_0(NQ^2),\cO)_{\ffrm_Q}$ is free of rank 2 over $\T$, without having to  know {\it a priori }  Gorenstein properties of Hecke algebras.
  
  It is an interesting feature of the proof of the main result of this paper,  Theorem \ref{mc} below, that to prove results  about the Wiles defect of the weight 2  newform $f$ of level $NQ$ that we are interested in, we have to use another newform of  weight 2 and minimal level $N(\rhobar)$ via its use in the proof of Theorem \ref{fred}.
 \end{remark}

\subsection{Abelian varieties of $\GL_2$ type with multiplicative reduction at $q$}
  
  We note  a useful result about Tate modules of abelian varieties $A$ over $\Q$ of $\GL_2$-type, with field of endomorphisms over $\Q$ a number field  $F$, and which have multiplicative reduction at a prime $q$. Let $p$ be a  prime, and consider $F \otimes \Q_p= \prod_v F_v$,  which induces a factorization ${\rm Ta}_p(A) \otimes \Q_p= \prod_{v} {\rm Ta}_v(A)$ and ${\rm Ta}_v(A)=F_v^2$, and we denote by $\rho_{A,v}:G_\Q \ra \GL_2(F_v)$ the representation of $G_\Q$ arising from its action  on ${\rm Ta}_v(A)$.   If $\cO_v$ as usual is  the valuation ring of $F_v$, $\rho_{A,v}$ has an integral model with values in $\GL_2(\cO_v)$.  We assume that the determinant of $\rho_{A,v}$ is  the $p$-adic  cyclotomic character $\varepsilon$ which is assured in our applications as $A=A_f$  arise from a newform $f$ in $S_2(\Gamma_0(M))$ for some positive integer $M$.

\begin{lemma}\label{uniformization}
\begin{enumerate}
\item   If $\rho_{A,v}$ is  residually irreducible, then $\rho_{A,v}$ is integrally well-defined, i.e.,   $G_\Q$  preserves a unique lattice of $E_v^2$ upto scaling.

 \item   Furthermore  the restriction $\rho_{A, v}|_{G_q}$ of an integral model of $\rho_{A,v}$ is 
 $\GL_2(\cO_v)$-conjugate to   an unramified twist of  
  \[  \left( \begin{array}{cc} \varepsilon & \ast \\ 0 & 1\end{array} \right). \]   The reducible representation $\rho_{A, v}|_{G_q}$ when regarded as valued in $\GL_2(E)$   is ramified and non-split.
  \end{enumerate}
\end{lemma}

\begin{proof}
The first part is  a well-known result of  Carayol \cite{Carayol2}. The second part  is a consequence of the Mumford-Raynaud-Tate  $q$-adic uniformization of the abelian variety $A$ over $\Q_q$ which has multiplicative reduction at $q$ 
and weight-monodromy results.
Namely,  consider the character
group ${\cal X}(A,q)$  (with values in $\Z_p$) of
the torus of the reduction of the N\'eron model of the abelian
variety $A$
at the prime $q$. We denote by  ${\cal X}(A,q)_{\cO_v}={\cal X}(A_q) \otimes \cO_v$ where the tensor product is over ${\rm End}(A)  \otimes \Z_p$. Note that as $A$ has purely
toric reduction at $q$, we have the exact sequence of $\cO_v[G_q]$-modules
$$ 0 \rightarrow {\rm Hom}_{\cO_v}(({\cal X}(A,q)_{\cO_v},\cO_v(1))
\rightarrow {\rm Ta}_p(A)_{\cO_v} \rightarrow {\cal
X}(A^d,q)_{\cO_v}  \rightarrow 0.$$
Note that ${\rm Hom}_{\cO_v}(({\cal X}(A,q)_{\cO_v},\cO_v(1)),  {\cal
X}(A^d,q)_{\cO_v}  $ are both free $\cO_v$-modules of rank 1 (as they are both torsion-free and after inverting $p$ are both isomorphic to $F_v$), where $A^d$ denotes the dual abelian variety of $A$. The fact that the representation is ramified follows
from the $q$-adic   uniformization.
\end{proof}

We will use this when $A=A_f$ with $A_f$ the abelian variety associated by Shimura to a newform $f \in S_2(\Gamma_0(M))$ that is Steinberg at a place $q$.
The shape of the Galois representation $\rho_f|_{G_q}$ helps to apply the computations of section \ref{sec:local-calc}
to Theorem \ref{mc1}.

\section{Patching}\label{sec:patch}

Let $R$ and $R^{\st}$ be the full (resp. Steinberg) unframed global deformation rings for $\rhobar$, as defined in section \ref{sec_deformation_theory}. Note that the representations $\rho^{\mod}:G_\Q\to \GL_2(\T)$ and $\rho^{Q,\mod} = \rho^{\mod}\otimes_\T\T^{\st}:G_\Q\to\GL_2(\T^{\st})$ induce surjective maps $R\to \T$ and $R^{\st}\to \T^{\st}$, which are compatible with quotient maps $R\to R^{\st}$ and $\T\to\T^{\st}$. 

As noted in Section \ref{ssec:global def}, we may (non-canonically) view $R$ and $R^{\st}$ as quotients of $R^\square$ and $R^{\square,\st}$ respectively, and moreover these quotient maps fit into a commutative diagram:

\begin{center}
	\begin{tikzpicture}
	\node(11) at (2,2) {$R_{\loc}$};
	\node(21) at (4,2) {$R^\square$};
	\node(31) at (6,2) {$R$};
	\node(41) at (8,2) {$\T$};
	
	\node(10) at (2,0) {$R_{\loc}^{\st}$};
	\node(20) at (4,0) {$R^{\square,\st}$};
	\node(30) at (6,0) {$R^\st$};
	\node(40) at (8,0) {$\T^{\st}$};	
	
	\draw[->>] (11)--(10);
	\draw[->>] (21)--(20);
	\draw[->>] (31)--(30);
	\draw[->>] (41)--(40);

	\draw[-latex] (11)--(21);
	\draw[->>] (21)--(31);
	\draw[->>] (31)--(41);
	
	\draw[-latex] (10)--(20);
	\draw[->>] (20)--(30);
	\draw[->>] (30)--(40);
	\end{tikzpicture}
\end{center}

Using these maps, we may view $M(NQ^2)$ and $M^{\st}(N)$ as $R_{\loc}$- and $R^{\st}_{\loc}$-modules, respectively.

The following results are fairly standard, we  only sketch the proofs:

\begin{lemma}\label{lem:R_loc smooth}
	Treating $M(NQ^2)$ as an $R_{\loc}$-module, $\Spec (R_{\loc}\otimes_{\cO} E)$ is formally smooth at each point in the support of $M(NQ^2)\otimes_{\cO}E$. In particular, each point in the support of $M(NQ^2)\otimes_{\cO}E$ lies in exactly one irreducible component of $\Spec R_{\loc}$.
\end{lemma}
\begin{proof}
To see this it is enough to observe that for a classical newform $g$, the  associated Galois representation $\rho_{g,\iota}
$ for $\iota: \overline \Q \to  \overline \Q_p$, is generic (see \cite[Lemma 1.1.5]{Allen}) at places $q \in Q$. This is a standard consequence of the weight-monodromy conjectures that are known in the case of Galois representations attached to classical newforms (and which yield that $H^2(G_q,\ad(\rho_{g,\iota}))=0$). In the special case of forms $g$ of weight 2 that are Steinberg at a prime $q$, the genericity of  $\rho_{g,\iota}|_{G_q}
$  follows from the fact that  the associated $\GL_2$-type abelian variety $A_g$ has multiplicative reduction at $q$  (see Lemma \ref{uniformization}(2)). This shows, for instance see \cite[Proposition 1.2.2]{Allen}, that $\rho_{g,\iota}
$  corresponds to a smooth point of  $\Spec R_q^\square$.
\end{proof}

\begin{lemma}\label{lem:M(NQ^2) support}
	Each irreducible component of $\Spec R_{\loc}$ contains at least one point in the support of $M(NQ^2)\otimes_{\cO}E$.
\end{lemma}
\begin{proof} 
The irreducible components of $\Spec R_{\loc}$ are given by Proposition \ref{prop:R_v-Loc}, combined with the description of the irreducible components of $\Spec R_q^\square$ given in \cite[Section 5]{Shotton}.  Thus  the components of  $\Spec R_q^\square$ for $q \in Q$ are labelled by inertial Weil-Deligne parameters compatible with the residual representation,  with the exception  that when  $q$ congruent to $-1$ mod $p$  and $\rhobar$ is unramified at $q$, there are two  Steinberg components $R_q^{\st(+1)}$ and $R_q^{\st(-1)}$,  the related  inertial Weil-Deligne parameter being $(\tau,N)=(0,N)$ with $N$ a  non-zero nilpotent matrix in $M_2(E)$, corresponding to the two eigenvalues $\pm 1$ of $\rhobar(\Frob_q)$

We claim that by  using the analysis in \cite{DT2}  (see \cite[Theorem1]{DT2}) of the local behavior of lifts of $\rhobar$  (at all primes away from $p$) arising from newforms  one  sees   that the support of $M(NQ^2)\otimes_\cO E$ contains a point in each of the components of $\Spec R_{\loc}$.    To  justify this claim,   in addition to the work of  loc. cit.,  we have to take  care of the additional components that arise in the case 
of $q$ congruent to $-1$ mod $p$.  For this we  need a slight refinement of \cite[Theorem 3]{DT2}. Namely we can find  a newform $g$ of weight 2 and  level $NQ$, that gives rise to $\rhobar$, that is in particular Steinberg at all places in $Q$,  as in loc. cit., with the additional flexibility    that  for each  $q \in Q$  that is  $-1$ mod $p$ and unramified in $\rhobar$,  we can  further specify the eigenvalue of  $U_q$ acting $g$ to be either $1$ or $-1$.   To prove this refinement, using the level raising methods of \cite{RibCong} and \cite{DT2}, we use the fact that if $h$ is a $q$-old eigenform in $S_2(\Gamma_0(NQ/q),\cO)$ that gives rise to $\rhobar$, then it has two stabilizations of level $NQ$ that are eigenforms for $U_q$ with eigenvalues that are respectively $1$ and $-1$ modulo the maximal ideal of $\cO$, and the reductions of both these forms are in the kernel of $U_q^2-1$.

\end{proof}

The Taylor--Wiles--Kisin method applied in our situation now gives the following:

\begin{theorem}\label{thm:patching}
There exist integers $g,d\ge 0$, rings
\begin{align*}
R_\infty &= R_{\loc}[[x_1,\ldots,x_g]]\\
R_\infty^{\st} &= R_{\loc}^{\st}[[x_1,\ldots,x_g]]\\
S_\infty &= \cO[[y_1,\ldots,y_d]]
\end{align*}
and modules $M_\infty$ and $M_\infty^{\st}$ over $R_\infty$ and $R_\infty^{\st}$, respectively, satisfying the following:
\begin{enumerate}
\item $\dim S_\infty = \dim R_\infty = \dim R_\infty^{\st}$.
\item There exists an $\cO$-module morphism $i:S_\infty\to R_\infty$ (inducing a morphism $i^{\st}:S_\infty\to R_\infty^{\st}$ via the quotient map $R_\infty\to R_\infty^{\st}$) making $M_\infty$ and $M_\infty^{\st}$ into finite free $S_\infty$-modules. From now on, we will view $R_\infty$ and $R_\infty^{\st}$ as $S_\infty$-algebras via this map $i$.
\item $M_\infty$ and $M_\infty^{\st}$ are maximal Cohen--Macaulay modules over $R_\infty$ and $R_\infty^{\st}$, respectively.
\item $R_\infty$ and $R_\infty^{\st}$ act faithfully on $M_\infty$ and $M_\infty^{\st}$, respectively.
\item We have isomorphisms $R_\infty\otimes_{S_\infty} \cO \cong R$, $R_\infty^{\st}\otimes_{S_\infty} \cO \cong R^{\st}$, $M_\infty\otimes_{S_\infty}\cO\cong M(NQ^2)$ and $M_\infty^{\st}\otimes_{S_\infty}\cO\cong M^{\st}(N)$.
\item The maps $i:S_\infty \to R_\infty$ and $i^{\st}: S_\infty \to R_\infty^{\st}$ are injective, and these maps make $R_\infty$ and $R_\infty^{\st}$ into finite free $S_\infty$-modules.
\item The surjective maps $R\to \T$ and $R^{\st}\to \T^{\st}$ are isomorphisms. In particular, $R$ and $R^{\st}$ are finite over $\cO$, reduced and $\varpi$-torsion free.
\end{enumerate}
\end{theorem}
\begin{proof}
As explained in section \ref{sec:trivial primes} the assumptions on $\rhobar$ (namely that it's absolutely irreducible and that $N(\rhobar)$ is squarefree) imply that $\rhobar|_{G_{\Q(\zeta_p)}}:G_{\Q(\zeta_p)}\to \GL_2(k)$, that is, that $\rhobar$ satisfies the ``Taylor--Wiles condition''.

First, as $\rhobar|_{\Q(\zeta_p)}$ is irreducible, we may appeal to Lemma 2 of \cite{DT1} to find an arbitrarily large prime $r$ so that no lift of $\rhobar$ is ramified at $r$. This allows us to impose $\Gamma_1(r^2)$ level structure without affecting any of the objects considered in this theorem, and so we may ignore any issues arising from isotropy (see \cite[Section 4.2]{Manning} or \cite[Section 6.2]{EGS} for details).

We can now apply the Taylor--Wiles--Kisin patching method applied to the modules $M(NQ^2)$ and $M^{\st}(N)$. We will primarily follow the treatment in \cite[Section 4]{Manning}\footnote{One may note that \cite{Manning} only handles the ``minimal level'' case, i.e. only $M^{\st}(N)$ and not $M(NQ^2)$ in our notation, however all results we will cite can easily be seen to hold in our situation with no modifications to the proofs.}

First, by the method outlined in \cite[Section 4.3]{Manning}, there exist: 
\begin{itemize}
	\item Integers $g,d\ge 0$, satisfying $d+1 = \dim R_{\loc}+g = \dim R_{\loc}^{\st}+g$ (see \cite[Lemma 2.5]{Manning} and \cite[Proposition (3.2.5)]{Kisin});
	\item For each $n\ge 1$, an $S_\infty$-algebra $R_n^\square$ with a surjective map $R_\infty\onto R_n$ and an isomorphism $R_n^\square \otimes_{S_\infty} \cO \cong R$, where $S_\infty$ and $R_\infty$ are as in the theorem statement, and the kernel of the structure map $S_\infty \to R_n^\square$ has a specific from (as given in \cite[Proposition 4.7]{Manning}) 
	\item For each $n\ge 1$, a finitely generated $R_n^\square$-module $M_n^\square$, which is free over $\mathrm{im}(S_\infty\to R_n^\square)$ and satisfies $M_n^\square\otimes_{S_\infty}\cO \cong M(NQ^2)$.
\end{itemize}
Briefly, $R_n^\square$ is a (framed) global deformation ring constructed exactly as in Section \ref{ssec:global def} of this paper, except that we also allow deformations with ramification at a carefully chosen set of primes $Q_n$ (constructed in \cite[Propositon (3.2.5)]{Kisin}), disjoint from $S$. The module $M_n^\square$ is simply constructed from the homology group of an appropriate modular curve just as in \cite[Section 4.3]{Manning} (although to keep this construction consistent with the construction of $M(NQ^2)$ given in this paper, we must also modify this space by localizing at the maximal ideal $\ffrm_Q$ of the full Hecke algebra, as in Section \ref{ssec:modular curve} of this paper --- this does not cause any issue for the arguments of \cite{Manning}).

The `ultrapatching' construction described in \cite[Section 4.1]{Manning} (as well as in the proof of Lemma 4.8) then produces an $S_\infty$-algebra $\Rt_\infty$ (which would be called $\mathscr{P}(\mathscr{R}^\square)$ in the notation of that paper) as well as an $\Rt_\infty$-module $M_\infty$, for which:
\begin{itemize}
	\item $M_\infty$ is finite free over $S_\infty$;
	\item $\Rt_\infty\otimes_{S_\infty}\cO \cong R$ and $M_\infty\otimes_{S_\infty}\cO\cong M(NQ^2)$;
	\item There is a surjection $R_\infty\onto \Rt_\infty$ such that the composition
	\[R_{\loc}\into R_\infty\onto \Rt_\infty\onto R\]
	is the map $R_{\loc}\to R$ from the start of this section.
\end{itemize}
Analogously, for each $n\ge 1$ we may also construct a quotient $R_n^{\square,\st}$ of $R_n^\square$ and an $R_n^{\square,\st}$-module $M_n^{\square,\st}$ satisfying analogous properties to $R_n^\square$ and $M_n^\square$ (this time identically to the construction given in \cite{Manning}). Again applying the ultrapatching construction produces an $S_\infty$-algebra $\Rt_\infty^{\st}$ and a $\Rt_\infty^{\square,\st}$-module $M_\infty^{\st}$, which is finite free over $S_\infty$, and satisfies $\Rt_\infty^{\st}\otimes_{S_\infty}\cO \cong R^{\st}$ and $M_\infty^{\st}\otimes_{S_\infty}\cO\cong M^{\st}(N)$.

Moreover, as each $R_n^{\square,\st}$ is a quotient of $R_n^\square$, the natural functorality of the ultrapatching construction (see \cite[Section 4.1]{Manning}) gives a quotient map $\Rt_\infty\onto\Rt_\infty^{\st}$, and one can easily check that the maps we've described fit into the commutative diagram:

\begin{center}
	\begin{tikzpicture}
	\node(11) at (2,2) {$R_{\loc}$};
	\node(21) at (4,2) {$R_\infty$};
	\node(31) at (6,2) {$\Rt_\infty$};
	\node(41) at (8,2) {$R$};
	
	\node(10) at (2,0) {$R_{\loc}^{\st}$};
	\node(20) at (4,0) {$R_\infty^{\st}$};
	\node(30) at (6,0) {$\Rt_\infty^{\st}$};
	\node(40) at (8,0) {$R^{\st}$};	
	
	\draw[->>] (11)--(10);
	\draw[->>] (21)--(20);
	\draw[->>] (31)--(30);
	\draw[->>] (41)--(40);

	\draw[right hook-latex] (11)--(21);
	\draw[->>] (21)--(31);
	\draw[->>] (31)--(41);
	
	\draw[right hook-latex] (10)--(20);
	\draw[->>] (20)--(30);
	\draw[->>] (30)--(40);
	\end{tikzpicture}
\end{center}
where the compositions of the top and bottom rows are the maps $R_\loc\to R$ and $R_{\loc}^{\st}\to R^{\st}$ from before.

Now as $\Rt_\infty$ is an $S_\infty$-algebra, $R_\infty$ is a complete local ring, and $S_\infty = \cO[[y_1,\ldots,y_d]]$ is a power series ring, we may lift the structure map $S_\infty\to \Rt_\infty$ to a map $i:S_\infty\to R_\infty$ making $\pi_\infty:R_\infty \to \Rt_\infty$ into an $S_\infty$-module surjection. This proves (1) and (2).

(3) follows by noting that $(\varpi,i(y_1),\ldots,i(y_d))$ and $(\varpi,i^{\st}(y_1),\ldots,i^{\st}(y_d))$ are regular sequences for $M_\infty$ and $M_\infty^{\st}$, respectively, as $M_\infty$ and $M_\infty^{\st}$ are both free over $S_\infty$.

But now by standard properties of maximal Cohen--Macaulay modules, the supports of $M_\infty$ and $M_\infty^{\st}$ are unions of irreducible components of $\Spec R_\infty$ and $\Spec R_\infty^{\st}$, respectively. As $R_\infty^{\st}$ is a domain, this implies that it acts faithfully on $M_\infty^{\st}$. 

Now as $R_\infty = R_{\loc}[[x_1,\ldots,x_g]]$, the irreducible components of $\Spec R_\infty$ are in bijection with those of $\Spec R_{\loc}$. By Lemma \ref{lem:M(NQ^2) support}, it follows that each irreducible component of $\Spec R_\infty$ contains a point in the support of $M_\infty/(i(y_1),\ldots,i(y_d))\otimes_\cO E = M(NQ^2)\otimes_\cO E$ (which is not contained in any other component, by Lemma \ref{lem:R_loc smooth}). Since the support of $M_\infty$ clearly contains the support of $M_\infty/(i(y_1),\ldots,i(y_d))\otimes_\cO E$, it follows that $M_\infty$ is supported on all of $\Spec R_\infty$. Since $R_\loc$, and hence $R_\infty$, is reduced (by Proposition \ref{prop:R_v-Loc}) it follows that $R_\infty$ acts faithfully on $M_\infty$.

Now by definition, the action of $R_\infty$ on $M_\infty$ factors through $\pi_\infty:R_\infty\to \Rt_\infty$. Since $R_\infty$ acts faithfully on $M_\infty$, it follows that $\pi_\infty$ must be an isomorphism of $S_\infty$-algebras. By an identical argument $\pi_\infty^{\st}:R_\infty^{\st}\to \Rt_\infty^{\st}$ is also an isomorphism of $S_\infty$-algebras. (5) now follows immediately from the properties of $\Rt_\infty$ and $\Rt_\infty^{\st}$ mentioned above.

Now as $M_\infty$ is supported on every geometric point of $\Spec R_\infty$, $M(NQ^2) = M_\infty\otimes_{S_\infty}\cO$ is supported on every geometric point of $\Spec (R_\infty\otimes_{S_\infty}\cO) = \Spec R$, and so as the action of $R$ on $M(NQ^2)$ factors through the surjection $R\to \T$, we must have $R^{red} = \T$ (as $\T$ is reduced). In particular, as $\T$ is finite over $\cO$, $R$ is finite over $\cO$ as well. Similarly $R^{\st}$ is finite over $\cO$.

In particular, $R/\varpi = R_\infty/(\varpi,i(y_1),\ldots,i(y_d))$ and $R^{\st}/\varpi = R_\infty^{\st}/(\varpi,i^{\st}(y_1),\ldots,i^{\st}(y_d))$ are both finite, and hence are zero dimensional. It follows that $(\varpi,i(y_1),\ldots,i(y_d))$ and $(\varpi,i^{\st}(y_1),\ldots,i^{\st}(y_d))$ are systems of parameters for $R_\infty$ and $R_\infty^{\st}$, respectively. As both of these rings are Cohen--Macaulay by Proposition \ref{prop:R_v-Loc}, it follows that these are regular sequences. Equivalently, treating $R_\infty$ and $R_\infty^{\st}$ as $S_\infty$-modules, these both have $(\varpi,y_1,\ldots,y_d)$ as a regular sequence, and so are maximal Cohen--Macaulay over $S_\infty$. But all finitely generated maximal Cohen--Macaulay modules over $S_\infty$ are free (by the Auslander-Buchsbaum formula), so (6) follows.

In particular, $R = R_\infty\otimes_{S_\infty}\cO$ and $R^{\st} = R_\infty^{\st}\otimes_{S_\infty}\cO$ are finite flat over $\cO$.

Now similarly to \cite[Proposition (3.3.1)]{Kisin}, take any $x\in \Spec R[1/\varpi]\subseteq \Spec (R_\infty\otimes_\cO E)$. As noted above, $x$ is in the support of $M(NQ^2)$ and so Lemma \ref{lem:R_loc smooth} gives that $\Spec (R_\infty\otimes_\cO E)$ is formally smooth at $x$, which implies that $(M_\infty\otimes_\cO E)_x$ is projective, and hence free, over $(R_\infty\otimes_\cO E)_x$. It follows that $(M(NQ^2)\otimes_\cO E)_x = (M_\infty\otimes_\cO E)_x/(i(y_1),\ldots,i(y_d))$ is free (of rank $1$) over $R[1/\varpi]_x$. Thus $M(NQ^2)\otimes_\cO E$ is locally free of rank $1$ over $R[1/\varpi]$. As $R[1/\varpi]$ has dimension $0$, and is thus a direct product of local $E$-algebras, it follows that $M(NQ^2)\otimes_\cO E$ is free of rank $1$ over $R[1/\varpi]$.

Since the action of $R[1/\varpi]$ on $M(NQ^2)\otimes_\cO E$ factors through $R[1/\varpi]\onto \T[1/\varpi]$, it follows that the map $R[1/\varpi]\onto \T[1/\varpi]$ is an isomorphism, and so the kernel of the surjection $R\onto\T$ is a torsion $\cO$-module. As noted above $R$ is free over $\cO$, and thus has no nontrivial torsion submodules. It follows that the map $R\to \T$ is an isomorphism.\footnote{See also the arguments in \cite[Section 5]{Snowden} or \cite[Section 4.3]{Manning} for the observation that $R_\infty$ being Cohen--Macaulay implies an integral $R=\T$ theorem.} Similarly $R^{\st}\to \T^{\st}$ is an isomorphism.  This proves (7).
\end{proof}

\section{Patching and the growth of cotangent spaces}\label{sec:patchandgrow}

\subsection{Continuous K\"ahler differentials}

The module $\Omega_{B/A}$ of K\"ahler differentials of $B$ over $A$ (for a ring map $A\to B$) can be a rather poorly behaved object when $B$ is merely \emph{topologically} finitely generated over $A$ rather than literally finitely generated. As the local deformation rings we will be considering in our arguments are merely topologically finitely generated, it will often be necessary for us to consider the module of \emph{continuous} K\"ahler differentials $\cOmega_{B/A}$ in place of $\Omega_{B/A}$. In this section, we review the definition and basic properties of continuous K\"ahler differentials.

For the reminder of this section, for any $A\in \CNLO$ we will always use $\ffrm_A$ to denote the unique maximal ideal of $A$. For any $A,B\in\CNLO$ we will say that a homomorphism $f:A\to B$ is local if $\ffrm_A = f^{-1}(\ffrm_B)$ (or equivalently $f(\ffrm_A) \subseteq \ffrm_B$).

For any $A,B\in\CNLO$ and any local homomorphism $A\to B$ we will define
\[\cOmega_{B/A} = \invlim_n \Omega_{(B/\ffrm_B^n)/(A/\ffrm_A^n)} = \invlim_n \Omega_{(B/\ffrm_B^n)/A}\]
(where the equality $\Omega_{(B/\ffrm_B^n)/(A/\ffrm_A^n)} =\Omega_{(B/\ffrm_B^n)/A}$ follows from \cite[\href{https://stacks.math.columbia.edu/tag/00RR}{Lemma 00RR}]{stacks-project}). Note that if $B$ is topologically generated as an $A$-algebra by $b_1,\ldots,b_N\in B$ then each $\Omega_{(B/\ffrm_B^n)/A}$, and thus $\cOmega_{B/A}$ itself, is generated as a $B$-module by $\df b_1,\ldots, \df b_N$.

First we observe the following:

\begin{lemma}\label{lem:diff completion}
Let $A\in\CNLO$ and let $\cR$ be a finitely generated as $A$-algebra. Let $f:A\to \cR$ be the structure map, and let $\ffrm_{\cR}\subseteq \cR$ be a maximal ideal for which $f^{-1}(\ffrm_{\cR}) = \ffrm_A$, and define $R = \cR_{\ffrm_{\cR}}\in \CNLO$. Then 
\[\cOmega_{R/A} = \Omega_{\cR/A}\otimes_{\cR}R.\]
\end{lemma}
\begin{proof}
By \cite[\href{https://stacks.math.columbia.edu/tag/02HQ}{Lemma 02HQ}]{stacks-project} for any $k$ and any $n>k$ we have a canonical isomorphism 
\[\Omega_{\cR/A} \otimes_{\cR} \cR/\ffrm_{\cR}^k \xrightarrow{\sim} \Omega_{(\cR/\ffrm_{\cR}^n)/A}\otimes_{\cR} \cR/\ffrm_{\cR}^k\]
and so taking inverse limits we have
\[\Omega_{\cR/A} \otimes_{\cR} \cR/\ffrm_{\cR}^k \xrightarrow{\sim} \invlim_n\left(\Omega_{(\cR/\ffrm_{\cR}^n)/A}\otimes_{\cR} \cR/\ffrm_{\cR}^k\right).\]
Now for each $n$, $\cR/\ffrm_{\cR}^n$ is certainly a finitely generated $A$-algebra, and so $\Omega_{(\cR/\ffrm_{\cR}^n)/A}$ is a finitely generated $\cR/\ffrm_{\cR}^n$-module. In particular, it has finite length. As $\cR/\ffrm_{\cR}^k$ is a finitely presented $\cR$-module, the above work, and the Mittag--Leffler conditions (see \cite[\href{https://stacks.math.columbia.edu/tag/0594}{Section 0594}]{stacks-project}), imply that
\begin{align*}
\left(\Omega_{\cR/A}\otimes_{\cR}R\right)\otimes_{\cR} \cR/\ffrm_{\cR}^k &\cong
\Omega_{\cR/A} \otimes_{\cR} \cR/\ffrm_{\cR}^k\cong \invlim_n\left(\Omega_{(\cR/\ffrm_{\cR}^n)/A}\otimes_{\cR} \cR/\ffrm_{\cR}^k\right)\\
&\cong \invlim_n\left(\Omega_{(\cR/\ffrm_{\cR}^n)/A}\right)\otimes_{\cR} \cR/\ffrm_{\cR}^k\cong \cOmega_{R/A}\otimes_{\cR} \cR/\ffrm_{\cR}^k
\end{align*}
for all $k$. Since $\Omega_{\cR/A}\otimes_{\cR}R$ and $\cOmega_{R/A}$ are both finitely generated $R$-modules, the claim follows.
\end{proof}

\begin{lemma}\label{lem:diff finite}
If $A,B\in\CNLO$ are such that $B$ is an $A$-algebra which is finitely generated as an $A$-module, then $\cOmega_{B/A}\cong \Omega_{B/A}$.
\end{lemma}
\begin{proof}
$B$ is already finitely generated as an $A$-algebra, so we may simply take $R=\cR=B$ in Lemma \ref{lem:diff completion}.
\end{proof}

\begin{lemma}\label{lem:diff aug}
Take any $R\in\CNLO$ and any augmentation $\lambda:R\to \cO$. Defining $\Phi_{\lambda,R}:= (\ker \lambda)/(\ker \lambda)^2$ as in section \ref{sec:congruence}, we have $\Phi_{\lambda,R}\cong \cOmega_{R/\cO}\otimes_\lambda\cO$

In particular, in the setup of Lemma \ref{lem:diff completion} with $A=\cO$, if $\lambda_{\cR}:\cR\to \cO$ is an augmentation extending to an augmentation $\lambda:R\to \cO$ then we have $\Phi_{\lambda,R}\cong \Omega_{\cR/A}\otimes_{\lambda_{\cR}}\cO\cong (\ker \lambda_{\cR})/(\ker \lambda_{\cR})^2$.
\end{lemma}
\begin{proof}
Let $I = \ker \lambda$, so that $R/I\cong \cO$. As in the proof of Lemma \ref{lem:diff completion} we get that
\begin{align*}
\cOmega_{R/A}\otimes_\lambda\cO &= \cOmega_{R/A}\otimes_R R/I = \invlim_n\left(\Omega_{(R/\ffrm_{R}^n)/\cO}\right)\otimes_{R} R/I = \invlim_n\left(\Omega_{(R/\ffrm_{R}^n)/\cO}\otimes_{R} R/I\right).
\end{align*}
Now for each $n$, $\lambda:R\to \cO$ gives a surjection $\overline{\lambda}_n:R/\ffrm_{R}^n\to \cO/\varpi^n$. Let $\overline{I}_n = \ker \overline{\lambda}_n = (I+\ffrm_R^n)/\ffrm_R^n \cong I/(I\cap \ffrm_R^n)$. Now note that the structure map $\cO\to \cO/\varpi^n\to R/\ffrm_R^n$ provides a section to $\overline{\lambda}_n$ so \cite[\href{https://stacks.math.columbia.edu/tag/02HP}{Lemma 02HP}]{stacks-project} (as well as the fact that $\Omega_{(\cO/\varpi^n)/\cO} = 0$) give a natural isomorphism
\[
\Omega_{(R/\ffrm_{R}^n)/\cO}\otimes_{R} R/I \cong \Omega_{(R/\ffrm_{R}^n)/\cO}\otimes_{R/\ffrm_{R}^n} (R/\ffrm_{R}^n)/\overline{I}_n \cong \overline{I}_n/\overline{I}_n^2.
\]
Taking inverse limits (and again applying the Mittag--Leffler conditions to $0\to\overline{I}_n^2\to\overline{I}_n\to \overline{I}_n/\overline{I}_n^2\to 0$) gives the isomorphism $\cOmega_{R/A}\otimes_\lambda\cO\cong I/I^2 = \Phi_{\lambda,R}$.

The second claim follows by noting that $\Omega_{\cR/A}\otimes_{\lambda_{\cR}}\cO\cong (\ker \lambda_{\cR})/(\ker \lambda_{\cR})^2$ (by a similar, but simpler, argument to the one above) and $(\ker \lambda)/(\ker \lambda)^2\cong (\ker \lambda_{\cR})/(\ker \lambda_{\cR})^2$.
\end{proof}

\begin{lemma}\label{lem:diff tensor}
Given $A,B,C\in \CNLO$ and local homomorphisms $A\to B$ and $A\to C$ making $B$ and $C$ into $A$-algebras we have an isomorphism
\[\cOmega_{(B\cotimes_A C)/A} \cong \left(\cOmega_{B/A}\cotimes_AC\right)\oplus \left(B\cotimes_A\cOmega_{C/A}\right)\]
\end{lemma}
\begin{proof}
Let $R = B\cotimes_A C$. For any $n\ge 1$ define the ideal $I_n\subseteq R$ by
\[I_n = \ffrm_B^nR+\ffrm_C^nR = \ker\left(R\onto (B/\ffrm_B^n)\otimes_A (C/\ffrm_C^n)\right).\]
As $\{I_n\}$ is cofinal with $\{\ffrm_R^n\}$ we get that 
\[\cOmega_{R/A} \cong \invlim_n \Omega_{(R/I_n)/A} \cong \invlim_n \Omega_{[(B/\ffrm_B^n)\otimes_A (C/\ffrm_C^n)]/A}.\]
But now by \cite[Proposition 16.5]{Eisenbud},
\[
\Omega_{[(B/\ffrm_B^n)\otimes_A (C/\ffrm_C^n)]/A}\cong \left(\cOmega_{(B/\ffrm_B^n)/A}\otimes_A(C/\ffrm_C^n)\right)\oplus \left((B/\ffrm_B^n)\otimes_A\cOmega_{(C/\ffrm_C^n)/A}\right)
\]
so taking inverse limits gives the result.
\end{proof}

Lastly, a standard property of K\"ahler differentials (cf. \cite[\href{https://stacks.math.columbia.edu/tag/00RS}{Lemma 00RS}]{stacks-project}) states that for any ring maps $A\to B\to C$ we have an exact sequence
\[\Omega_{B/A}\otimes_BC\to \Omega_{C/A}\to \Omega_{C/B}\to 0.\]
We will need the following continuous version of this below:

\begin{lemma}\label{lem:diff A>B>C}
Let $A\to B\to C$ be local ring homomorphisms, for $A,B,C\in \CNLO$. Then there is an exact sequence
\[
\cOmega_{B/A}\otimes_BC\to \cOmega_{C/A}\to \cOmega_{C/B}\to 0.
\]
\end{lemma}
\begin{proof}
Since the maps $A\to B\to C$ are local ring homomorphisms, they clearly induce homomorphisms $A\to B/\ffrm_B^n\to C/\ffrm_C^k$ for all $n>k$. Thus we have an exact sequence
\[\Omega_{(B/\ffrm_B^n)/A}\otimes_BC/\ffrm_C^k\to \Omega_{(C/\ffrm_C^k)/A}\to \Omega_{(C/\ffrm_C^k)/B}\to 0\]
(where we used the fact that $\Omega_{(C/\ffrm_C^k)/(B/\ffrm_B^n)} = \Omega_{(C/\ffrm_C^k)/B}$).

As $C/\ffrm_C^k$ is a finite group, it is certainly finitely presented as a $B$-module. Thus as each $\Omega_{(B/\ffrm_B^n)/A}$ has finite length, the Mittag--Leffler conditions then imply that
\[
\invlim_n\left(\Omega_{(B/\ffrm_B^n)/A}\otimes_BC/\ffrm_C^k\right)\cong \invlim_n\left(\Omega_{(B/\ffrm_B^n)/A}\right)\otimes_BC/\ffrm_C^k \cong \cOmega_{B/A}\otimes_BC/\ffrm_C^k,
\]
and so we have an exact sequence 
\[\cOmega_{B/A}\otimes_BC/\ffrm_C^k\to \Omega_{(C/\ffrm_C^k)/A}\to \Omega_{(C/\ffrm_C^k)/B}\to 0\]
for each $k$. As $\cOmega_{B/A}$ is a finitely generated $B$-module, each module in the above sequence again has finite length so taking the inverse limit with respect to $k$, and using the Mittag--Leffler conditions again, give the desired exact sequence.
\end{proof}

\subsection{Relative local cotangent spaces at an augmentation}
\label{sec:local-calc}

We consider primes $q \in Q$, recall these are Steinberg for $\rhobar$. We begin with a general definition that we need.

\begin{definition}\label{relative}
	Given $R,S \in \CNL_\cO$, a surjective map $\alpha: R \ra S$ in $\CNL_\cO$,  and an augmentation  $\lambda: S \ra \cO$ whose pull back to $R$ we denote by $\lambda_R$, we  define the \emph{relative cotangent space} $\Phi_{\lambda,R/S}$ to be the kernel of the surjective map $\Phi_{\lambda_R,R} \ra \Phi_{\lambda,S}$, i.e., the kernel of the map 
	$ \cOmega_{R/\cO} \otimes_{\lambda_R} \cO \ra \cOmega_{S/\cO} \otimes_{\lambda_S} \cO$.
\end{definition}

\begin{definition}\label{invariants}
	We fix an augmentation $\lambda: R_q^\St \ra \cO$ such that the corresponding  representation of $G_q \ra \GL_2(\cO)$   is of the form \[  \left( \begin{array}{cc} \varepsilon & \ast \\ 0 & 1 \end{array} \right)\] up to twist by an unramified character of order dividing 2, and which is ramified (i.e. generic). We define the local invariants $m_q$ (respectively, $n_q$) to be the largest integer $n$ such that $\rho_\lambda(I_q)$ (respectively, $\rho_\lambda(G_q)$) mod $\varpi^n$ has trivial projective image.
\end{definition}

\begin{lemma}
	We have the inequalities $n_q \leq m_q$, and $n_q \leq {\rm ord}_\varpi(q-1)$
\end{lemma}

\begin{proof}
	The first inequality follows from $I_q \leq G_q$. The second inequality follows from the fact that the characteristic
	polynomial of $\rho_\lambda({\rm Frob}_q)= x^2 \pm (q+1)x+q$  (where ${\rm Frob}_q$ is any lift of the Frobenius  automorphism of the residue field of $q$).
\end{proof}

We have the following key proposition.

\begin{proposition}\label{key-local-comp} Consider a prime $q \in Q$.
We fix an augmentation $\lambda: R_q^\St \ra \cO$ such that the corresponding  representation of $G_q \ra \GL_2(\cO)$   is of the form \[ \left( \begin{array}{cc} \varepsilon & \ast \\ 0 & 1 \end{array} \right)\] up to twist by an unramified character of order dividing 2, and  which is ramified (i.e. generic). Then the following hold for the relative cotangent space at $\lambda$, where cases 2 and 3 only occur if $\rhobar$ is unramified:

\begin{enumerate}

\item For the morphism $R_q^\square \ra R_q^\St$, the module $\Phi_{\lambda,R_q^\square/R_q^\St}$ is finite over $\cO$ of length $ m_q + \ord_\cO(q^2-1)  -2n_q$.

\item For the morphism $R_q^\square \ra R_q^\uni$, the module $\Phi_{\lambda,R_q^\square/R_q^\uni}$ is finite cyclic over $\cO$ of length $ \ord_\cO(q^2-1) -n_q$.

\item For the morphism $R_q^\uni \ra R_q^\St$, the module $\Phi_{\lambda,R_q^\uni/R_q^\St}$ is finite cyclic over $\cO$ of length $ m_q -n_q$.

\end{enumerate}

\end{proposition}

\begin{proof}
	By possibly twisting with the unramified character of order $2$, we may assume that our representation is of the form  
	\[ \left( \begin{array}{cc} \varepsilon & \ast \\ 0 & 1 \end{array} \right).\]
	We first recall well-known explicit descriptions of the rings $R_q^\square \ra R_q^\uni \ra R_q^\St$ introduced in Subsection~\ref{Subsec:LocDefRings}. We have two limiting conditions: We require our deformations to have determinant the $p$-adic cyclotomic character.  Moreover the deformations of $\rhobar$ factor via the tame quotient $G_q^t$ of $G_q$, because the order of inertia under $\rhobar$ divides $p\neq q$. The group $G_q^t$ can be identified with the profinite completion of the group $\langle\sigma,\tau\mid \sigma\tau\sigma^{-1}=\tau^q\rangle$ in such a way that $\sigma$ maps to Frobenius in the unramified quotient of $G_q$. 
	
	Imposing the conditions just stated on the images of $\sigma$ and $\tau$ in a matrix group, and choosing the augmentation $\lambda$ as the origin of our coordinates, leads to the following description of the universal deformation $\rho_q^\square  \colon G^t_q \to \GL_2(R^\square_q)$: Let $\cR:= \cO[[a,b,c,e,\alpha,\beta,\gamma,\delta]]$, let $A,B\in\GL_2(\cR)$ be the matrices
	\[ A:=\left( \begin{array}{cc} q+a & b +s\\ c & 1-a-e \end{array} \right),\quad B:=\left( \begin{array}{cc} 1+\alpha & \beta+t \\ \gamma & 1+\delta \end{array}\right)\]
	with $s,t\in\cO$ and $t$ non-zero, and let $\cI_q^\square=(r_1^\square,\ldots,r_6^\square)\subset\cR$ be the ideal generated by $r_1^\square=\det A-q=a-(q+a)(a+e)-(b+s)c$, $r_2^\square=\det B-1=\alpha+\delta+\alpha\delta-(\beta+t)\gamma$ and 
	\[ \left( \begin{array}{cc} r^\square_3 & r^\square_4 \\ r^\square_5 & r^\square_6 \end{array} \right)=AB-B^qA. \]
	Then $R_q^\square=\cR/\cI_q^\square$ and $\rho_q^\square$ is defined by $\sigma\mapsto A \pmod {\cI_q^\square}$ and $\tau \mapsto B \pmod {\cI_q^\square}$. It is not hard to see that $\cI_q^\square$ is generated by $r_1^\square$ and $r_2^\square$ and three additional elements from $\{r_3^\square,r_4^\square,r_5^\square,r_6^\square\}$. The ring $R_q^\square$ is a complete intersection ring, as we shall recall from work of Shotton in Proposition~\ref{prop:R_v-Atq}.

The ring $R_q^\St $ is described as a quotient of $\cR$ by the following ideal $\cI_q^\St$ by using either \cite[proof of Proposition~5.8]{Shotton} or Lemma~\ref{Lem:OnShotton} below; we give further details in Remarks~\ref{Rem-DetailsOnRst-Using7.12} and~\ref{Rem-CompareToShotton}.
Explicitly $\cI_q^\St$ is generated by the six $2\times 2$-minors $r_1^\St,\ldots,r_6^\St$ of  
\begin{equation}\label{eq-MinorsMatrix}
\left(
\begin{array}{cccc}
 \alpha & \beta+t  & (q\!-\!1\!+\!a)  & b+s \\
\gamma  & \delta  & c  & -a-e 
\end{array}
\right).
\end{equation}
together with $r_7^\St=\alpha+\delta$ and $r_8^\St=e$. For explicit use below, let $r_1^\St$ and $r_2^\St$ denote $2\times2$ minors of the above matrix formed by the first two columns, and the last two columns, respectively. The ring-theoretic properties of $R_q^\St$ are given in Proposition~\ref{prop:R_v-Atq}.

To describe $R_q^\uni$, let $\cI^\unr$ denote the ideal of $\cR$ containing $\cI_q^\square$ so that $R_q^\unr=\cR/\cI_q^\unr$ describes unramified deformations of $\rhobar$, i.e. $\cI_q^\unr=(\alpha,\beta+t,\gamma,\delta,r_1^\square)$. Then $R^\uni_q=\cR/\cI_q^\uni$ where $\cI_q^\uni=\cI_q^\St\cap\cI_q^\unr$: Explicitly one has $\cI^\uni_q=(r_1^\uni,\ldots,r_{10}^\uni)$, where $r_1^\uni=r_1^\square$ and $r_2^\uni=r_2^\square$ are the determinant conditions from $\cI_q^\square$, $r_i^\uni=r_i^\St$ for $i=3,\ldots,7$ and $r_8^\uni=\alpha e$,  $r_9^\uni=(\beta+t) e$,  $r_{10}^\uni=\gamma e$. The exact sequence
\[ 0\to \cR/I_q^\uni \to R_q^\unr\times R_q^\st \to \cR/(\cI_q^\st+\cI_q^\unr) \to 0\]
allows one to verify that $\cR/I_q^\uni$ is Cohen-Macaulay (this also requires \cite[Exercise 18.13]{Eisenbud}), reduced, flat over $\cO$ and of relative dimension $3$ and that $\Spec \cR/I_q^\uni$ is the reduced schematic closure in $\Spec R_q^\square$ of $\Spec R_q^\St\cup R_q^\unr$. This proves the claim $R^\uni_q=\cR/\cI_q^\uni$.

Note that our generating sets of the ideals $\cI_q^?$ are not claimed to be minimal. But for the computations to come, this is not relevant. 

	\smallskip
	
	To compute the relative cotangent spaces, observe first that $\cOmega_{\cR/\cO}\otimes_\cR\cO$ is the free $\cO$-module of rank $8$ with basis $\df a$, $\df b$, $\df c$, $\df e$, $\df \alpha$, $\df \beta$, $\df \gamma$, $\df \delta$. We identify this module with $\cO^8$ via this basis and regard it as a space of row vectors. We shall write $\df r|_\lambda$ for the evaluation of a differential $\df r$, $r\in\cR$, under $\lambda$. Evaluation under $\lambda$ simply means that the coefficients in $\cR$ of $\df r$ are evaluated at the origin, i.e., we set all of $a,b,c,e,\alpha,\beta,\gamma,\delta$ to zero. For $?\in\{\square,\uni,\St\}$ we denote by $\Lambda^?\subset \cOmega_{\cR/\cO}\otimes_\cR\cO$ the $\cO$-submodule spanned by elements $\df r_j|_\lambda$ such that the $\df r_j$ generate $\cI^?_q$. Then $\cOmega_{R_q^?/\cO}\otimes_\cR\cO$ is the quotient of $\cOmega_{\cR/\cO}\otimes_\cR\cO$ by $\Lambda^?$.
	We shall compute $\cO$-basis of the $\cO$-submodules $\Lambda^?$ of $\cOmega_{\cR/\cO}\otimes_\cR\cO$. The relative cotangent spaces are the $\cO$-modules
\begin{equation}\label{eq-RelCotangAndLattices}
 \Phi_{\lambda,R^?_q/R_q^{??}}\cong \Lambda^{??}/\Lambda^?.
\end{equation}
The quotient is described by the square matrix expressing the basis we find for $\Lambda^?$ in terms of the basis we find for $ \Lambda^{??}$. Finding $\Phi_{\lambda,R^?_q/R_q^{??}}$ is thereby reduced to the elementary divisor theorem.

	The computation of $\Lambda^\square$ proceeds~as~follows: Let $N=B-I$, so that $N|_\lambda=\Mat{0}{t}{0}{0}$ and $(N|_\lambda)^2=0$.
	\[\df r_1^\square=(1-a-e)\df a- (q+a) (\df a+\df e) 
	-(b+s)\df c-c\df b\rightsquigarrow \df r_1^\square|_\lambda=\df a-q\df a-q\df e-s\df c.\]
	\[\df r_2^\square=\df \alpha+\df \delta+\alpha \df\delta+\delta\df \alpha-(\beta+t)\df\gamma-\gamma\df\beta \rightsquigarrow \df r_2^\square|_\lambda=\df \alpha+\df \delta-t\df\gamma .\]
	\[\df \bigg( \begin{array}{cc} r^\square_3 & r^\square_4 \\ r^\square_5 & r^\square_6 \end{array} \bigg)=\df\Bigg( AN-\sum_{i=1}^q\binom{q}i N^i A\Bigg)=
	(\df A)N-\sum_{i=1}^q\binom{q}i N^i \df A +A(\df N)- \sum_{i=1}^q\binom{q}i \sum_{j=1}^i N^{j-1} (\df N)N^{i-j} A . \]
	
	\[\rightsquigarrow
	\Mat{\df r^\square_3|_\lambda}{\df r^\square_4|_\lambda}{\df r^\square_5|_\lambda}{ \df r^\square_6|_\lambda}=
	\df A \Mat0{t}00 -\binom{q}1 \Mat0t00 \df A +\bigg(\!\Mat{q}s01 -qI\!\bigg) \df N-  \df N  \bigg(\!\binom{q}1 \Mat{q}s01 -qI\!\bigg)\]
	\[-\binom{q}2 \bigg( \df N \Mat0t00 + \Mat0t00 \df N\bigg) \Mat{q}s01 -  \binom{q}3 \Mat0t00 \df N \Mat0t00\Mat{q}s01 =
	\]
\[\left(\!\!\!
\setlength\arraycolsep{3.3pt}
\begin{array}{cccc}\setlength\arraycolsep{0pt}
- qt \df c+s\df \gamma-(q^2-q)\df \alpha &
t\df a+qt\df a+qt\df e+s\df \delta-qs\df \alpha \\
-(q-1)\df \gamma-(q^2-q)\df \gamma&
t\df c-(q-1)\df \delta-qs\df \gamma
\end{array}
\!\!\!\right)
-t\binom{q}2\!
\left(\!\!\!
\setlength\arraycolsep{3.3pt}
\begin{array}{cccc}
q\df\gamma&
s\df\gamma+\df \alpha+\df\delta\\
0&
\df\gamma
\end{array}
\!\!\!\right)
-t^2\binom{q}3\!
\left(\!\!\!
\setlength\arraycolsep{3.3pt}
\begin{array}{cccc}
0&
\df\gamma\\
0&
0
\end{array}
\!\!\!\right)\!.
\]
This gives 6 row vectors in $\cO^8$ that span $\Lambda^\square$ and that we write as a matrix. Gau\ss\ elimination shows that the rows of the following matrix form an $\cO$-basis of $\Lambda^\square$, 
	\[
	\left( \begin{array}{cccccccc} 
	q-1&0&s&q&0&0&0&0\\ 
	0&0&0&0&1&0&-t&1\\ 
	0&0&0&0&0&0&q^2-1&0\\ 
	0&0&t&0&0&0&\zeta&1-q\\ 
	2t&0&0&0&0&0&\xi&2s\\ 
	\end{array} \right) ,
	\]
	where we abbreviate $\xi=-qs^2-st-t^2\binom{q+1}3$ and $\zeta=-qs-t\binom{q}2$.
	
To compute a basis of $\Lambda^\uni$, observe first that $\df r^\uni_i=\df r^\square_i$, for $i=1,2$, and that their evaluation under $\lambda$ is given by the first two rows of the last matrix. A straightforward computation gives $\df r^\uni_i|_\lambda$, $i=7,\ldots,10$ as
\[ \df r^\uni_7|_\lambda=\df \alpha+\df \delta, \df r^\uni_8|_\lambda=(\alpha\df e+e\df \alpha)|_\lambda=0,  \df r^\uni_9|_\lambda= \ldots=t\df e,  \df r^\uni_{10}|_\lambda= \ldots=0.  \]
When computing $\df(\ldots)|_\lambda$ for the $2\times 2$-minors of the matrix in \eqref{eq-MinorsMatrix}, observe that the evaluation under $\lambda$ of the second row is identically zero. One deduces that $\df(\ldots)|_\lambda$ of these $2\times 2$-minors are the $2\times 2$-minors of the matrix
\begin{equation}\nonumber
\left(
\begin{array}{cccc}
0 & t  & q-1  & s \\
\df \gamma  & \df \delta  & \df c  & -\df a-\df e 
\end{array}
\right).
\end{equation}
The `inner' four of the resulting 6 minors provide us with the missing generators for $\Lambda^\uni$. Gau\ss\ elimination on the resulting $10\times 8$ matrix shows that a basis of $\Lambda^\uni$ over $\cO$ is given by the rows of 
	\[
	\left( \begin{array}{cccccccc} 
	q-1&0&s&q&0&0&0&0\\ 
	0&0&0&0&1&0&0&1\\ 
	0&0&0&0&0&0&f&0\\ 
	0&0&t&0&0&0&0&1-q\\ 
	t&0&0&0&0&0&0&s\\ 
	\end{array} \right) 
	\]
with $f=\gcd(s,t,q-1)$.  We clearly have $\Lambda^\uni\supset \Lambda^\square$. Writing the basis of $\Lambda^\square$ in terms of the basis of $\Lambda^\uni$ gives the matrix
	\[
	\left( \begin{array}{ccccc} 
	1&0&0&0&0\\ 
	0&1&0&0&0\\ 
	0&{\textstyle \frac{-t}f }&{\textstyle \frac{q^2-1}f}&{\textstyle \frac \zeta{f}}&{\textstyle \frac \xi{f}}\\
	0&0&0&1&0\\ 
	0&0&0&0&2\\ 
	\end{array} \right) .
	\]
	From the last matrix and the isomorphism \eqref{eq-RelCotangAndLattices}, we find
	\[\Phi_{\lambda,R^\square/R^\uni}\cong \cO/\frac{q^2-1}{\gcd(s,t,q-1)}\cO=\cO/ ((q^2-1)\varpi^{-n_q}).\]
	
	Finding a set of generators of $\Lambda^\St$ over $\cO$ is analogous to the other two cases. In fact, out of the $8$ spanning vectors in $\cO^8$ from the $\df r_i^\St|_\lambda$, we already computed six and the other two are easy. When performing Gau\ss\ elimination, one has to distinguish two cases. \\
	Case I: $(q-1)$ divides $s$. Then a basis of $\Lambda^\St$ over $\cO$ is given by the rows of 
\begin{equation}\label{eq-Stbg-Gens}
	\left( \begin{array}{cccccccc} 
	0&0&0&1&0&0&0&0\\ 
	0&0&0&0&1&0&0&1\\ 
	0&0&0&0&0&0&f&0\\ 
	0&0&t&0&0&0&0&1-q\\ 
	f&0&f\frac{s}{q-1}&0&0&0&0&0\\ 
	\end{array} \right) ,
\end{equation}
where we note that $\gcd(t,q-1)=\gcd(s,t,q-1)=f$ in case I.\\
Case II: $s$ divides $(q-1)$. Then a basis of $\Lambda^\St$ over $\cO$ is given by the first three rows of the matrix \eqref{eq-Stbg-Gens} together with the rows of
\begin{equation}\nonumber
	\left( \begin{array}{cccccccc} 
	f\frac{q-1}s&0&f&0&0&0&0&0\\ 
	t&0&0&0&0&0&0&s\\ 
	\end{array} \right) ,
\end{equation}
where in case II we use $\gcd(s,t)=\gcd(s,t,q-1)=f$.

In both cases we have $\Lambda^\St\supset \Lambda^\uni$, and writing the basis of $\Lambda^\uni$ in terms of the basis of $\Lambda^\St$ gives the matrices

	\[
	\hbox{Case I: }\left( \begin{array}{ccccc} 
	q&0&0&0&0\\ 
	0&1&0&0&0\\ 
	0&0&1&0&0\\
	0&0&0&1&\frac{s}{1-q}\\ 
	\frac{q-1}f&0&0&0&\frac{t}f\\ 
	\end{array} \right) , \quad \hbox{Case II: }	
	\left( \begin{array}{ccccc} 
	q&0&0&0&0\\ 
	0&1&0&0&0\\ 
	0&0&1&0&0\\
	\frac{s}f&0&0&\frac{t}f&0\\ 
	0&0&0&\frac{1-q}s&1\\ 
	\end{array} \right) .
	\]
	In either case we find:
	\[\Phi_{\lambda,R^\uni/R^\St}\cong \cO/\cO{\textstyle \frac{t}f}=\cO/ (\varpi^{m_q-n_q}).\]
	From the two expressions for $\Phi_{\lambda,R^?/R^{??}}$ the claims in the proposition are immediate. We only gave the computations for $\rhobar$ unramified, but similar arguments allow one also to deduce 1 without 2 and 3 by directly comparing $\Lambda^\square$ with $\Lambda^\st$.
		\end{proof}

	


Now let $R_\infty = R_{\loc}[[x_1,\ldots,x_g]]$ and $R_\infty^{\St} = R_{\loc}^{\St}[[x_1,\ldots,x_g]]$ be as in section \ref{sec:patch} and let $\lambda:R_\infty^{\St}\to\cO$ be an augmentation, which then induces augmentations $R_q^{\St}\to\cO$ for each $q\in Q$.

\begin{corollary}\label{cor:relative Phi for R_infty} The relative cotangent space  $\Phi_{R_{\infty}/ R_\infty^\St}$ at $\lambda$ of the morphism $R_{\infty}  \ra R_\infty^\St$ is of length 
 \[ \sum_{q \in Q}(m_q+\ord_\cO(q^2-1)-2n_q)\] as an $\cO$-module.
\end{corollary}
\begin{proof}
Write $T = \left[\widehat{\bigotimes}_{\ell|N}R^{\min}_\ell\right]\widehat{\otimes}R^{\fl}_p[[x_1,\ldots,x_g]]$ so that
$R_\infty = \left[\widehat{\bigotimes}_{q|Q}R_q\right]\cotimes_\cO T$ and $R_\infty^\St = \left[\widehat{\bigotimes}_{q|Q}R_q^{\St}\right]\cotimes_\cO T$. Lemma \ref{lem:diff tensor} now gives
\begin{align*}
\cOmega_{R_\infty/\cO} &\cong \left[\bigoplus_{q|Q} R_\infty\cotimes_{R_q}\cOmega_{R_q/\cO}\right]\oplus (R_\infty\cotimes_T\cOmega_{T/\cO})&
&\text{and}&
\cOmega_{R_\infty^{\St}/\cO} &\cong \left[\bigoplus_{q|Q} R_\infty^{\St}\cotimes_{R_q^{\St}}\cOmega_{R_q^{\St}/\cO}\right]\oplus (R_\infty^{\St}\cotimes_T\cOmega_{T/\cO}),
\end{align*}
and so $\Phi_{\lambda,R_\infty}\cong \bigoplus_{q|Q}\Phi_{\lambda,R_q}\oplus \Phi_{\lambda,T}$ and $\Phi_{\lambda,R_\infty^{\St}}\cong \bigoplus_{q|Q}\Phi_{\lambda,R_q^{\St}}\oplus \Phi_{\lambda,T}$. But this implies $\Phi_{\lambda, R_\infty/R_\infty^\St}=\bigoplus_{q| Q} \Phi_{\lambda, R_q/R_q^\St}$, so the claim follows by Proposition \ref{key-local-comp}.
\end{proof}

\subsection{Presentation of $R_q^\St$}
	The following lemma gives a coordinate free explanation of how to obtain the relations  in \cite[Proposition~5.8]{Shotton} for $R_q^\St$ that were used in the above proof.

\begin{lemma}\label{Lem:OnShotton}Let $R$ be a ring of characteristic zero and let $q>1$ be an integer. Let $A,B$ be $2\times2$-matrices over $R$. Then the following hold:
\begin{enumerate}
\item 
Suppose $A$ and $B$ have characteristic polynomials $\chi_A=(T-q)(T-1)$ and $\chi_B=(T-1)^2$ in $R[T]$. Then $AB=B^qA$ if and only if $(B-I)(A-I)=0$.
\item For $A:=\Mat{q+a}{b}{c}{1-e}$ and $B:=\Mat{1+\alpha}{\beta}{\gamma}{1+\delta}$ the following are equivalent:
\begin{enumerate}
\item $\chi_A=(T-q)(T-1)$ and $\chi_B=(T-1)^2$ and $(B-I)(A-I)=0$.
\item $e=a$, $\delta=-\alpha$ and all $2\times 2$-minors of the following matrix vanish
\begin{equation}\label{MinorsOfWhich}
\left(
\begin{array}{cccc}
 \alpha & \beta  & (q\!-\!1\!+\!a)  & b \\
\gamma  & -\alpha  & c  & -a 
\end{array}
\right).
\end{equation}
\end{enumerate}
\end{enumerate}
\end{lemma}
\begin{proof}
(a) Let $N:=B-I$. The conditions on $\chi_A$ and $\chi_B$ are equivalent to $\tr A=q+1$, $\det A=q$, $\tr N=0$ and $\det N=0$. The latter also implies $N^2=0$, so that $B^qA=(I+N)^qA=(I+qN)A$ and $ AB=A(I+N)$. It follows that $AB=B^qA$ is equivalent to $AN=qNA$ and hence to $ (A-qI)N=qN(A-I)$.
Thus we need to show that
\begin{equation}\label{eq;St-Eq}
 (A-qI)N=qN(A-I) \Longleftrightarrow N(A-I)=0.
\end{equation}
Set $C:=(A-qI)N$ and recall that the main involution is given by  the linear map
$$D:=\Mat{w}{x}{y}{z}\mapsto D^\iota:=\Mat{0}{1}{-1}{0}D^t \Mat{0}{1}{-1}{0}^{-1}=\Mat{z}{-x}{-y}{w}.$$
It satisfies $D+D^\iota=\tr D\cdot I$ and $(DD')^\iota=(D')^\iota D^\iota$. Hence $\tr A=q-1$ and $\tr N=0$ yield $(A-I)^\iota=(qI-A)$ and $N^\iota=-N$ and the anti commutativity implies that $C^\iota=N(A-I)$, so that the left hand side of \eqref{eq;St-Eq} is equivalent to $C=qC^\iota$.

Now if $N(A-I)=0$, then $C^\iota=0$ and hence $C=qC^\iota=0=C^\iota$. Conversely if $C=qC^\iota$, then $(q+1)C=qC+qC^\iota=q\tr(C)$, and so $C$ must be a scalar matrix because $q+1$ is a non-zero divisor. But then $C^\iota=C$ and thus $C^\iota=qC$ and $q-1$ non-zero imply that $C=0$ and hence $N(A-I)=C^\iota=0$.

(b) The equations $e=a$ and $d=a$ are equivalent to $\tr (A-I)=q-1$ and $\tr N=0$. The six $2\times 2$-minors in (ii) express the following conditions on $A$ and $N$: the left minor means $\det(N)=0$, the right minor that $\det(A-I)=0$, and the remaining minors express that $N(A-I)=0$. These are precisely the conditions obtained in (a).
\end{proof}

\begin{remark}\label{Rem-DetailsOnRst-Using7.12}
Denote by $\widetilde R_q^\St$ the quotient of $\cR=\cO[[a,b,c,e,\alpha,\beta,\gamma,\delta]]$ by the ideal $\cI_q^\St=(r_1^\St,\ldots,r_8^\St)$ from the proof of Proposition~\ref{key-local-comp} -- except for a slight change of notation, this is the ring described by the relations coming from Lemma~\ref{Lem:OnShotton}. To be able to identify this ring with $R_q^\St$ we need to show that it is reduced and flat over $\cO$, because Lemma~\ref{Lem:OnShotton} a priori only gives equations cutting out the closed points on the generic fiber. We now give a proof. The arguments have much overlap with those used in \cite{Shotton}.

As the $r_i^\St$, $i=1,\ldots,6$, are the $2\times 2$-minors of \eqref{eq-MinorsMatrix}, the ring $\cO[\alpha,\beta,\gamma,\delta,a,b,c,e]/(r^\St_1,\ldots,r^\St_6)$ is determinantal. By \cite[Section~2]{Bruns-Vetter}, this quotient ring is a normal Cohen-Macaulay domain of dimension $6$, and hence its completion, say $\cT$, at the maximal ideal $(\alpha,\beta,\gamma,\delta,a,b,c,e)$ is Cohen-Macaulay of dimension~$6$. A short calculation shows that the quotient of $\cT$ modulo the $6$-term sequence $r_7^\St,r_8^\St,\varpi,\gamma-\beta,c+b,\beta+b$ is finite, and hence the sequence is regular. It follows that $\widetilde R_q^\St=\cT/(r_7^\St,r_8^\St)$ is Cohen-Macaulay and flat over $\cO$ of relative dimension~$3$. To prove its reducedness, we may thus invert $p$, and rely on results on $R_q^\square[\frac1p]$. 

Because  $\widetilde R_q^\St$ is flat over $\cO$, it has characteristic zero, and hence Lemma~\ref{Lem:OnShotton} shows that the matrices $A$, $B$ from the proof of Proposition~\ref{key-local-comp} satisfy $AB=B^qA$. The universal property of  $R_q^\square$ thus yields a surjective homomorphism $R_q^\square\to \widetilde R_q^\St$ in $\CNL$.
By the previous paragraph, the target of $R_q^\square[\frac1p] \to  \widetilde R_q^\St[\frac1p]$ is equidimensional of dimension~$3$.~Adapting \cite[Section 1.3]{BLGGT14} to a fixed determinant situation shows that $\Spec R_q^\square[\frac1p]$ contains a Zariski dense set of \emph{robustly smooth} points and hence all irreducible components of  $\Spec R_q^\square[\frac1p]$ are generically formally smooth of dimension~$3$. We deduce that $\widetilde R_q^\St[\frac1p]$ is reduced.
\end{remark}

\begin{remark}\label{Rem-CompareToShotton}
Let us also give the coordinate change between the relations found here for $R_q^\St\cong \widetilde R_q^\St=\cR/(r_1^\St,\ldots,r_8^\St)$, and those given in \cite[proof of Proposition~5.8, p.~1467]{Shotton}: The ring $\widetilde R^\St_q$ is to be compared with the ring $\cS:=\cB/(J_0+J_1+(T-q+1))$ from op.cit.~where $\cB=\cO[[X_1,\ldots,X_4,Y_1,\ldots,Y_4,T]]$, $J_0$ are the $2\time 2$-minors of the $2\times4$-matrix with rows $(X_1 \ \ldots X_4)$ and $(Y_1\ \ldots \ Y_4)$ and $J_1=(X_1+Y_2,Y_3-X_4+2\frac{q-1}{q+1})$; note that in the definition of $J_1$ in \cite{Shotton} the variables $Y_3$ and $X_3$ as well as $X_4$ and $Y_4$ are mixed up. By going through the proof of  \cite[Proposition~5.8]{Shotton}, the change of variables from $\cS$ is $X_1=\alpha$, $X_2=\beta$, $Y_1=\gamma$, $Y_2=-\alpha$, $X_3=\frac{-2b}{q+1}$, $X_4=\frac{a+q-1}{q+1}$, $Y_3=\frac{2a}{q+1}$, $Y_4=\frac{2c}{q+1}$ and $T=q-1$. It is immediate to see that the relations here and there agree. 
\end{remark}

\subsection{Patching and global cotangent spaces}

 Consider a global augmentation $\lambda: R^\St \ra \cO$ arising from a newform in $ f\in S_2(\Gamma_0(NQ))$ which is new at $Q$, and the augmentations it induces by pull back to other global deformation rings and local deformation rings, which by abuse of notation we denote by $\lambda$ again.  We assume that the residual representation  $\rhobar$ of $\rho_{f,\iota}$ satisfies the conditions of \S \ref{sec:trivial primes}. Further  if $q$ is  $-1$ mod $p$, such that $\rhobar$ is unramified at $q$ and $q \in Q$, and  $\rho_{f,\iota}|_{G_q}$ is  \[  \left( \begin{array}{cc} \varepsilon & \ast \\ 0 & 1 \end{array} \right)\]   up to twist by an unramified character of order dividing 2, then we choose $\beta_q$ of \S \ref{sec_deformation_theory}  to be $1$ or $-1$ accordingly.

The main result of this section is:

\begin{theorem}\label{mc1}
   
There is an isomorphism $\Phi_{\lambda,R_\infty/R_\infty^{\St} } \xrightarrow{\sim}\Phi_{\lambda,R/R^{\St}}$.
\end{theorem}

  \begin{proof}
  
We begin by observing that  Theorem \ref{thm:patching} shows that there exists rings $R_\infty$, $R_\infty^\St$ and $S_\infty = \cO[[y_1,\ldots,y_d]]$
with $\dim R_\infty = \dim R_\infty^\St = \dim S_\infty = d+1$, and a commutative diagram
\begin{center}
	\begin{tikzpicture}
	\node(11) at (2,2) {$S_\infty$};
	\node(21) at (4,2) {$R_\infty$};
	\node(31) at (6,2) {$R$};
	
	\node(10) at (2,0) {$S_\infty$};
	\node(20) at (4,0) {$R_\infty^\St$};
	\node(30) at (6,0) {$R^\St$};
	\node(40) at (8,0) {$\cO$};

	\draw[-latex] (11)--(10) node[midway,right]{$=$};
	\draw[->>] (21)--(20);
	\draw[->>] (31)--(30);

	\draw[right hook-latex] (11)--(21);
	\draw[->>] (21)--(31);
	
	\draw[right hook-latex] (10)--(20);
	\draw[->>] (20)--(30);
	\draw[->>] (30)--(40) node[midway,above]{$\lambda$};
	\end{tikzpicture}
\end{center}
where $R = R_\infty\otimes_{S_\infty}\cO = R_\infty/(y_1,\ldots,y_d)$ and $R^\St = R_\infty^\St\otimes_{S_\infty}\cO = R_\infty^\St/(y_1,\ldots,y_d)$. We will also use $\lambda$ to denote the maps $\lambda:R\onto \cO$, $\lambda:R_\infty\onto \cO$ and $\lambda:R_\infty^\St\onto \cO$ induced by $\lambda:R^{\St}\onto \cO$.

Now by Lemma \ref{lem:diff A>B>C} the local ring maps $\cO\to S_\infty\to R_\infty$ give rise to an exact sequence 
\[\cOmega_{S_\infty/\cO}\otimes_{S_\infty} R_\infty\to \cOmega_{R_\infty/\cO}\to \cOmega_{R_\infty/S_\infty}\to 0.\]
Applying $\underline{\quad}\otimes_{S_\infty}\cO = \underline{\quad}\otimes_{R_\infty}R$ to this sequence gives
\[\cOmega_{S_\infty/\cO}\otimes_{S_\infty} R\to \cOmega_{R_\infty/\cO}\otimes_{R_\infty}R\to \cOmega_{R_\infty/S_\infty}\otimes_{S_\infty}\cO\to 0.\]
Now we have $\cOmega_{S_\infty/\cO} = \Omega_{\cO[y_1,\ldots,y_d]/\cO}\otimes_{\cO[y_1,\ldots,y_d]}S_\infty \cong \cO[y_1,\ldots,y_d]^d\otimes_{\cO[y_1,\ldots,y_d]}S_\infty \cong S_\infty^d$. Also by Lemma \ref{lem:diff finite} and \cite[\href{https://stacks.math.columbia.edu/tag/00RV}{Lemma 00RV}]{stacks-project},
\[\cOmega_{R_\infty/S_\infty}\otimes_{S_\infty}\cO\cong \Omega_{R_\infty/S_\infty}\otimes_{S_\infty}\cO \cong \Omega_{R/\cO},\]
so the above exact sequence becomes
\[R^d\to \cOmega_{R_\infty/\cO}\otimes_{R_\infty}R\to \Omega_{R/\cO}\to 0.\]
Finally applying $\underline{\quad}\otimes_{\lambda}\cO$ to this, and noting that $\left(\cOmega_{R_\infty/\cO}\otimes_{R_\infty}R\right)\otimes_{\lambda}\cO = \cOmega_{R_\infty/\cO}\otimes_{\lambda}\cO =\Phi_{\lambda,R_\infty}$, we get the exact sequence
\[\cO^d\to \Phi_{\lambda,R_\infty}\to \Phi_{\lambda,R}\to 0.\]
By an analogous argument with $R^\St_\infty$ and $R^\St$ in place of $R_\infty$ and $R$, we now get a commutative diagram 
\begin{center}
	\begin{tikzpicture}
	\node(11) at (2,2) {$\cO^d$};
	\node(21) at (4,2) {$\Phi_{\lambda,R_\infty}$};
	\node(31) at (7,2) {$\Phi_{\lambda,R}$};
	\node(41) at (9,2) {$0$};
	
	\node(10) at (2,0) {$\cO^d$};
	\node(20) at (4,0) {$\Phi_{\lambda,R_\infty^\St}$};
	\node(30) at (7,0) {$\Phi_{\lambda,R^\St}$};
	\node(40) at (9,0) {$0$};
	
	\node(22) at (4,4) {$\Phi_{\lambda,R_\infty/R_\infty^\St}$};
	\node(32) at (7,4) {$\Phi_{\lambda,R/R^\St}$};
	
	\node(23) at (4,6) {$0$};
	\node(33) at (7,6) {$0$};
	
	\draw[-latex] (11)--(10) node[midway,right]{$=$};
	\draw[-latex] (21)--(20);
	\draw[-latex] (31)--(30);
	
	\draw[-latex] (23)--(22);
	\draw[-latex] (22)--(21);
	\draw[-latex] (33)--(32);
	\draw[-latex] (32)--(31);
	
	\draw[-latex] (22)--(32);
	
	\draw[-latex] (11)--(21);
	\draw[-latex] (21)--(31);
	\draw[-latex] (31)--(41);
	
	\draw[-latex] (10)--(20);
	\draw[-latex] (20)--(30);
	\draw[-latex] (30)--(40);
	\end{tikzpicture}
\end{center}
with exact rows and columns.

Recall from Proposition \ref{prop:R_v-Atq} that for any $q\in Q$ there is a finitely generated $\cO$-algebra $\cR^{\St}_q$ for which $R_q^{\St}$ is the completion of $\cR^{\St}_q$ at a maximal ideal. More specifically, we may take $\cR_q^\St = \cO[\alpha,\beta,\gamma,a,b,c]/I_q$, where $I_q$ is generated by the $2\times 2$ minors of the matrix in \eqref{MinorsOfWhich} if $q$ is a trivial prime, and $\cR_q^{\St} = \cO[a,b,c]$ otherwise. Now let
\[\cR_\infty^\St = \left[\bigotimes_{q|Q}\cR_q^\St\right]\otimes\left[\bigotimes_{\ell|N}\cO[X_1,X_2,X_3]\right]\otimes\cO[X_1,X_2,X_3,X_4]\otimes\cO[x_1,\ldots,x_g]\] so that Proposition \ref{prop:R_v-Atq} gives that $R_\infty^{\st}$ is the completion of $\cR_\infty^{\st}$ at a maximal ideal $\ffrm$.

Similarly, by the computations in \cite[Section 5]{Shotton}, there is also a natural choice of finitely generated $\cO$-algebra $\cR_\infty$ such that $\cR_\infty^{\st}$ is a quotient of $\cR_\infty$ (making $\Spec \cR_\infty^{\st}$ an irreducible component of $\Spec \cR_\infty$) and $R_\infty$ is the completion of $\cR_\infty$ at the same maximal ideal $\ffrm$.

Now notice that
\[\Phi_{\lambda,R_\infty}\otimes_{\cO}E = \Omega_{\cR_\infty/\cO}\otimes_{\lambda}E 
= \left(\Omega_{\cR_\infty/\cO}\otimes_{\cR_\infty}\cR_\infty[1/\varpi]\right)\otimes_{\lambda}E
= \Omega_{\cR_\infty[1/\varpi]/\cO}\otimes_{\lambda}E = \Omega_{\cR_\infty[1/\varpi]/E}\otimes_{\lambda} E,\]
by Lemma \ref{lem:diff aug} and \cite[\href{https://stacks.math.columbia.edu/tag/00RT}{Lemma 00RT}]{stacks-project}, and similarly 
\begin{align*}
\Phi_{\lambda,R_\infty^\St}\otimes_{\cO}E &= \Omega_{\cR_\infty^\St[1/\varpi]/E}\otimes_{\lambda} E\\
\Phi_{\lambda,R}\otimes_{\cO}E &= \Omega_{R[1/\varpi]/E}\otimes_{\lambda} E\\
\Phi_{\lambda,R^\St}\otimes_{\cO}E &= \Omega_{R^\St[1/\varpi]/E}\otimes_{\lambda} E
\end{align*}
Now observe that $R$ and $R^\St$ are finite over $\cO$ and reduced (by Theorem \ref{thm:patching}). This implies that $R[1/\varpi]$ and $R^\St[1/\varpi]$ are both finite direct products of finite extensions of $E$. It follows that $\Phi_{\lambda,R}\otimes_{\cO}E = \Phi_{\lambda,R^\St}\otimes_{\cO}E = 0$, and so $\Phi_{\lambda,R}$ and $\Phi_{\lambda,R^\St}$ are both (finitely generated) torsion $\cO$-modules.

As computed in \cite[Proposition 5.8]{Shotton}, for each $q$, $\Spec \cR_q^\St[1/\varpi]$ is a smooth, irreducible variety of dimension $3$ over $E$ (specifically this follows from a computation identical to the computation that $\Spec(R^\square(\rhobar,\tau_{1,\mathrm{ns}})\otimes E)$ is formally smooth in \cite[Proposition 5.8]{Shotton}). It follows that $\Spec\cR_\infty^\St[1/\varpi]$ is a smooth, irreducible variety of dimension $d = \dim_{\cO}R_\infty^\St$ over $E$. By \cite[\href{https://stacks.math.columbia.edu/tag/02G1}{Lemma 02G1}]{stacks-project}, it follows that $\Omega_{\cR_\infty^\St[1/\varpi]/E}$ is locally free of rank $d$ over $\cR_\infty^\St[1/\varpi]$, and so $\Phi_{\lambda,R_\infty}\otimes_{\cO}E = \Omega_{\cR_\infty^\St[1/\varpi]/E}\otimes_{\lambda} E\cong E^d$.

Recall that $\Spec\cR_\infty^\St[1/\varpi]$ is an irreducible component of $\Spec\cR_\infty[1/\varpi]$. Treating $\lambda:\cR_\infty\twoheadrightarrow \cR_\infty^\St\xrightarrow{\lambda} \cO\hookrightarrow E$ as an $E$ point of $\Spec\cR_\infty[1/\varpi]$, we have that $\lambda$ is not contained in any irreducible components of $\Spec\cR_\infty[1/\varpi]$, other than $\Spec\cR_\infty^\St[1/\varpi]$ (as $\Spec \cR_\infty[1/\varpi]$ must be smooth at $\lambda$, by Lemma \ref{lem:R_loc smooth}). It follows from this that
\[\Phi_{\lambda,R_\infty}\otimes_{\cO}E \cong \Omega_{\cR_\infty[1/\varpi]/E}\otimes_{\lambda} E\cong \Omega_{\cR_\infty^\St[1/\varpi]/E}\otimes_{\lambda} E\cong E^d\] 
(the second equality follows from \cite[\href{https://stacks.math.columbia.edu/tag/02G1}{Lemma 02G1}]{stacks-project} again, after noting that $\cR_\infty[1/\varpi]_{(J)}\cong \cR_\infty^\St[1/\varpi]_{(J)}$, where $J = \ker(\cR_\infty[1/\varpi]\xrightarrow{\lambda}E)$, and so $\left(\Omega_{\cR_\infty[1/\varpi]/E}\right)_{(J)}\cong\left(\Omega_{\cR_\infty^\St[1/\varpi]/E}\right)_{(J)}$).

Thus $\Phi_{\lambda,R_\infty}$ and $\Phi_{\lambda,R_\infty^\St}$ both have rank $d$ as $\cO$-modules.

But now $\Phi_{\lambda,R_\infty}$ and $\cO^d$ have the same rank as $\cO$-modules and $\Phi_{\lambda,R}=\coker(\cO^d\to \Phi_{\lambda,R_\infty})$ is a torsion $\cO$-module. Hence $\ker(\cO^d\to \Phi_{\lambda,R_\infty})$ is a torsion $\cO$-module, and is thus $0$. Therefore the sequence
\[0\to\cO^d\to \Phi_{\lambda,R_\infty}\to \Phi_{\lambda,R}\to 0\]
from above is exact. By identical reasoning,
\[0\to\cO^d\to \Phi_{\lambda,R_\infty^\St}\to \Phi_{\lambda,R^\St}\to 0\]
is also exact. The above commutative diagram, together with the snake lemma, now gives that the natural map $\Phi_{\lambda,R_\infty/R_\infty^\St}\to \Phi_{\lambda,R/R^\St}$ is an isomorphism.
\end{proof}

\begin{corollary}\label{cor:main}
\[\ell_\cO({\Phi_{\lambda,R/R^{\St}}})= \ell_\cO({\Phi_{\lambda,\T/\T^{\St}}})=\sum_{q \in Q}( m_q+\ord_\cO(q^2-1)-2n_q).\] 
\end{corollary}

\begin{proof}
This is a consequence of Theorem \ref{thm:patching}, Corollary \ref{cor:relative Phi for R_infty} and Theorem \ref{mc1}.
\end{proof}

\begin{remark}

The above result proves that the size of the  relative global cotangent space at an augmentation of a global deformation ring  is the same as the   size of the  relative local  cotangent space.  We prove this  using a combination of level raising results (as used in the proof of Lemma \ref{lem:M(NQ^2) support}) and patching.
Similar  local-global results are proved in \cite{Wiles} when the change in the local deformation condition is from unramified to unrestricted ramification.    The arguments are different caused by the fact that  the rings  $R^\st=\T^\st$  are not complete intersections and hence  have non-trivial Wiles defect with respect to our fixed augmentation $\lambda$  (while in \cite{Wiles} the minimal deformation and Hecke rings are c.i. and hence satisfy Wiles numerical criterion). Wiles  computes an upper bound on change of cotangent space (when allowing ramification at a prime), and a matching lower bound on change in congruence module,  which together implies an exact formula  for change of cotangent space and congruence module.  In contrast, in our work  here to compute the Wiles defect for $R^\st,\T^\st$, we have to  compute the exact change of cotangent spaces (between $R$ and $R^{\St}$) and congruence modules (between $\T$ and $\T^\st$) independently of each other.

\end{remark}

\section{Relating congruence modules}\label{sec:congruences}

We recall the main result of \cite{Manning} explicitly determines the structure of $M_\infty^{\st}$ as an $R_\infty^{\st}$-module (or at least of $M_\infty^{\st}/\varpi$ as an $R_\infty^{\st}/\varpi$-module). In particular, provided that $Q$ is divisible by at least one trivial prime for $\rhobar$, $M_\infty^{\st}$ is not free over $R_\infty^{\st}$. We will not need to use the full description of $M_\infty^{\st}$ in this paper, however we will use the following consequence of these computations (see \cite{Manning} Theorem 1.2, or more specifically Propositions 4.13 and 4.14):

\begin{theorem}\label{thm:sat}
In both the definite and indefinite cases, the map $R^{\st}\to \End_{R^{\st}}(M^{\st}(N))$ is an isomorphism.
\end{theorem}

\begin{remark}
There is a small subtlety here in that \cite{Manning} excludes the case where $D$ ramifies at a prime which is $-1\pmod{p}$, whereas we have not excluded this case. Fortunately the restriction in loc. cit. is purely for the sake of convenience, so this does not present an issue. Under the notation of this paper we have $R_q^{\st(\beta_q)}\cong \cO[[A,B,C]]$ at all primes for which $q\equiv -1\pmod{p}$, which means that the ring $R_\infty^{\st}$ considered in this paper will coincide with the ring $R_\infty$ considered in loc. cit. in all cases. Thus the proof given in loc. cit. will establish Theorem \ref{thm:sat} in all cases considered in this paper, without modifications.
\end{remark}

\begin{corollary}\label{cor:sat}
For an augmentation $\lambda:R^{\st}\to \cO$ as above,  we have $\card{\cng{\lambda}(R^{\st})}= \card{\cng{\lambda}(M^{\st}(N))}$, and $\delta_{\lambda,R^{\st}}(R^{\st}) = \delta_{\lambda,R^{\st}}(M^{\st}(N)) = \delta_{\lambda,R^{\st}}(S^Q(\Gamma_0^Q(N))_{\ffrm_Q'})$. 
\end{corollary}

\begin{proof}

This follows upon combining  Theorem \ref{thm:surjective}  and Theorem \ref{thm:sat}.
\end{proof}

Thus to prove our main result, it will suffice to compute $\delta_{\lambda,\T^{\st}}(S^Q(\Gamma_0^Q(N))_{\ffrm_Q'})$, which will be our goal in the remainder of this paper.

\section{Computation of change of the $\eta$-invariant}\label{sec:ribet}

We  put ourselves in the setting of \S \ref{hecke}. Consider a quaternion algebra $D$ over $\Q$  that is ramified at a  finite set of primes $Q$ (and hence  is definite if $|Q|$ is odd, and indefinite if $|Q|$ is even). We have the augmentation  $\lambda_f: \T^\st \ra \cO$ which corresponds (by Jacquet--Langlands)  to a newform $f \in S_2(\Gamma_0(NQ))$,  with residual representation $\rhobar$. Recall that $N|N(\rhobar)|NQ$.  We  recall from \S \ref{hecke}  that $\T^{\St}$ acts  faithfully on $H^1(X_0^Q(N),\cO)_{\ffrm_Q'}$, with $\ffrm_Q'$ characterized by the property that it is the maximal ideal of the full Hecke algebra  acting on   $H^1(X_0^Q(N),\cO)$ that contains  the kernel of the augmentation arising from $f$.  We also have  from \S \ref{hecke} the Hecke algebra $\T$   that acts faithfully  on $S(NQ^2)_{\ffrm_Q}=H^1(X_0(NQ^2), \cO)_{\ffrm_Q}$, the oldform $f^Q$ of level $NQ^2$ that gives rise to the maximal ideal $\ffrm_Q$ of the full Hecke algebra acting on  $S(NQ^2)=H^1(X_0(NQ^2), \cO)$,  and $S(NQ^2)_{\ffrm_Q}^*=M(NQ^2)^{\oplus 2}$. 

Let $\cA_f$ stand for the isogeny class of the abelian variety $A_f$ (which is an optimal quotient of $J_0(NQ)$). The residual representations arising from the class $\cA_f$ with respect to the fixed embedding $K_f \hookrightarrow  \overline \Q_p$ are all isomorphic to our fixed absolutely irreducible $\rhobar$. 
We have the following result which is deduced from  the methods of  \cite{RiTa}.

 \begin{proposition}\label{change-of-eta}

  Assume that $Q$ contains a prime  $t$ that divides $N(\rhobar)$ and we are in the  indefinite case  (i.e.,  that $Q$ is of even  cardinality).
 We have the equality of lengths of $\cO$-modules:
  \[\ell_\cO(\Psi_\lambda(M(NQ^2)))= \ell_\cO(\Psi_\lambda(M^{\st}(N))) + \sum_{q \in Q} (m_q+\ord_\cO(q^2-1)).\]

  \end{proposition}

  \begin{proof}
  
  Let $Q=\{q_1,q_2,\cdots,q_{2r}\}$ with $q_1=t$ such that $t|N(\rhobar)$, and for $0 \leq s \leq r$,  set  $Q_s=\{q_1,\cdots,q_{2r-2s}\}$ and   $J_s$ to be the Jacobian of  $X_s=X_0^{Q_s}(Nq_{2r-2s+1}\cdots q_{2r})$.  
  
  To study the congruence modules of $H^1(X_s,\cO)_\ffrm$ using the work of section \ref{sec:congruence} we will need a Hecke equivariant perfect pairing on $H^1(X_s,\cO)_\ffrm$. In general the perfect pairing $\langle\ ,\ \rangle_s':H^1(X_s,\cO)\times H^1(X_s,\cO)\to \cO$ arising from Poincar\'e duality will not be Hecke equivariant (instead for any double coset operator $T$ it will satisfy $\langle T^*x,y\rangle_s' = \langle x,T_*y\rangle_s'$ where $T^*$ and $T_*$ are induced respectively by Picard and Albanese functoriality, respectively). As in Lemma 5.5 of \cite{HelmShimCurve} one may modify this to obtain a Hecke equivariant perfect pairing $\langle\ ,\ \rangle_s: H^1(X_s,\cO)\times H^1(X_s,\cO)\to \cO$ defined by $\langle x,y\rangle_s = \langle x,wy\rangle_s'$, where $w$ is an appropriately chosen Atkin--Lehner involution. We will also use $\langle\ ,\ \rangle_s$ to denote the localized perfect pairing $\langle\ ,\ \rangle_s: H^1(X_s,\cO)_\ffrm\times H^1(X_s,\cO)_\ffrm\to \cO$. There is an analogous Hecke equivariant perfect pairing $\langle \ , \ \rangle:H^1(X_0(NQ^2),\cO)_{\ffrm_Q} \times  H^1(X_0(NQ^2),\cO)_{\ffrm_Q} \to \cO$.  The $\cO$-modules $H^1(X_0,\cO)[\ker(\lambda_f)],  H^1(X_r,\cO)[\ker(\lambda_f)],
  H^1(X_0(NQ^2),\cO)[\ker(\lambda_{f^Q})] $ are all free of rank 2, and we choose  $\cO$-bases  $\{A_0,B_0\}$, $\{A_r,B_r\}$ and $\{A,B\}$ respectively.

  \begin{itemize}

  \item Step 1:  We prove the equality $(\langle A_r,B_r\rangle_r)= (\prod_{q \in Q} \varpi^{m_q}) (\langle A_0,B_0\rangle_0)$ of ideals of $\cO$.    Thus is done by using \cite{RiTa} and its extension in \cite{KhareRT}.
  
\item Step 2:    We prove the equality $(\langle A,B\rangle)= (\prod_{q \in Q} (q^2-1)) ( \langle A_r,B_r\rangle_r)$ of ideals   of $\cO$.
This uses the arguments of  \cite[\S 2]{Wiles} in particular \cite[Lemma 2.5]{Wiles} (see also \cite[\S 4]{DiRi}). \end{itemize}

We  go into more of the details of the two steps.

\noindent{\bf Step 1:}   In \cite{RiTa} the authors  study  the change in degree of parametrizations of  (an isogeny class of)  an elliptic curve by Shimura curves as the Shimura curves vary. We use the  extension of \cite{RiTa} in \cite{KhareRT}  to the case of parametrisations of $A_f$, when $A_f$  is not an elliptic curve. Thus we first briefly  recall some of the work of \cite{KhareRT}.

  Let $\xi_s:J_s\rightarrow A_s$  with $A_s \in \cA_f$
be an optimal quotient:  the map $\xi_s$ is equivariant for the Hecke action
and its kernel is connected. 
Analogously we have the optimal
quotient $\xi_{s+1}:J_{s+1} 
\rightarrow A_{s+1}$. (The optimal quotients that we consider
are all isogenous to each other via isogenies defined over ${\bf Q}$ and
which are Hecke equivariant.) We consider  the maps ${\rm Ta}_p(J_s) 
\rightarrow {\rm Ta}_p(A_s)$
and ${\rm Ta}_p(J_{s+1})
\rightarrow {\rm Ta}_p(A_{s+1})$ that $\xi_s$ and $\xi_{s+1}$ induce on 
the corresponding Tate
modules. 
To ease the notation, we set $J:=J_s$
and $J':=J_{s+1}$, $A_s:=A$ and $A_{s+1}:=A'$,
$\xi:=\xi_s$
and $\xi':=\xi_{s+1}$ and  also set $q_{2(r-s)}=q$ and $q_{2(r-s)-1}=q'$.

We have induced maps $\xi'^*:{\rm Ta}_p({(A')}^d)_{{\ffrm},\cO}
\rightarrow {\rm Ta}_p(J')_{{\ffrm},\cO}$,
and $\xi'_*:{\rm Ta}_p(J')_{{\ffrm},\cO} \rightarrow {\rm Ta}_p(A')_{{{\ffrm}},\cO}$ (injective
and surjective respectively by optimality, and their analogs for
$\xi$), where for instance ${\rm Ta}_p({(A')}^d)_{{\ffrm},\cO}:=({\rm
Ta}_p({(A')}^d) \otimes_{{\bf
Z}_p} \cO)_{\ffrm}$ is the localisation at $\ffrm$ of $p$-adic Tate module 
of the dual abelian variety $(A')^d:={\rm Pic}^0(A')$ of $A'$ tensored 
with $\cO$, and the other symbols are defined analogously. Observe
that ${\rm Ta}_p((A')^d)_{{{\sf
m}},\cO}$ 
is free of rank 2 over $\cO$. The module
${\rm Ta}_p((A')^d)_{{{\ffrm}},\cO}$ is
(non-canonically) isomorphic to ${\rm Ta}_p(A')_{{{\ffrm}},\cO}$ 
as a $\cO[G_{\bf Q}]$ module by \cite{Carayol2} using the irreducibility of the residual representation $\rhobar$. We identify these by choosing any such isomorphism: the
ambiguity is only up to elements of $\cO^*$ which is immaterial in the
calculations below. 
Analogously we have the maps induced by $\xi$ 
involving a choice of a $\cO[G_{\bf Q}]$-isomorphism between localisations at $\ffrm$ of 
the Tate modules of $A$ and its dual. 
The maps $\xi'_*\xi'^*$ and  $\xi_*\xi^*$
of ${\rm Ta}_p(A')_{{{\ffrm}},\cO}$ and ${\rm Ta}_p(A)_{{\ffrm},\cO}$ commute with the
$\cO[G_{\bf Q}]$-action 
and by the irreducibility of this action
can be regarded as given by multiplication by elements of $\cO$. We denote the
corresponding ideals of $\cO$ by $(\xi'_*\xi'^*)$ and  $(\xi_*\xi^*)$.

 The method of \cite{RiTa} and \cite[Theorem 4]{KhareRT}  gives the key relation: \[(\xi'_*\xi'^*)=(\omega^{m_{q_{2r-2s}}+ m_{q_{2r-2s-1}}})(\xi_*\xi^*), \] for $0  \leq s \leq r-1$, where recall from Definition \ref{invariants} that  for $q \in Q$, $m_q$ is defined to be the largest integer such that $\rho_f|_{I_q}$ is the identity   mod $\omega^{m_q}$.   We justify  how this result  follows  from  our hypotheses. Lemma  2 and Corollary on pg. 11113 of \cite{RiTa}, and \cite[Proposition 3]{KhareRT},  using the vanishing of $\Phi_t(A)_{\cO}$,  the component group at $t$ of $A$ which is a consequence of $t|N(\rhobar)$, and the fact  that $t$ divides the discriminant of the quaternion algebra  from which $J$ arises,   imply that  the  map $\xi_*:\phi_{q'}(J)_{{\ffrm},\cO} \rightarrow
           \phi_{q'}(A)_{{\ffrm},\cO}$ induced by $\xi$, with $\phi_{q'}(J)_{\ffrm,\cO}, \phi_{q'}(A)_{\ffrm,\cO}$ component groups  at $q'$ of $J$ and $A$ localized at $\ffrm$, is surjective. Using this and Ribet exact sequences \cite[Equation 2, page 209]{KhareRT}, the asserted relation is deduced in \cite[Theorem 4]{KhareRT}.

This relation may be applied to compute change of congruence modules as follows. 
     Let $A_{s},B_{s}$ be  a basis of the free $\cO$-module
  $H^1(X_s,\cO)_\ffrm[\ker \lambda]$ of rank 2 for $0 \leq s \leq r$. 
    Then  $(\langle A_s, B_s\rangle_s)=  ({\xi_s}_*\xi_s^*)$ as $(\langle  \xi_s^*(x), \xi_s^*(y)\rangle_s) =(\langle x, {\xi_s}_*\xi_s^*(y)\rangle)$ where $\{x,y\}$ is an $\cO$-basis of
${\rm Ta}_p(A)_{\ffrm,\cO}$, and  $(\langle x,y\rangle)$ is $\cO$ where we use the perfect pairing on  ${\rm Ta}_p(A)_{\ffrm,\cO}$ induced by the perfect pairings above. Using  \cite[Theorem 4]{KhareRT}, applied successively for $0 \leq s\leq r-1$,  we deduce that $(\langle A_r, B_r \rangle_r)=(\omega^{ \Sigma_{q \in Q} m_q} )  (\langle A_0, B_0 \rangle_0)$.

     \vskip .5cm
     
     \noindent{\bf Step 2:}  Let $q \in Q$ and  consider the  usual degeneracy map \[(\delta_0,\delta_1): H^1(X_0(NQ),\cO)^2 \to H^1(X_0(NQq),\cO)\]  where $\delta_i$ is induced by the map of the upper half plane which sends $z \to q^iz$.  Wiles's  result
    \cite[Lemma 2.5]{Wiles} shows  that the torsion in the cokernel of $(\delta_0,\delta_1)$ is supported only at  Eisenstein maximal ideals  (i.e., the corresponding residual Galois representations are reducible) of the Hecke algebra acting on  $H^1(X_0(NQq),\cO)$. 
     
      We  claim  that the map
     $\beta: H^1(X_0(NQ),\cO) \to H^1(X_0(NQq),\cO)$ given by $\delta_0-q^{-1}U_q\delta_1$ induces an isomorphism  (that we denote by the same symbol) $\beta:H^1(X_0(NQ),\cO)[\ker(\lambda_f)] \to  H^1(X_0(NQq),\cO)[\ker(\lambda_{f^q})]$ (see   \cite[\S 4]{DiRi} for related arguments in a similar situation).  Here $\lambda_{f^q}$ is the homomorphism of the full Hecke algebra acting on $S_2(\Gamma_0(NQq),\cO)$ associated to an oldform $f^q$ in it with associated newform $f$ such that $f^q$  that has the same eigenvalues as $f$ for  Hecke operators $T_r$ (or $U_r$ for $r|NQ$)   for positive integers $r$  that are prime to $q$ and $f^q|U_q=0$.  To see the claim we first note that as $f$ is  a newform  of level $NQ$, the map $\beta:H^1(X_0(NQ),\cO)[\ker(\lambda_f)] \to  H^1(X_0(NQq),\cO)$ is injective (using   an easy implication of \cite[Lemma 2.5]{Wiles} that can be seen by considering the map induced by $\beta$   on $ H^1(X_0(NQ),\cO)[\ker(\lambda_f)] \otimes E$), and hence its image  is a free $\cO$-module of rank 2. Then the fact that  $\beta:H^1(X_0(NQ),\cO)[\ker(\lambda_f)] \to  H^1(X(\Gamma_0(NQq),\cO)[\ker(\lambda_{f^q})]$ is  actually an isomorphism follows from     \cite[Lemma 2.5]{Wiles} and the fact that  $\rho_f$ is residually irreducible.

      Let $\beta'$ be the  map $H^1(X_0(NQq),\cO) \to H^1(X_0(NQ),\cO)$ dual to $\beta$. Then  we see that the induced map (that we denote by the same symbol)  $\beta'\beta:H^1(X_0(NQ),\cO)[\ker(\lambda_f)] \to H^1(X_0(NQ),\cO)[\ker(\lambda_f)]  $ is  (up to a unit) given by multiplication by $q^2-1$ (see  also \cite[\S 4]{DiRi}).
     
     Thus  if $ A_r,B_r$ as above is a basis of  the free $\cO$-module $H^1(X_0(NQ),\cO)[\ker(\lambda_f)]$ of rank 2, then $(\langle \beta(A_r),\beta(B_r)\rangle)=(\langle A_r, \beta'\beta(B_r) \rangle)= ((q^2-1)\langle A_r,B_r \rangle)$, where the first    pairing  is the twisted Poincar\'e  pairing on the cohomology of $X_0(NQq)$.  Iterating this argument $2r$ times, using the degeneracy maps $ H^1(X_0(NQq_1 \ldots q_t),\cO)^2 \to H^1(X_0(NQq_1 \ldots q_tq_{t+1}), \cO)$,  yields the equality of ideals $(\langle A,B\rangle)= (\prod_{q \in Q} (q^2-1)) ( \langle A_r,B_r\rangle_r)$ of $\cO$ finishing  Step 2.

     \vskip .5cm

     Now using  $S(NQ^2)_{\ffrm_Q}^*=M(NQ^2)^{\oplus 2}$ and   invoking   Lemma \ref{lem:cng det self-dual}, Lemma \ref{lem:cng  direct  sum} and Lemma \ref{lem:cng dual},  finishes  the proof of  Proposition \ref{change-of-eta}.

   \end{proof}

  \begin{corollary}\label{eta:cor1}
  Assume that $Q$ contains a prime that divides $N(\rhobar)$ and that $Q$ is of even cardinality.
    We  have the equality $\eta_\lambda(\T)=(\prod_{q \in Q} (q^2-1)\varpi^{m_q})\eta_\lambda(\T^{\st})$.
  \end{corollary}

  \begin{proof} The result follows from Proposition \ref{change-of-eta}, on using Corollary \ref{cor:sat} (which gives that $\card{\cng{\lambda}(\T^{\st})}= \card{\cng{\lambda}(M^{\st}(N))}$) and  Theorem \ref{fred} (which gives that  $M(NQ^2)$ is free of rank one over $\T$ and hence $\card{\cng{\lambda}(\T)}= \card{\cng{\lambda}(M(NQ^2))}$).   
    \end{proof}
  
  \begin{corollary}\label{eta:cor2} $ $
  
  \begin{enumerate}
  
  \item  Assume we are in the definite case (i.e.,  that $|Q|$ is odd), and in the case that $N=1$ assume $N(\rhobar)$ is divisible by at least two primes.  We  have the equality $\eta_\lambda(\T)=(\prod_{q \in Q} (q^2-1)\varpi^{m_q})\eta_\lambda(\T^{\st})$.
  
  \item Assume  we are in the indefinite case  (i.e., that $|Q|$ is even),  and in the case that $(Q,N(\rhobar))=1$ that $N(\rhobar)$ is divisible by at least two primes.
   We  have the equality $\eta_\lambda(\T)=(\prod_{q \in Q} (q^2-1)\varpi^{m_q})\eta_\lambda(\T^{\st})$.

  \end{enumerate}

  \end{corollary}
  
  \begin{proof}
 
  Just for this proof, in which we vary $Q$ and $N$, we denote by $ \T_{Q,N}^{\st}$ the  Hecke algebra  previously denoted $\T^{\St}$ that  acts faithfully  on $H^1(X_0^Q(N),\cO)_{\ffrm_Q'}$.  We reduce to the case proved in Corollary \ref{eta:cor1}  by using  the  Jacquet-Langlands correspondence as follows. 
 
 \begin{enumerate}
 \item   Assume that $|Q|$ is odd 
 \begin{itemize}
 \item  If $N=1$,   and $t|(Q,N(\rhobar))$, by the Jacquet-Langlands isomorphism we have that $\T_{Q,1}^{\st}=\T_{Q \backslash \{t\} ,t}^{\st}$, and $(N(\rhobar),Q/t)>1$ by assumption.
 
 \item If $N\neq 1$, and $t|N$ for a prime $t$,  by the Jacquet-Langlands isomorphism we have that $\T_{Q,N}^\st=\T_{Q\cup \{t\},N/t}^{\st}$.
 \end{itemize}

\item  Assume that $Q$ is of even   cardinality and that $Q$ is coprime to $N(\rhobar)$ (as the case when $Q$ and $N(\rhobar)$ are not  coprime and $Q$ is of even order  is covered by Proposition \ref{change-of-eta}).   Thus $N=N(\rhobar)$,  by assumption $N$ is divisible by 2 primes $t_1,t_2$,  the Jacquet-Langlands isomorphism gives  that $\T_{Q,N}^\st=\T_{Q\cup \{t_1,t_2\}, N/t_1t_2}^{\st}$.
   \end{enumerate}
   
  \end{proof}

  \begin{remark} $ $
  \begin{itemize}
  
  \item  We see from Corollary \ref{eta:cor1} that  the change of the congruence module when dropping  trivial primes from  the discriminant is qualitatively the same  as when dropping primes that are non-trivial, in contrast to the change of cotangent spaces in Corollary \ref{cor:main}.
  
\item   One expects the general case of  Proposition \ref{change-of-eta},  without assuming  that $Q$ contains a prime  $t$ that divides $N(\rhobar)$ and   $Q$ is of even  cardinality,  to be true,   but we do not know how to  modify its  proof  to yield this.
  One may remark (at  least in the  indefinite case, i.e.,  that $Q$ is of even cardinality), that Corollary \ref{eta:cor2} can be established without the conditions assumed in it if we knew that the top exact sequence:
  
  \begin{equation}
\begin{CD}
0 @> >> \Hom({\cal X}(J,q')_{{\ffrm},\cO},\cO(1)) @> >>  {\rm Ta}_p(J)_{\ffrm,\cO}
 @> >>   {\cal X}(J,q')_{{\ffrm},\cO}   @>  >> 0
\\ 
&& @V VV @V VV @V VV 
\\
0 @> >>{\rm Hom}({\cal X}(A,q')_{{\ffrm},\cO},\cO(1))  @> >> {\rm Ta}_p(A)_{\ffrm,\cO}
 @> >>  {\cal X}(A,q')_{{\ffrm},\cO}\ @>  >> 0
\end{CD}
\end{equation} 
splits. This would imply that the map of component groups at $\phi(J,q')_\ffrm \ra \phi(A,q')_\ffrm$ is surjective. Further if $q'$ is not 1 mod $p$,  it is easy to see that the top exact sequence indeed splits. On the other hand if $q'$ is a trivial prime, then not knowing this,  we  instead rely on the trick
in \cite{RiTa} which uses the assumption that there is a prime $t$ at which $A$ has trivial component group, and  which furthermore divides the discriminant of the quaternion algebra which gives rise  to $J$.

\item  In forthcoming work we hope to be able to compute the change in congruence modules as in Proposition \ref{change-of-eta} without its  superfluous assumptions, by  computing  the Wiles defect without using the work of \cite{RiTa}.\end{itemize}
  \end{remark}

\section{Computing the Wiles defect}\label{sec:defect}

We can now prove our main theorem:
\begin{theorem}\label{mc}
Let $N$ and $Q$ be relatively prime squarefree integers. Let $p>2$ be a prime not dividing $NQ$ and let $E/\Q_p$ be a finite extension with ring of integers $\cO$, uniformizer $\varpi$ and residue field $k$. Let $\rho_f:G_{\Q}\to \GL_2(\cO)$ be a Galois representation arising from a newform $f\in S_2(\Gamma_0(NQ))$, and let $\rhobar_f:G_{\Q}\to \GL_2(k)$ be the residual representation. Assume that $\rhobar_f$ is irreducible and $N|N(\rhobar_f)$.

Let $R^{\st}$ be the Galois deformation ring of $\rhobar_f$ defined in section \ref{sec_deformation_theory}, parameterizing lifts of $\rhobar_f$ of fixed determinant which are Steinberg at each prime dividing $Q$, finite flat at $p$ and minimal at all other primes.

Let $D$ be the quaternion algebra with discriminant $Q$ and let $\Gamma_0^Q(N)$ be the congruence subgroup for $D$ defined in section \ref{hecke}. Let $\T^Q(N)$ and $S^Q(\Gamma_0^Q(N))$ be the Hecke algebra and cohomological Hecke module at level $\Gamma^Q_0(N)$ and let $\ffrm\subseteq \T^Q(N)$ be the maximal ideal corresponding to $\ffrm$. Let $\T^{\st} = \T^Q(N)_\ffrm$ and let $\lambda:\T^{\st}\to \cO$ be the augmentation corresponding to $f$.

Assume that at least one of the following holds:
\begin{enumerate}
	\item $Q$ is a product of an even number of primes (i.e. $D$ is indefinite),and  $(N(\rhobar),Q)>1$;
	\item    $Q$ is a product of an  odd number of primes (i.e. $D$ is definite), and $N>1$;
	\item $N(\rhobar)$ is divisible by at least two primes.
\end{enumerate}
Then the Wiles defects (defined in Definition \ref{WilesDefect}) of $\T^{\st}$ and $S^Q(\Gamma_0^Q(N))$ with respect to the map $R^{\st}\onto \T^{\st}$ and the augmentation $\lambda$ are
\[
\delta_\lambda(\T^{\St}) = \delta_\lambda(S^Q(\Gamma_0^Q(N))_{\ffrm_Q'}) =  \delta_\lambda(M^\st(N))=\sum_{q|Q}\frac{2n_q}{e}
\]
where $e$ is the ramification index of $\cO$ and for each $q|Q$, $n_q$ is the largest integer for which $\rho_f|_{G_{\Q_q}} \pmod{\varpi^{n_q}}$ is unramified and $\rho_f(\Frob_q)\equiv \pm {\rm Id} \pmod{\varpi^{n_q}}$.
\end{theorem}
  
\begin{proof}
  
The  first  two equalities $\delta_\lambda(\T^{\St}) = \delta_\lambda(S^Q(\Gamma_0^Q(N))_{\ffrm_Q'}) =  \delta_\lambda(M^\st(N))$  follow from Corollary \ref{cor:sat}. For the remaining equality \[\delta_\lambda(\T^\st) = \sum_{q|Q}\frac{2n_q}{e},
\] we study the map $R^{\St} \ra \T^{\St}$, with augmentation $\lambda=\lambda_f$,  in relation  to the    isomorphism  $R \ra \T$  of complete intersections of   Theorem \ref{fred},  that arises when we relax the ramification conditions at $Q$, and the  induced augmentation  $\lambda:\T \to \cO$ as in \S \ref{hecke}.    From Theorem \ref{fred}   and Theorem \ref{numerical criterion} we deduce that   $\card{ \Phi_{\lambda}(R)}=\card{ \Phi_{\lambda}(\T)}=\card{ \cO/\eta_\lambda(\T)}$. 
   Combining this with:
  
  \begin{itemize}
  
  \item 
         Corollary \ref{cor:main} which gives 
         
         $$\ell_\cO(\Phi_{\lambda,\T/\T^{\St}})=\sum_{q \in Q} (m_q+\ord_\cO(q^2-1)-2n_q) ;$$

      \item Corollary  \ref{eta:cor2}  which gives   $\eta_\lambda(\T)=(\prod_{q \in Q} (q^2-1)\varpi^{m_q})\eta_\lambda(\T^{\st})$;
       \end{itemize}
       proves our theorem.
   
  \end{proof}

  \begin{remark}
 The Wiles defect in the situation studied in this paper  (say when $\cO=\Z_p$) is always an even number.  It would be interesting to have a conceptual explanation for this. We hope to return to this question in a future work which will give a computation of the Wiles defect that has a stronger local-global flavor.
\end{remark}

\section{Wiles defect for  semistable elliptic curves and tame regulators}\label{semistable}

The Wiles defect  associated to an augmentation $\lambda_f:\T^\St \ra \cO$ is a global invariant, and as defined seems  hard to compute 
(it is an ``$f$-extrinsic'' invariant in Mazur's suggestive terminology, see \cite{Mazur-Hida}), while the $n_q$'s  which our theorem relates the Wiles defect of $\lambda_f$ to are local invariants of the local Galois representation $\rho_f|_{G_q}$ (and are ``$f$-intrinsic'' in the sense of Mazur).

In this section we illustrate how Theorem \ref{mc} makes the Wiles defect rather easily computable in many examples. 
Let $\cE$ be a semistable elliptic curve over $\Q$ of conductor $N$. Let $p$ be a prime such that the mod $p$ representation $\rhobar$  arising from $\cE$ is irreducible and assume $(N,p)=1$.

\subsection{Inertial invariants $n_q$ for Tate elliptic curves over $\Q_q$}

Let $\rho_{\cE,p}: G_\Q \ra \GL_2(\Z_p)$ be the representation on the $p$-adic Tate module of $\cE$. Then the local invariant $n_q$ at $q|N$ is the maximal integer such that $\rho_{\cE,p}$ mod $p^{n_q}$ is unramified and sends $\Frob_q$ to $\pm {\rm Id}$. The quantity $n_q$ may be computed using Tate's uniformization of $\cE_{\Q_q}$ and Kummer theory. 

Write $v_q\colon\Q_q^\times\to\Z$ and $|\cdot|_q\colon\Q_q^\times\to q^\Z$ for the normalized valuation and absolute value at $q$. Let $q_{\cE_{\Q_q}}\in\Q_q^\times$  be the Tate period of $\cE_{\Q_q}$, define $m_q$ and $t_q$ so that $p^{m_q}|| v_q(q_{\cE_{\Q_q}})$ and $p^{t_q}||q-1$, and set $k_q={\rm min}(m_q,t_q)$ and write $\widetilde{q_{\cE_{\Q_q}}}=q_{\cE_{\Q_q}}\cdot |q_{\cE_{\Q_q}}|_q $ for the unit part of the Tate period. Consider the map $\log_{q,p^{k_q}}:(\Z/q\Z)^* \onto\Z/p^{k_q}\Z$, and denote for an element $x \in (\Z/q\Z)^*$ by $\ind_{q,p^{k_q}}(x)$ the index of the subgroup generated by $\log_{q,p^{k_q}}(x)$ in  $\Z/p^{k_q}\Z$ (tame regulator \'a la \cite{MT}).

\begin{proposition}\label{prop:MT}
   We have the formula $n_q= \log_p \big(\ind_{q,p^{k_q}}(\widetilde{q_{\cE_{\Q_q}} })\big)$ .
\end{proposition}

\begin{proof}
 This   follows from Kummer theory and Tate's uniformization of $\cE$ at $q$.
\end{proof}

 By the modularity theorem, we know that $\cE$ arises from $f \in S_2(\Gamma_0(N))$.   Assume that  $N(\rhobar)$ is not prime.
 Let $Q$ be a set of primes which divide $N$,  such that $N$ factors as $N=N'Q$ with $N'|N(\rhobar)$, and  let $\T^\St=\T^Q(N')$ and consider the augmentation $\lambda_{\cE}: \T^\St \ra \Z_p$ arising from $f$. (Our future work will  address removing the condition that $N(\rhobar)$ is not prime, as it will avoid dependence on the results of \cite{RiTa}. )
\begin{corollary}\label{cor:MT}
  
   $$\delta_{\lambda_{\cE}}(\T^\St)=\Sigma_{q \in Q} 2  \log_p( \ind_{q,p^{k_q}}(\widetilde{q_{\cE_{\Q_q}} })).$$

\end{corollary}

\begin{proof}
 This follows  from Theorem \ref{mc}.
\end{proof}

\begin{remark}

Let $q$ be a prime of multiplicative reduction of $\cE$ with Tate period $q_{\cE_{\Q_q}}$, and let $j_\cE$ denote the $j$-invariant of $\cE$. It takes some effort to compute $q_{\cE_{\Q_q}}$ to some non-trivial precision. However, it is rather elementary to compute its unit part from $j_\cE$. For this recall from \cite[Thm.~V.3.1(b)]{Silverman2} that $v_q(q_{\cE_{\Q_q}})=-v_q(j_\cE)>0$ and that $j_\cE=\frac1{q_{\cE_{\Q_q}}}+744+O(q)$. Hence $j_\cE |j_\cE|= |j_\cE| \frac{1}{q_{\cE_{\Q_q}}}+O(q)$ is a unit, and we deduce
\[\widetilde{q_{\cE_{\Q_q}}}\equiv j_\cE^{-1} |j_\cE|_q^{-1} \in (\Z_q/q\Z_q)^\times = \F_q^\times.\]

\end{remark}

\subsection{Numerical computations}

It seemed natural to carry out computations for the example given in the introduction, where the initial level is $\Gamma_0(11)$, the prime is $p=3$ and a trivial prime mod $p$ occurs at $q=193$. However, as indicated in Remark~\ref{Remark-OnComputations}, this leads to a form $f$ in level $11\cdot 193$ with $[K_f:\Q]=46$.\footnote{
There are $8$ eigenspaces (over $\Q$) in level $2123$. Two of them have the same Atkin-Lehner sign at 11 as the unique eigenform in level $11$. The eigenspaces have dimension $33$ and $46$, and we call representing forms $f_{33}$ and $f_{46}$. The eigenspace in level $11$ is $\Q$-rational. Hence the mod $3$ reduction of the Hecke algebra of the wanted form in level 2123 needs to have a factor $\F_3$. By \cite[Thm.~2.5]{Edixhoven}, this implies for the mod $3$ Hecke eigenvalue $a_3$ that $a_3^2+1$ lies in $\F_3$. It turns out that no such root exists for the mod $3$ reduction of $T_3$ for $f_{33}$; for $f_{46}$ exactly one such root (with multiplicity $7$) lies in $\F_3$.} Even if in principle one can compute $n_{193}$ using Mumford--Raynaud uniformization of $A_f$ at $193$ and Kummer theory, as in Proposition \ref{prop:MT}, to us the condition $[K_f:\Q]=46$ made computations impossible.

The importance of Proposition~\ref{prop:MT} is that it allows practical computations for a different class of examples. Namely, one starts with an elliptic curve $\cE$ of a `large' squarefree level $\cN$ with associated cusp form $f$; think of $\cN=NQ$ with $N,Q$ as in Theorem~\ref{mc}. Then one searches for prime divisors $q$ of $\cN$ at which $\rho_{f,p}$ admits level lowering modulo a prime $p$. Proposition~\ref{prop:MT} allows one to then compute the local invariant attached to $(\cE,q,p)$. We eliminate cases where $p|\cN$, so that $\rho_{f,p}$ is finite at $p$, and we check that $\bar\rho_{f,p}$ is irreducible, so that level lowering results apply.  We also ensure that the minimal conductor $N(\bar\rho_{f,p})$ of the reduction $\bar\rho_{f,p}$ has at least two prime divisors, so that Theorem~\ref{mc} is applicable unconditionally, for any $Q>1$. Using the data in \cite{lmfdb}, it is also often possible to determine the form $g$ in level  $N(\bar\rho_{f,p})$. This is not needed though to compute the invariant $n_q$ or the Wiles defect for~$Q$.

Concretely, we did an incremental search through all elliptic curves of squarefree conductor $\cN<400000$ in the Cremona tables. For each prime divisor $q$ of $\cN$ we make a list of all odd primes $p$ that divide $q-1$ but not $\cN$. Next we computed for each elliptic curve $\cE$ of conductor $\cN$ the $q$-valuation $v_q(q_{\cE_{\Q_q}})$ and the unit part $\widetilde{q_{\cE_{\Q_q}}}$ of its Tate period at $q$. Finally we compute the quantities $t_{q}=\val_p(q-1)$, $m_{q}=v_p(v_q(q_{\cE_{\Q_q}}))$ and $u_{q}=v_p(\card{\F_q^*/\langle\widetilde{ q_{\cE_{\Q_q}}} \rangle})$. Then $n_{q}=\min \{t_{q},m_{q},u_{q}\}$ is the local invariant from Proposition~\ref{prop:MT} (it also depends on $p$ and $\cE$). For all primes $p$ in our list, found for $\cE$ and some $q$, we determine if the Galois representation $\bar\rho_{\cE,p}$ on the $p$-torsion points of $\cE$ is absolutely irreducible. Then level lowering is possible if and only if $m_q>0$. Using this, we can compute $N(\bar\rho_{f,p})$ and single out those with at least two prime divisors. The Wiles defect can then be computed as $2\sum_{q\in Q}n_q$ over the local invariants. All computations were carried out using SageMath.

For the primes  $p=3,5,7,11$, the smallest examples $(\cE,q,p)$, small in terms of size of conductor $\cN$, with $n_q\ge1$, are the following
\[
(805b1,7,3), \ \ (5673a1, 61,5), \ \ (4171a1,43,7), \ \ (27186m1,23,11),\]
where we use Cremona notation for the labeling of elliptic curves $\cE$; the number to the left of the letter is the conductor of $\cE$. For $\cN< 400000$, we found the following
\[
\begin{array}{c|c|c|c|c}
\hbox{prime }p& 3&5&7&11\\
\hline
\#\{(\cE,q)\mid n_q>0\} &8346&950&43&3
\end{array}
\]
For $p=5,7,11$ we found no example where $n_q\ge2$, and no example with $n_q=1$ at two prime divisors $q$ of $\cN$ (for a fixed $\cE$). For $13\le p\le61$ we found no $(\cE,q)$ with $n_q>0$. From this we suspect that 
for $\cN<400000$ and $p\ge5$
the Wiles defect of $\cE$ is either~$0$ or $2$. 

For $p=3$, we also found pairs $(\cE,q)$ with $n_q=2$ (none with $n_q>2$); there were $25$ such examples with $\cN<400000$; they all had a single prime divisor $q|\cN$ at which the local invariant was non-zero. Hence for $Q\supseteq\{q\}$ the Wiles defect is $4$ in these cases. The two examples $(\cE,q)$ of smallest conductor and with $n_q=2$ are
\[(36613a1,19), \ (104710l1,37).\]
They have  $N(\bar\rho_{\cE,q})=1927=41\cdot47$ and $N(\bar\rho_{\cE,q})=2830=2\cdot5\cdot283$, respectively.

For $p=3$ only, we also found elliptic curves $\cE$ which at two primes $q\neq q'$ dividing $\cN$ had invariants $n_q=n_{q'}=1$. There were $15$ of these. For these and for $Q\supseteq\{q,q'\}$, the Wiles defect is again $4$. In all other cases where the Wiles defect was $2$ when it is non-zero. The two examples $(\cE,q,q')$ of smallest conductor and with $n_q=n_{q'}=1$ are
\[(149149b1,7,13), \ (149149c1,7,13)\]
with $N(\bar\rho_{\cE,q})= 1639=11\cdot149$.

We also determined in many of the above cases the form $g$ congruent to $f$ mod $p$ of minimal level using the database \cite{lmfdb}. One very helpful invariant when determining $g$ were the Atkin-Lehner signs. The following table displays our findings. In the second row we use the labeling of forms given in \cite{lmfdb}; note that as in the case of elliptic forms, the label of the cusp form $g$ contains its level as the left number.

\[
\begin{array}{c|c|c|c|c|c|c|c}
\hbox{triple }(\cE,q,p)& (805b1,7,3)& (5673a1, 61,5) & (4171.a1,43,7) &\ (27186.m1,23,11)&\ (104710l1,37,3)\\
\hline
\hbox{label of $g$} &115.a&93.2.a.a&1731.a.f. &1182.2.a.d &2830g1
\end{array}
\]
The last example gives a rare example of a congruence of elliptic curves, since also $g$ is $\Q$-rational.

\bibliographystyle{amsalpha}
\bibliography{refs}

\noindent {\bf Address of authors:} 

(GB) Interdisciplinary Center for Scientific Computing, Universit\"at Heidelberg, INF 205, 69120 Heidelberg, Germany. \textit{Email address}: \texttt{gebhard.boeckle@iwr.uni-heidelberg.de}

(CK) Department of Mathematics, UCLA, Los Angeles, CA 90095-1555,
  USA. \textit{Email address}: \texttt{shekhar@math.ucla.edu}

(JM) Department of Mathematics, UCLA, Los Angeles, CA 90095-1555,
  USA. \textit{Email address}: \texttt{jmanning@math.ucla.edu}

\end{document}